\documentclass{compositio}
\usepackage[T1]{fontenc}
\usepackage[utf8]{inputenc}
\usepackage{amsmath,amssymb,url,upref,xspace,verbatim}
\RequirePackage[all,cmtip]{xy}

\newdir^{ (}{{}*!/-5pt/\dir^{(}}
\newdir_{ (}{{}*!/-5pt/\dir_{(}}
\RequirePackage[mathscr]{eucal}
\let\mathcal\mathscr
\RequirePackage[%
pdfview=XYZ,
pdfstartview=FitV,
pdffitwindow,
hyperfootnotes=false,
raiselinks=true,
hypertexnames=false,
colorlinks=true,
bookmarks=true,
bookmarksopen=false,
anchorcolor=black,
citecolor=black,
linkcolor=black,
filecolor=black,
urlcolor=black]{hyperref}

\AtBeginDocument{

}

\numberwithin{equation}{section}

\theoremstyle{plain}
\newtheorem{theorem}[equation]{Theorem}
\newtheorem{proposition}[equation]{Proposition}
\newtheorem{lemma}[equation]{Lemma}
\newtheorem{corollary}[equation]{Corollary}
\newtheorem{theoremintro}{Theorem}
\newtheorem*{claim*}{Claim}

\theoremstyle{definition}
\newtheorem{definition}[equation]{Definition}
\newtheorem{notation}[equation]{Notation}
\newtheorem{setting}[equation]{Setting}
\newtheorem{example}[equation]{Example}
\newtheorem{remark}[equation]{Remark}

\AtBeginDocument{%
\newenvironment{enumeratea}
{\bgroup\def\theenumi{\alph{enumi}}\begin{enumerate}}
{\end{enumerate}\egroup}%
}

\let\oldtheequation\theequation
\newcommand\numstareq{\let\oldtheequation\theequation\renewcommand{\theequation}{\oldtheequation$\,*$}}
\newcommand\numstarstareq{\let\oldtheequation\theequation\renewcommand{\theequation}{\oldtheequation$\,**$}}
\newenvironment{starequation}{\numstareq\begin{equation*}\tag{\theequation}}{\end{equation*}\let\theequation\oldtheequation\ignorespaces}
\newenvironment{starstarequation}{\numstarstareq\begin{equation*}\tag{\theequation}}{\end{equation*}\let\theequation\oldtheequation\ignorespaces}

\newcommand{\step}{\itemindent9pt\refstepcounter{equation}\smallskip\item[\textup{(\theequation)}]\itshape}

\newcommand{\RedefinitSymbole}[1]{
\expandafter\let\csname old\string#1\endcsname=#1
\let#1=\relax
\newcommand{#1}{\csname old\string#1\endcsname\,}
}
\RedefinitSymbole{\forall} \RedefinitSymbole{\exists}

\newcommand{\cbbullet}{{\raisebox{1pt}{$\sbullet$}}}
\newcommand{\sbullet}{{\scriptscriptstyle\bullet}}
\newcommand{\pOS}{p^{-1}\sho_S}
\newcommand{\pXOS}{p_X^{-1}\sho_S}
\newcommand{\pYOS}{p_Y^{-1}\sho_S}
\newcommand{\pZOS}{p_Z^{-1}\sho_S}
\newcommand{\pOSp}{p^{\prime-1}\sho_{S'}}
\newcommand{\wpOS}{\wt{p}^{-1}\sho_S}
\newcommand{\wpOSp}{\wt{p}^{\prime-1}\sho_{S'}}
\newcommand\wtj{\wt\jmath}
\newcommand{\Df}{{}_{\scriptscriptstyle\mathrm{D}}f}
\newcommand{\Dg}{{}_{\scriptscriptstyle\mathrm{D}}g}
\newcommand{\Dh}{{}_{\scriptscriptstyle\mathrm{D}}h}
\newcommand{\Di}{{}_{\scriptscriptstyle\mathrm{D}}i}
\newcommand{\Ddelta}{{}_{\scriptscriptstyle\mathrm{D}}\delta}
\newcommand{\Dp}{{}_{\scriptscriptstyle\mathrm{D}}p}
\newcommand{\tint}{{\textstyle\int\!}}
\newcommand{\Dint}{{}_{\scriptscriptstyle\mathrm{D}}\tint}
\newcommand{\quand}{\quad\text{and}\quad}

\newcommand\sha{\mathcal{A}}
\newcommand\shc{\mathcal{C}}
\newcommand\shd{\mathcal{D}}\let\cD\shd
\newcommand\shf{\mathcal{F}}
\newcommand\shi{\mathcal{I}}\let\cI\shi
\newcommand\shj{\mathcal{J}}
\newcommand\shk{\mathcal{K}}
\newcommand\shh{\mathcal{H}}
\newcommand\shl{\mathcal{L}}
\newcommand\shm{\mathcal{M}}
\newcommand\shn{\mathcal{N}}
\newcommand\sho{\mathcal{O}}
\newcommand\shp{\mathcal{P}}
\newcommand\shq{\mathcal{Q}}
\newcommand\shs{\mathcal{S}}
\newcommand\shscr{\mathscr{S}}
\newcommand{\C}{\mathbb{C}}\let\CC\C
\newcommand{\N}{\mathbb{N}}\let\NN\N
\newcommand{\R}{\mathbb{R}}\let\RR\R
\newcommand{\Z}{\mathbb{Z}}\let\ZZ\Z
\newcommand{\PP}{\mathbb{P}}

\newcommand{\wrc}{{\textup{w-}\R\textup{-c}}}

\newcommand{\sfA}{\mathsf{A}}

\newcommand{\fm}{\mathfrak{m}}

\newcommand{\bms}{\boldsymbol{s}}
\newcommand{\bD}{\boldsymbol{D}}

\newcommand{\rhog}{\boldsymbol{\rho}}
\renewcommand{\thetag}{\boldsymbol{\theta}}

\newcommand{\rb}{\mathrm{b}}
\newcommand{\rmc}{\mathrm{c}}
\newcommand{\rd}{\mathrm{d}}
\newcommand{\gr}{\mathrm{gr}}
\newcommand{\coh}{\mathrm{coh}}
\newcommand{\Oqcoh}{\sho\textup{-}\mathrm{q}\textup{-}\mathrm{coh}}
\newcommand{\good}{\mathrm{good}}
\newcommand{\hol}{\mathrm{hol}}
\newcommand{\rhol}{\mathrm{rhol}}
\newcommand{\rmod}{\mathrm{mod}}
\newcommand{\Mod}{\mathrm{Mod}}
\newcommand{\nb}{\mathrm{nb}}
\newcommand{\op}{\mathrm{op}}
\newcommand{\rt}{\mathrm{t}}
\newcommand{\sa}{\mathrm{sa}}
\newcommand{\tors}{\mathrm{tors}}
\newcommand{\tf}{\mathrm{tf}}

\newcommand{\cc}{{\C\textup{-c}}}
\newcommand{\rc}{{\R\textup{-c}}}

\newcommand{\XH}{X\times\nobreak H}
\newcommand{\XS}{X\times\nobreak S}
\newcommand{\wXS}{\wt X\times\nobreak S}
\newcommand{\XpS}{X'\times\nobreak S}
\newcommand{\XsS}{X^*\times\nobreak S}
\newcommand{\XT}{X\times\nobreak T}
\newcommand{\YS}{Y\times\nobreak S}
\newcommand{\ZS}{Z\times\nobreak S}
\newcommand{\DX}{\shd_{X}}
\newcommand{\DXH}{\shd_{X\times H/H}}
\newcommand{\DXS}{\shd_{\XS/S}}

\newcommand{\DXsS}{\shd_{\XsS/S}}
\newcommand{\DXT}{\shd_{\XT/T}}
\newcommand{\DXpS}{\shd_{\XpS/S}}
\newcommand{\DXSp}{\shd_{\XS'/S'}}
\newcommand{\DYS}{\shd_{\YS/S}}
\newcommand{\DUS}{\shd_{U\times S/S}}
\newcommand{\DXSR}{\shd_{\XS_\RR/S}}
\newcommand{\DYSR}{\shd_{\YS_\RR/S}}
\newcommand{\DZSR}{\shd_{\ZS_\RR/S}}

\DeclareMathOperator{\Aut}{Aut}
\DeclareMathOperator{\Char}{Char}

\DeclareMathOperator{\coker}{coker}
\DeclareMathOperator{\rC}{\mathsf{C}}
\DeclareMathOperator{\rD}{\mathsf{D}}
\DeclareMathOperator{\End}{End}
\DeclareMathOperator{\rK}{\mathsf{K}}
\DeclareMathOperator{\DR}{DR}
\DeclareMathOperator{\Db}{\mathfrak{Db}}
\DeclareMathOperator{\pDR}{{}^\mathrm{p}DR}
\DeclareMathOperator{\Hom}{Hom}
\newcommand{\Rhom}{R\shhom}
\newcommand{\shhom}{\mathcal{H}\!\mathit{om}}
\DeclareMathOperator{\rh}{\mathit{R}\shhom}
\DeclareMathOperator{\tho}{\mathit{T}\shhom}
\DeclareMathOperator{\id}{Id}\let\Id\id
\DeclareMathOperator{\Op}{Op}
\DeclareMathOperator{\RH}{RH}
\DeclareMathOperator{\rk}{rk}
\DeclareMathOperator{\Sp}{Sp}
\DeclareMathOperator{\pSol}{{}^\mathrm{p}Sol}
\DeclareMathOperator{\Supp}{Supp}\let\supp\Supp
\let\TH\THH
\DeclareMathOperator{\Ima}{im}

\newcommand{\fV}{\mathfrak{V}}
\newcommand{\fW}{\mathfrak{W}}

\newcommand{\bA}{{\mathbf{A}}}
\newcommand{\Ab}{{\mathbf{Ab}}}
\newcommand{\bP}{{\mathbf{P}}}
\newcommand{\bT}{{\mathbf{T}}}

\let\tilde\widetilde
\let\wt\widetilde
\let\wh\widehat
\let\ov\overline
\let\epsilon\varepsilon
\let\emptyset\varnothing
\let\setminus\smallsetminus
\let\leq\leqslant
\let\geq\geqslant
\let\moins\smallsetminus

\newcommand{\lcr}{[\![}
\newcommand{\rcr}{]\!]}
\newcommand{\lpr}{(\!(}
\newcommand{\rpr}{)\!)}

\newcommand\loccit{loc.\kern3pt cit.{}\xspace}
\newcommand\cf{cf.\kern.3em}
\newcommand\eg{e.g.\kern.3em}
\newcommand\ie{i.e.,\ }
\newcommand\resp{\text{resp.}\kern.3em}

\renewcommand\to{\mathchoice{\longrightarrow}{\rightarrow}{\rightarrow}{\rightarrow}}
\newcommand\mto{\mathchoice{\longmapsto}{\mapsto}{\mapsto}{\mapsto}}
\newcommand\hto{\mathrel{\lhook\joinrel\to}}
\newcommand\hfrom{\mathrel{\from\joinrel\rhook}}
\newcommand\from{\mathchoice{\longleftarrow}{\leftarrow}{\leftarrow}{\leftarrow}}
\newcommand\implique{\mathchoice{\Longrightarrow}{\Rightarrow}{\Rightarrow}{\Rightarrow}}
\newcommand\isom{\stackrel{\sim}{\longrightarrow}}
\newcommand\isofrom{\stackrel{\sim}{\longleftarrow}}

\newcommand\To[1]{\mathchoice{\xrightarrow{\textstyle\kern4pt#1\kern3pt}}{\stackrel{#1}{\longrightarrow}}{}{}}

\let\oldbigoplus\bigoplus
\renewcommand{\bigoplus}{\mathop{\textstyle\oldbigoplus}\displaylimits}
\let\oldbigcap\bigcap
\renewcommand{\bigcap}{\mathop{\textstyle\oldbigcap}\displaylimits}
\let\oldbigcup\bigcup
\renewcommand{\bigcup}{\mathop{\textstyle\oldbigcup}\displaylimits}
\let\oldprod\prod
\renewcommand{\prod}{\mathop{\textstyle\oldprod}\displaylimits}

\let\dpl\displaystyle

\begin{document}
\title{Relative Regular Riemann-Hilbert correspondence II}

\author[L. Fiorot]{Luisa Fiorot}
\email{luisa.fiorot@unipd.it}
\address[L. Fiorot]{Dipartimento di Matematica ``Tullio Levi-Civita'' Universit\`a degli Studi di Padova\\
Via Trieste, 63\\
35121 Padova Italy}

\author[T. Monteiro Fernandes]{Teresa Monteiro Fernandes}
\email{mtfernandes@fc.ul.pt}
\address[T. Monteiro Fernandes]{Centro de Matemática e Aplicações Fundamentais -- Centro de investigação Operacional e Departamento de Matemática da FCUL, Edifício C 6, Piso 2, Campo Grande, 1700, Lisboa, Portugal}

\author[C.~Sabbah]{Claude Sabbah}
\email{Claude.Sabbah@polytechnique.edu}
\address[C.~Sabbah]{CMLS, CNRS, École polytechnique, Institut Polytechnique de Paris, 91128 Palaiseau cedex, France}

\thanks{The research of L.\,Fiorot was supported by project PRIN 2017 Categories, Algebras: Ring-Theo\-ret\-ical and Homological Approaches (CARTHA). The research of T.\,Monteiro Fernandes was supported by
Fundação para a Ciência e Tecnologia, under the project: UIDB/04561/2020.}

\classification{14F10, 32C38, 32S40, 32S60, 35Nxx, 58J10}

\keywords{Holonomic relative $D$-module, regularity, relative constructible sheaf, relative perverse sheaf}

\begin{abstract}
We develop the theory of relative regular holonomic $\mathcal D$-modules with a smooth complex manifold~$S$ of arbitrary dimension as parameter space, together with their main functorial properties. In particular, we establish in this general setting the relative Riemann-Hilbert correspondence proved in a previous work in the one-dimensional case.
\end{abstract}
\maketitle

\vspace*{\baselineskip}
\tableofcontents

\addtocontents{toc}{\protect\bgroup\protect\small}
\section{Introduction}
In this article, we are concerned with holomorphic families of coherent $\shd$-modules on a complex manifold $X$ of dimension $d_X$, parametrized by a complex manifold $S$, that~is, coherent modules over the ring $\DXS$ of linear relative differential operators with respect to the projection $p_X:\XS\to S$ (simply denoted by $p$ when no confusion is possible). More specifically, we consider families for which the characteristic variety in the relative cotangent space $(T^*X)\times S$ is contained in the product by $S$ of a fixed closed conic Lagrangian analytic subset $\Lambda\subset T^*X$. Following \cite{FMFS1,MFCS1,MFCS2}, we~call these systems \emph{relative holonomic $\DXS$-modules}. Here are some examples.
\begin{enumerate}
\item\label{enum:1}
Deligne considered in \cite{De} the case of vector bundles $E$ on $\XS$ with a flat relative connection $\nabla$, and established an equivalence with the category of locally constant sheaves of coherent $\pXOS$-modules on $\XS$. In this case, the relative characteristic variety is the product of the zero section $T^*_XX$ by $S$.
\item\label{enum:2}
For any holonomic $\DX$-module $\shm$ on $X$ with characteristic variety $\Lambda$, the pullback $q^*\shm=\sho_{\XS}\otimes_{q^{-1}\sho_X}q^{-1}\shm$ by the projection $q:\XS\to X$ is naturally endowed with a $\DXS$-module structure, and the relative characteristic variety of~$\shm$ is equal to $\Lambda\times S$. For $(E,\nabla)$ as in \eqref{enum:1}, the characteristic variety of $q^*\shm\otimes_{\sho_{\XS}}E$ (equipped with its natural $\DXS$-module structure) is contained in $\Lambda\times S$.
\item\label{enum:3}
Some integral transformations from objects on $X$ to objects on $S$ have kernels which are such flat bundles $(E,\nabla)$. One of them is the Fourier-Mukai transformation $\mathrm{FM}$ introduced by Laumon \cite{Laumon96} and Rothstein \cite{Rothstein96}, which attaches to any bounded complex of $\shd$-modules with coherent cohomology on an abelian variety $A$ a bounded complex of $\sho$-modules with coherent cohomology on the moduli space $A^\sharp$ of line bundles with integrable connection on $A$ (\cf \cite{Schnell15} and the references therein). It is obtained as the integral transform with kernel $\shp$ on $A\times A^\sharp$ associated to the Poincaré bundle on the product $A\times\mathrm{Pic}^0(A)$. By construction, $\shp$ is equipped with a flat relative connection, \ie is a $\shd_{A\times A^\sharp/A^\sharp}$-module. Then $q^*\shm\otimes_{\sho_{A\times A^\sharp}}\shp$ is an instance of \eqref{enum:2}, and $\mathrm{FM}(\shm)$ is the pushforward $\Dp_*(q^*\shm\otimes_{\sho_{A\times A^\sharp}}\shp)$. It~is an object of $\rD^\rb_\coh(\sho_{A^\sharp})$.
\item\label{enum:4}
Given a holonomic $\DX$-module $\shm$ and holomorphic functions $f_1,\dots,f_p$ on~$X$ defining a divisor \hbox{$Y\!=\!\{\prod_if_i\!=\!0\}$}, and setting $S=\C^p$ with coordinates $s_1,\cdots, s_p$, the $S$-analytic counterpart of \cite[Prop.\,13]{Maisonobe16} asserts that the $\DXS$-submodule generated by $q^*\shm\cdot(\prod_if_i^{s_i})$ in the twisted coherent $\DXS(*(\YS))$-module $q^*\shm(*(\YS))\cdot(\prod_if_i^{s_i})$ is relative holonomic.
\item
In his construction of moduli spaces for regular holonomic $\shd$-modules, Nitsure \cite{Nitsure99} fixes a divisor with normal crossings in $X$ and deforms pre-$\shd$-modules (extending the notion of vector bundle with flat logarithmic connection) relative to this divisor and its canonical stratification. The corresponding holomorphic family of regular holonomic $\shd$-modules has its characteristic variety adapted to this stratification, hence of the form $\Lambda\times S$.
\item
Mixed twistor $\shd$-modules (\cf \cite{Mochizuki11}) are compound objects defined on the product of $X$ by the complex line $\CC$, whose module components are holomorphic families of holonomic $\shd$-modules parametrized by $S=\CC^*$ degenerating at $0\in\CC$ to a coherent module on the cotangent space $T^*X$. On $S$, the characteristic variety of each holonomic $\shd$-module is, by definition, contained in a fixed Lagrangian variety $\Lambda$. Of~particular interest are the regular mixed twistor $\shd$-modules, which have furnished the first example of families we are dealing with (\cf\cite{MFCS1}).
\end{enumerate}

Our definition of relative holonomicity imposes the following:
the only possible changes in the characteristic variety of the restricted $\DX$-module to a fixed parameter, when the parameter varies, is a change of multiplicities on each irreducible component of $\Lambda$. This condition is reasonable, as shown by the previous examples.

Besides, not any relative holonomic $\DXS$-module $\shk$ can serve as the kernel of an integral transformation as in \eqref{enum:3}, because it cannot be ensured that, for a holonomic $\DX$-module $\shm$, the tensor product $q^*\shm\otimes_{\sho_{\XS}}\shk$ is $\DXS$-coherent, in contrast with the classical result when $S$ is reduced to a point. Indeed, in general, Bernstein-Sato theory only applies on some open subset of $S$ depending on $\shm$, as exemplified in \cite[Ex.\,2.4]{MFCS3}. Fortunately, adding the regularity property---as defined in \cite{MFCS2}---of the kernel $\shk$ and of $\shm$ overcomes this difficulty, as already shown in \cite{FMFS1} when $\dim S=1$.

In order to replace Bernstein-Sato theory, the main tool is the stability of regular holonomicity by $\shd$-module pullback \cite[Th.\,2]{FMFS1}, that we generalize to the case where $\dim S$ can be arbitrary.

\begin{theoremintro}[Stability by pullback]\label{th:inverseimage}
Let $\shm\in\rD^\rb_{\rhol}(\DXS)$ and let $f:Y\to X$ be a morphism of complex manifolds and let us still denote by $f$ the morphism $f\times\Id$. Then $\Df^*\shm\in\rD^\rb_{\rhol}(\DYS)$.
\end{theoremintro}
A possible application of our approach is that, for an integral transformation as in \eqref{enum:3} with regular holonomic kernel $\shk$, the analysis of local properties of $q^*\shm\otimes^L_{\sho_{\XS}}\shk$
leads to a better understanding of the $\shk$-transform \hbox{$\Dp_*(q^*\shm\otimes^L_{\sho_{\XS}}\shk)$} of $\shm$.

Here are natural examples of regular holonomic $\DXS$-modules:
\begin{enumeratea}
\item
Coherent $\DXS$-submodules of regular holonomic $\shd_{\XS}$-modules, provided that the irreducible components of the characteristic variety of the latter decompose as products with respect to the product decomposition $T^*X\times T^*S$ (\cf Section \ref{subsec:3b}); they can be called \emph{integrable} regular holonomic $\DXS$-modules.
\item
In the setting of \eqref{enum:4} above, let us assume that $\shm$ is regular holonomic. Then the coherent $\DXS(*(\YS))$-module $q^*\shm(*\YS)\cdot\prod_if_i^{s_i}$ is $\DXS$-regular holonomic, hence $\DXS$-coherent (\cf Example \ref{ex:maisonobe}) and this enables us to recover the $S$\nobreakdash-analytic version of the result of Maisonobe in \eqref{enum:4} when $\shm$ is regular holonomic.
\end{enumeratea}

\smallskip
We have developed in \cite{MFCS1} the notion of relative $\C$-constructible complex and relative perverse complex on $\XS$, generalizing the notion of relative local system of coherent $\pXOS$-modules considered by Deligne. For each stratum $X_\alpha$ of a suitable complex stratification $(X_\alpha)$ of $X$, each cohomology of such a complex is locally isomorphic to $p_{X_\alpha}^{-1}G_\alpha$ for some \emph{coherent} $\sho_S$-module $G_\alpha$. We prove
here a Riemann-Hilbert correspondence for any $\dim S$ (the
$\dim S=1$ case was proved in~\cite{FMFS1}).

\begin{theoremintro}[Relative Riemann-Hilbert correspondence] \label{RHH}
The functors
\begin{align*}
\pSol_X: \rD^{\rb}_{\rhol}(\DXS)&\to \rD^\rb_{\cc}(\pXOS) \\
\RH^S_X:\rD^\rb_{\cc}(\pXOS)&\to \rD^{\rb}_{\rhol}(\DXS)
\end{align*}
are quasi-inverse equivalences of categories.
\end{theoremintro}

The proof of the correspondence is now made possible owing to \cite{MFP2}, as we construct the relative Riemann-Hilbert functor $\RH^S_X$ by means of the site $X_{\sa}\times S$ introduced in~\cite{MFP2} instead of the subanalytic site $X_{\sa}\times S_{\sa}$ considered in \cite{MFP1} and used in \hbox{\cite{MFCS2,FMFS1}}. The main results are the exact analogues of Theorems 2 and 1 in \cite{FMFS1} for $S$ of arbitrary dimension, with the same notation for relative regular holonomic complexes and $S$\nobreakdash-$\CC$\nobreakdash-constructible complexes, and the same meaning for the perverse solution functor $\pSol_X$. We change the order of the theorems with respect to \loccit\ since their proofs are done the other way round in the present paper.

The strategy for the proof of these theorems is similar to that of \cite{FMFS1}: it is made precise in Sections \ref{subsec:pfinv} and \ref{sec:synopsis}. However, when $\dim S\geq2$, we have to distinguish between $S$\nobreakdash-torsion-freeness and $S$-flatness (also called strictness in \loccit). While this is not much trouble for some of the results, owing to the analysis of t-structures made in~\cite{FMF1}, we are led to using Rossi's flattening theorem of \cite{Rossi68} at some point.

Along the way, we complement the results of \cite{MFCS2,FMFS1} by giving further characterizations and properties of regular holonomic $\DXS$-modules or objects of the bounded derived category $\rD^{\rb}_{\rhol}(\DXS)$, which are needed for the proof.\enlargethispage{\baselineskip}%

\subsubsection*{Comparison with other works on regular holonomic $\shd$-modules with parameters}
\begin{enumerate}
\item
In \cite{D-G-S11}, the authors introduce a formal parameter $\hbar$ and the corresponding rings $\CC\lcr\hbar\rcr$, $\sho_X\lcr\hbar\rcr$, $\DX\lcr\hbar\rcr$ denoted by $\DX^\hbar$. They define the notion of regular holonomic $\DX^\hbar$\nobreakdash-module by asking that the restriction to the closed point $\hbar=0$ is a regular holonomic $\DX$-module which is different from our point view since we ask regular holonomicity at any value of the parameter. They prove a Riemann-Hilbert correspondence with the category of $\CC$-constructible complexes of $\CC_X\lcr\hbar\rcr$-modules by means of a functor they denote by $\TH_\hbar$, which would correspond to the notation $\RH_X^S$ in the present paper. The main difference with the notions introduced in the present paper (or more precisely with those introduced in \cite{FMFS1}, which only considered the case where $\dim S=1$) is that the sheaves considered in \cite{D-G-S11} are sheaves of $\CC\lcr\hbar\rcr$-modules on $X$, while we consider sheaves of $\pXOS$-modules
on the product space $\XS$. The authors also consider the restriction at the generic point of $\CC\lcr\hbar\rcr$, that~is, modules over the ring $\DX\lpr\hbar\rpr$ which has no counterpart in our setting.

\item
If instead of considering $\DX$-modules one considers flat meromorphic bundles on $X$ (\ie meromorphic connections) with fixed pole divisor, there are many works in the algebraic or formal setting, with base fields or rings that can be different from~$\CC$. Since the literature is vast, let us only mention \cite{A-B01b} and the more recent preprint~\cite{H-DS-T21}. In the latter preprint, the authors extend Deligne's equivalence result~\cite[Th.\,II.5.9]{De} to an equivalence with a formal parameter. The corresponding equivalence in the present article would be that for $\DXS$-modules of D-type, as defined in Section \ref{sec:4}.
\item
Wu \cite{Wu21} has established a Riemann-Hilbert correspondence similar to that of Theorem \ref{RHH} in the case of Alexander complexes that occur in Example \eqref{enum:4} above.
\end{enumerate}

\subsubsection*{Organization of the paper}
In Section~\ref{sec:2}, we review and complete various results on coherent $\DXS$-modules obtained in our previous works \cite{MFCS1,MFCS2, FMF1,FMFS1}. We emphasize the behavior of holonomicity with respect to pullback and proper pushforward both with respect to $X$ and to the parameter space $S$. Note that the parameter space~$S$ is always assumed to be a complex manifold, while one should be able to generalize various statements to any complex analytic space.
For example, the sheaf $\DXT$ is well-defined if $T$ is a possibly singular and non reduced complex analytic \emph{space} with sheaf of functions~$\sho_T$. In this work, we restrict the setting to those complex spaces~$T$ that are embeddable in a smooth complex manifold $S$, with ideal $\shi_T\subset\sho_S$, and we regard $\DXT$-modules as $\DXS$-modules annihilated by~$\shi_T$.

In Section \ref{sec:3}, we complement various results of \cite{FMFS1} on the regularity property. In~particular, we give details on \cite[Rem.\,1.11]{FMFS1}. Furthermore, the relation with the usual notion of regularity of holonomic $\shd_{\XS}$-modules is made precise in Section~\ref{subsec:3b}. Stability under base pullback and proper base pushforward is established in Section~\ref{subsec:3c} (as usual, under a goodness assumption for proper pushforward). The case of pushforward with respect to a proper morphism $f:X\to Y$ (with a goodness assumption) has already been treated in \cite{MFCS2}, and stability by pullback, which is the content of Theorem \ref{th:inverseimage} needs first a detailed analysis of holonomic $\DXS$-modules of D-type.

This analysis is performed in Section \ref{sec:4}. The reasoning made in \cite{MFCS2} when \hbox{$\dim S=1$}, relying on the property that the torsion-free quotient of a coherent $\sho_S$-module is locally free, has to be adapted by using base changes with respect to~$S$, so that the base functoriality properties considered in Sections \ref{sec:2} and \ref{sec:3} are most useful.
Theorem~\ref{th:inverseimage} is proved in Section \ref{subsec:pfinv} in a way similar to that done in \cite{FMFS1}. In~Section \ref{subsec:5b}, we give a characterization of relative regular holonomicity in terms of formal solutions.

Section \ref{sec:synopsis} gives details on the main steps of the proof of Theorem \ref{RHH}. The strategy is similar to the one in \cite{MFCS2}, although we need various new technical details contained in Section \ref{S:RHS}. On~the one hand, the construction of the Riemann-Hilbert functor $\RH_X^S$ is now performed by using the partial subanalytic site $X_{\sa}\times S$ introduced in \cite{MFP2}. On the other hand, the comparison between two definitions of Deligne's extension, one using the standard notion of moderate growth and the other one obtained via the complex of tempered holomorphic functions on the partial subanalytic site $X_{\sa}\times S$, is also done in Section~\ref{S:RHS}, relying once more on results of \cite{MFP2}. We~emphasize that the proof given here is simpler, when $\dim S=1$, than that given in \cite{MFCS2}.

\subsubsection*{Acknowledgements}
We thank Luca Prelli and Pierre Schapira for useful discussions and suggestions. We thank Hélène Esnault for pointing out to us the preprint \cite{H-DS-T21}, and its authors for interesting discussions. We are deeply thankful to the referee whose comments helped us to rethink and thus to improve our work.

\section{A review on relative coherent and holonomic \texorpdfstring{$\shd$}{D}-modules}\label{sec:2}

For complex analytic manifolds $X$ and $S$, we~denote~by
\[
p_X:\XS\to S\quad\text{and}\quad q_S:\XS\to X
\]
the projections, and we use the notation $p,q$ when there is no confusion possible. The sheaf of relative differential operators $\DXS$ is naturally defined as
\begin{equation}\label{eq:DXS}
\DXS=q^{-1}\DX\otimes_{q^{-1}\sho_X}\sho_{\XS}=\sho_{\XS}\otimes_{q^{-1}\sho_X}q^{-1}\DX.
\end{equation}

\subsection{Coherence, goodness and holonomicity}
We adapt the definitions of \cite[\S4.7]{Ka2}, which we refer to for properties and proofs.

\begin{definition}\label{def:qcoh}
We say that an $\sho_{\XS}$-module $\shf$ is
\begin{enumerate}
\item\label{def:qcoh1}
\emph{$\sho$-good on $\XS$} if, on any relatively compact open set $U$ of $\XS$, $\shf_{|U}$ is the direct limit of an increasing sequence of $\sho_U$-coherent submodules (equivalently, the direct limit of an inductive system of $\sho_U$-coherent modules).

\item\label{def:qcoh2}
\emph{$\sho$-quasi-coherent} if each point of $\XS$ has an open neighborhood $U$ on which~$\shf$ is $\sho$-good.
\end{enumerate}
We say that a $\DXS$-module $\shm$ is
\begin{enumerate}\setcounter{enumi}{2}
\item\label{def:qcoh3}
\emph{good on $\XS$} if it is $\sho$-good on $\XS$ and $\DXS$-coherent.
\end{enumerate}
\end{definition}

The category $\Mod_{\Oqcoh}(\sho_{\XS})$ of $\sho$-quasi-coherent $\sho_{\XS}$-modules is an abelian full subcategory of $\Mod(\sho_{\XS})$ closed under extensions. Let $Y$ be a hypersurface of~$X$. We denote by $\cbbullet(*Y)$, instead of $\cbbullet(*(\YS))$, the localization functor along the hypersurface $\YS$ of $\XS$. We will use the notation $X^*=X\moins Y$.

\pagebreak[2]
\begin{lemma}\label{lem:localization}\mbox{}
\begin{enumerate}
\item\label{lem:localization1}
Let $\shl$ be an $\sho$-quasi-coherent module. Then $\shl(*Y)$ is $\sho$-quasi-coherent.
\item\label{lem:localization2}
Let $\shl$ be an $\sho$-quasi-coherent module supported on $\YS$ and localized along $\YS$ (i.e., $\shl\simeq \shl(*Y)$). Then $\shl=0$.
\item\label{lem:localization3}
Let $\shm\To{\phi} \shn$ be a morphism of $\sho$-quasi-coherent modules which are localized along $\YS$. If the restriction of $\phi$ at $\XsS$ is an isomorphism (\resp zero) then~$\phi$ is an isomorphism (\resp zero).
\end{enumerate}
\end{lemma}

\begin{proof}\mbox{}
\begin{enumerate}
\item
By definition any point of $\XS$ has an open neighborhood $U$, which we can suppose
to be a relatively compact open set,
on which
$\shl_{|U}\!=\!\bigcup_i\shl_i$ is the direct limit of an increasing sequence of $\sho$-coherent submodules $\shl_i$. Since $\shl_i(*Y)$, being equal to $\bigcup_n\shi_{\YS}^{-n}\shl_i$, is $\sho$-quasi-coherent for every $i$, so is $\shl_{|U}(*Y)=\bigcup_i\shl_i(*Y)$.
\item
The question is local. Since $\shl$ is an $\sho$-quasi-coherent module supported on $\YS$, we can suppose (up to shrinking the neighborhood) that it is a direct limit $\shl\!=\!\bigcup_i\shl_i$ of an increasing sequence of $\sho$-coherent submodules $\shl_i$ which are supported on $\YS$. Therefore $\shl_i(*Y)=0$ for every $i$ and thus $\shl(*Y)=\bigcup_i\shl_i(*Y)=0$.
\item
Let us denote by $\shl$ and $\shl'$ respectively the kernel and the cokernel of $\phi$. They are $\sho$-quasi-coherent modules. If the restriction of $\phi$ to $\XsS$ is an isomorphism, they are supported on $\YS$. By applying the localization functor along $Y$, we get the following exact sequence of $\sho$-quasi-coherent modules (by the first point)
\[
0\to \shl(*Y)\to \shm(*Y)\to \shn(*Y)\to \shl'(*Y)\to 0.
\]
By hypothesis, $\shm(*Y)\simeq \shm$ and $\shn(*Y)\simeq \shn$, thus
$\shl(*Y)\simeq \shl$, $\shl'(*Y)\simeq \shl'$ and by the previous point $\shl$ and $\shl'$ are zero.
It $\phi_{|\XsS}$ is zero, $\shl\to \shm$ is an isomorphism on $\XsS$, hence it is an isomorphism and so $\phi$ is
zero.\qedhere
\end{enumerate}
\end{proof}

Let $\DXS(*Y)$ be the coherent sheaf of rings $\sho_{\XS}(*Y)\otimes_{\sho_{\XS}}\DXS$. Then any coherent $\DXS$- or $\DXS(*Y)$-module is $\sho$\nobreakdash-quasi-coherent. On the other hand, any $\sho$-quasi-coherent $\DXS$-submodule $\shn$ of a coherent $\DXS$-module~$\shm$ is $\DXS$-coherent. The next lemma follows by an easy adaptation of \cite[Prop.\,4.23]{Ka2}.
\begin{lemma}\label{lem:cplxqgood}\mbox{}
\begin{enumerate}
\item\label{lem:cplxqgood1}
The category of coherent, \resp $\sho$-good, $\DXS$-modules, is abelian and stable by extensions in $\Mod(\DXS)$, and $\rD^\rb_\coh(\DXS)$, \resp $\rD^\rb_\good(\DXS)$, is a triangulated full subcategory of $\rD^\rb(\DXS)$.
\item\label{lem:cplxqgood2}
Let $\shm^\cbbullet$ be a bounded complex of $\sho$-good $\DXS$-modules. Then each cohomology $\shh^\ell(\shm^\cbbullet)$ is $\sho$-good on $\XS$.
\item\label{lem:cplxqgood3}
On any relatively compact open subset $U\!\Subset\!\XS$, each good $\DXS$\nobreakdash-module~$\shn$ comes in an exact sequence
\[
0\to\shn'\to\DXS\otimes_{\sho_{\XS}}\shl\to\shn_{|U}\to0
\]
with $\shl$ being $\sho_U$-coherent and $\shn'$ being good on $U$. In particular, $\shn_{|U}$ has a coherent $F_\sbullet\DXS$-filtration.
\end{enumerate}
\end{lemma}

\begin{remark}\label{rem:localiz}\mbox{}
\begin{enumerate}
\item\label{rem:localiz1}
With a slight adaptation of the proof of Lemma \ref{lem:localization}\eqref{lem:localization1} we conclude that if $\shm$ is good (\resp $\DXS$-coherent) then $\shm(*Y)$ is $\sho$-good (\resp $\sho$-quasi-coherent).
\item\label{rem:localiz2}
If a local section $m$ of an $\sho$-quasi-coherent module satisfying $\shm=\shm(*Y)$ is zero when restricted to $\XsS$, then it is zero. Indeed if $m$ is defined on $U\times U_S$, $m$ generates an $\sho$-coherent submodule of $\shm(*Y)_{|U\times U_S}$ which is supported on $\YS$; given a local defining equation $f=0$ for $Y$, we have $f^km=0$ for some $k$, thus $m=0$.
\end{enumerate}
\end{remark}

\subsubsection*{Characteristic varieties}
To any object $\shm$ of $\Mod_\coh(\DXS)$ is associated, by means of local coherent $\sho_{\XS}$-filtrations, its characteristic variety $\Char(\shm)$, which is contained in $T^*X\times S$. A~coherent $\DXS$-module is \emph{holonomic} if its characteristic variety is contained in $\Lambda\times S$, with~$\Lambda$ closed complex analytic Lagrangian $\C^*$-homogeneous in $T^*X$. Correspondingly are defined the derived categories $\rD^\rb_\coh(\DXS)$ and $\rD^\rb_\hol(\DXS)$ and the characteristic variety $\Char(\shm)$ for such objects $\shm$.

The structure of the characteristic variety of a holonomic $\DXS$-module~$\shm$ is described in \cite[Lem.\,2.10]{FMF1}: for each irreducible component $\Lambda_i$ of $\Lambda$ ($i\in I$) there exists a locally finite family $(T_{ij})_{j\in J_i}$ of closed analytic subsets of $S$ such that
\begin{equation}\label{eq:chardec}
\Char(\shm)=\bigcup_{i\in I}\bigcup_{j\in J_i}\Lambda_i\times T_{ij}=\bigcup_{i\in I}\Lambda_i\times T_i
\qquad \hbox{with } T_i=\bigcup_jT_{ij}.
\end{equation}
The projection to $X$ of $\Lambda_i$ is an irreducible closed analytic subset of $X$ that we denote by $Z_i$. These subsets form a locally finite family of closed analytic subsets of $X$. We~have
\begin{equation}\label{eq:strictperv2}
\Supp(\shm)=\bigcup_{i\in I}Z_i\times T_i,
\end{equation}
and we set
\[
\Supp_X(\shm)= \bigcup_{i\in I}Z_i,\quad \Supp_S(\shm)= \bigcup_{i\in I}T_i,
\]
that we call respectively the $X$-support (which is a closed analytic subset of $X$) and the $S$-support of $\shm$ (which may be not closed analytic if $I$ is infinite). Anyway, we~set $\dim\Supp_S(\shm)=\max_{i,j}\dim T_{ij}\leq\dim S$.

Any $\DXS$-coherent submodule or quotient module of $\shm$ is holonomic and its characteristic variety is the union of some irreducible components of $\Char(\shm)$.

\begin{lemma}\label{Lem:1}
The category $\Mod_{\hol}(\DXS)$ of holonomic $\DXS$-modules is closed under taking extensions in the category $\Mod(\DXS)$, and under taking sub-quotients in the category $\Mod_{\coh}(\DXS)$.
\end{lemma}

We say that a local section $m$ of a $\DXS$-module $\shm$ is an \emph{$S$\nobreakdash-tor\-sion section} if it is annihilated by some holomorphic function on $S$. The \emph{$S$-torsion submodule} $\shm_\rt$ of $\shm$ is the submodule consisting of $S$-torsion local sections. Note that if $\shm$ is a holonomic $\DXS$-module then
the $\DXS$-submodule $\shm_\rt$ is holonomic since it is an $\sho$-quasi coherent submodule of $\shm$. We~say that $\shm$ is \emph{$S$\nobreakdash-torsion-free} if $\shm_\rt=0$. We denote by~$\shm_{\tf}:\shm/\shm_\rt$ the torsion-free quotient.

We recall that the duality functor $\bD$ for $\DXS$-modules was considered in \cite[Def.\,3.4]{MFCS1} and that $\rD^\rb_{\hol}(\DXS)$ is stable under duality which is an involution.

\subsection{Behaviour with respect to pullback, pushforward and external product}

\begin{notation}\label{nota:fpi}\mbox{}
\begin{enumerate}
\item
For a holomorphic map $f:X\to X'$, we also denote by $f$ the morphism of complex manifolds $f\times\id:\XS\to \XpS$ and by $\Df^*$ and $\Df_*$ the pullback and pushforward functors in the derived category of relative $\shd$-modules.
\item
For a morphism $\pi:S\to S'$ between analytic spaces, we denote by $\pi^*$ and $R\pi_*$ the natural extension to the category of relative $\shd$-modules of the similar functors defined on the categories of $\sho$-modules
\end{enumerate}
\end{notation}

\begin{definition}\label{def:fpigood}
We say that a $\DXS$-module $\shm$ is \emph{$f$-($\sho$-)good} (\resp \emph{$\pi$-($\sho$-)good}) if it is ($\sho$-)good in some neighborhood of each fiber of $f$ (\resp $\pi$).
\end{definition}

We recall results concerning the behaviour with respect to a morphism of complex manifolds.

\begin{proposition}\label{prop:qcoh}
Let $\shm$ be a coherent $\DXS$-module.
\begin{enumerate}
\item\label{prop:qcohf1}
If $f:X'\to X$ is a holomorphic map of complex manifolds then, for each~$\ell\in\ZZ$, $\shh^\ell\Df^*\shm$ is an $f$-$\sho$-good $\DXpS$-module.
\item\label{prop:qcohf2}
Let $f:X\to X'$ be a proper holomorphic map. If $\shm$ is $f$-good (\resp $f$-good and holonomic), then for each~$\ell$, $\shh^\ell\Df_*\shm$ is $\DXpS$-coherent (\resp holonomic).
\end{enumerate}
\end{proposition}

\begin{proof}
We refer to \cite[Th.\,4.2 \& Cor.\,4.3]{SS1} for the proof of \eqref{prop:qcohf2} (or one can adapt the proof of \cite[Th.\,4.25 \& 4.27]{Ka2}), while for the proof of \eqref{prop:qcohf1} we can easily adapt the argument of \cite[Prop.\,2.1]{L-S87} which gives the absolute case.
\end{proof}

\subsubsection*{Base pullback}
Let $\pi:S'\to S$ be a morphism of complex analytic manifolds. We also denote by $\pi$ the induced map $\id\times\pi:\XS'\to\XS$. We have
\begin{equation}\label{eq:DXSp}
\DXSp=\sho_{\XS'}\otimes_{\pi^{-1}\sho_{\XS}}\pi^{-1}\DXS=:\pi^*\DXS,
\end{equation}
hence $\DXSp$ is a right $\pi^{-1}\DXS$-module. One then defines in a natural way the functor $L\pi^*:\rD(\DXS)\to\rD(\DXSp)$ as $L\pi^*(\cbbullet)=\DXSp\otimes^L_{\pi^{-1}\DXS}(\cbbullet)$.

The following is straightforward.

\begin{lemma}\label{lem:qcoh}
Let $\shl$ be an $\sho$-quasi-coherent $\sho_{\XS}$-module. Then each cohomology sheaf $L^j\pi^*\shl$ is $\pi$-$\sho$-good.
\end{lemma}

\begin{lemma}\label{lem:basepullback}
The functor $L\pi^*$ induces a functor $\rD^\rb_\hol(\DXS)\to\rD^\rb_\hol(\DXSp)$. For a holonomic $\DXS$-module $\shm$ and for each $j$, the characteristic variety $\Char L^j\pi^*\shm$ is contained in the pullback $\pi^{-1}\Char\shm$ (here, $\pi$ denotes the map $T^*X\times S'\to T^*X\times S$).
\end{lemma}

\begin{proof}
By using that, locally, a coherent $\DXS$-module has a resolution of length $2\dim X$ by free $\DXS$-modules of finite rank, one gets the first point for $\rD^\rb_\coh(\DXS)$. It suffices thus to prove the second point, which is a local question. One can find such a resolution which is strictly compatible with coherent filtrations relative to $F_\sbullet\DXS$. In such a way, one finds that for each~$j$, $\Char L^j\pi^*\shm$ is contained in the support of $L\pi^*\gr^F\shm$, hence the assertion.
\end{proof}

\subsubsection*{Base pushforward}
We now consider pushforward by a proper morphism \hbox{$\pi:S'\to S$}. Owing to \eqref{eq:DXSp}, we have a natural morphism $\pi^{-1}\DXS\to \DXSp$ which, by~ad\-junc\-tion, entails a natural morphism $\DXS\to\pi_*\DXSp$.

We recall \cite[Prop.\,1.6]{MFCS2}:

\begin{proposition}\label{prop:basepushhol}
Assume that $\pi$ is proper and that $\shm$ is $\DXSp$-holonomic and $\pi$\nobreakdash-good. Then, for each $j$, $R^j\pi_*\shm$ is $\DXS$-holonomic with characteristic variety contained in $\pi(\Char\shm)$.
\end{proposition}

\subsubsection*{An adjunction formula}
Let $\pi:S'\to S$ be a morphism of complex manifolds. We~will make use of the following adjunction formula.

\begin{lemma}\label{lem:adjunctionpi}
Let $\shm,\shn$ be objects of $\rD^\rb_\coh(\DXS)$. Then there is a bi-functorial isomorphism in
$\rD^\rb(\DXS)$:
\begin{starequation}\label{lem:adjunctionpi*}
R\pi_*\Rhom_{\DXSp}(L\pi^*\shm,L\pi^*\shn)\simeq\Rhom_{\DXS}(\shm,R\pi_*L\pi^*\shn).
\end{starequation}
If $\shn$ is a coherent $\DXS$-module, there are functorial morphisms in $\rD^\rb(\DXS)$:
\begin{starstarequation}\label{lem:adjunctionpi**}
\shn\to R\pi_*L\pi^*\shn\to R\pi_*\pi^*\shn,
\end{starstarequation}
and if $\pi$ is proper, $R\pi_*L\pi^*\shn\simeq (R\pi_*\DXSp)\otimes^L_{\DXS}\shn$ belongs to $\rD^\rb_\coh(\DXS)$.
\end{lemma}

\begin{proof}
We have (\cf \cite[p.\,241]{Ka2})
\[
\Rhom_{\DXSp}(L\pi^*\shm,L\pi^*\shn)\simeq\Rhom_{\pi^{-1}\DXS}(\pi^{-1}\shm,L\pi^*\shn),
\]
hence \eqref{lem:adjunctionpi*} is obtained by adjunction (\cf\cite[(2.6.15)]{KS1}).
By setting $\shm=\shn$ in \eqref{lem:adjunctionpi*} we get $\Hom_{\DXSp}(L\pi^*\shn,L\pi^*\shn)\simeq\Hom_{\DXS}(\shn,R\pi_*L\pi^*\shn)$. The image of~$\id$ by this isomorphism is the first morphism in \eqref{lem:adjunctionpi**} while the natural morphism $L\pi^*\shn\to\pi^*\shn$ in $\rD^\rb(\DXSp)$ provides the desired morphism $R\pi_*L\pi^*\shn\to R\pi_*\pi^*\shn$. The last isomorphism is obtained by the projection formula (\cf \cite[Prop.\,2.6.6]{KS1}).
\end{proof}

\begin{remark}[(External) tensor product]\label{rem:extprod}
If $\shm,\shn\in\rD^-(\DXS)$, the tensor product
\[
\shm\otimes_{\sho_{\XS}}^L\shn\in\rD^-(\DXS)
\]
is isomorphic to the pullback \hbox{$\Ddelta^*(\shm\boxtimes^L_\shd\shn)$} of the $S$\nobreakdash-ex\-ternal tensor product (recall that $\shd_{X\times \XS/S}$ is flat over $\DXS\boxtimes_{\sho_S}\DXS$)
\[
\shm\boxtimes^L_\shd\shn:=\shd_{X\times \XS/S}\otimes _{(\DXS\boxtimes_{\sho_S}\DXS)}(\shm\boxtimes^L_{\sho_S}\shn)
\]
by the diagonal embedding $\delta:\XS\hto(X\times X)\times S$ over~$S$.
If $\shm,\shn\in\rD^\rb_\coh(\DXS)$, $\shm\boxtimes^L_\shd\shn$ belongs to $\rD^\rb_\coh(\shd_{X\times \XS/S})$ and we have
\[
C:=\Char(\shm\boxtimes^L_\shd\shn)\subset\Char(\shm)\times_S\Char(\shn).
\]
This is seen by considering local resolutions of $\shm$ \resp $\shn$ by free $\DXS$-modules, showing both inclusions $C\subset T^*X\times\Char(\shn)$ and $C\subset\Char(\shm)\times T^*X$. Therefore, if~$\shm,\shn$ are holonomic, so is $\shm\boxtimes^L_\shd\shn$.
\end{remark}

\section{Regular holonomic \texorpdfstring{$\DXS$}{DXS}-modules}\label{sec:3}

In the special case of Lemma \ref{lem:basepullback} where $\pi$ is the inclusion $i_{s_o}:\{s_o\}\hto S$ of a point $s_o$ in $S$, we~recall a consequence of \cite[(A.10)]{Ka2}:
\begin{proposition}[{\cf \cite[Prop.\,3.1]{MFCS1}}]\label{prop:3.1}
For any $\shm, \shn$ in $\rD^\rb(\DXS)$, for any $s_o\in S$, the natural morphism
\[
Li^*_{s_o} \Rhom_{\DXS}(\shm, \shn)\to \Rhom_{i^*_{s_o}\DXS}(Li^*_{s_o}\shm, Li^*_{s_o}\shn)
\]
is an isomorphism in $\rD(\C_X)$.
\end{proposition}

\subsection{Characterization of relative regular holonomicity}
The category of \emph{regular holonomic $\DXS$\nobreakdash-mod\-ules} was introduced in \cite{MFCS2} as well as the full subcategory $\rD^\rb_\rhol(\DXS)$ of $\rD^\rb_{\hol}(\DXS)$ of bounded complexes of $\DXS$\nobreakdash-modules having regular holonomic cohomology. According to~\cite{MFCS2}, we say that an object $\shm\in\rD^\rb_{\hol}(\DXS)$ is \emph{regular} if it satisfies
\par\smallskip\noindent
$(\mathrm{Reg}\,1)$\enspace
For each $s_o\in S$ and any $j\in\ZZ$, $Li^*_{s_o}\shh^j(\shm)\in\rD^\rb_{\rhol}(\DX)$.
\par\smallskip
An alternative and natural property of regularity would be the following:
\par\smallskip
\noindent$(\mathrm{Reg}\,2)$\enspace
For each $s_o\in S$, $Li^*_{s_o}\shm\in\rD^\rb_{\rhol}(\DX)$.
\par\smallskip
Regularity in either sense is the same property for objects of $\Mod(\DXS)$. Property $(\mathrm{Reg}\,1)$ is by definition compatible with the truncation functors while Property
$(\mathrm{Reg}\,2)$ is compatible with base change on $S$, meaning that, for any morphism $\pi:S'\to S$ of complex manifolds and any object $\shm\in \rD^\rb_{\hol}(\DXS)$
which satisfies $(\mathrm{Reg}\,2)$, the pullback $L\pi^\ast (\shm)\in \rD^\rb_{\hol}(\DXpS)$
satisfies $(\mathrm{Reg}\,2)$ too. We~enlarge the setting for further use since both conditions make sense for any complex in $\rD^\rb(\DXS)$.

\begin{proposition}\label{Prop:2}
Let $Y$ be a hypersurface of $X$.
On any complex manifold $S$,
\begin{enumerate}\renewcommand{\theenumi}{\roman{enumi}}
\item\label{Prop:21}
for a complex in $\rD^\rb_\coh(\DXS)$ (\resp in $\rD^\rb_\coh(\DXS(*Y))$), the condition $(\mathrm{Reg}\,1)$ is equivalent to $(\mathrm{Reg}\,2)$, and we denote both by $(\mathrm{Reg})$,
\item\label{Prop:22}
the category of coherent $\DXS$ (\resp coherent $\DXS(*Y)$)-modules satisfying $(\mathrm{Reg})$ is closed under taking extensions in the category $\Mod(\DXS)$(\resp in $\Mod(\DXS(*Y))$) and sub-quotients in the category $\Mod_{\coh}(\DXS)$ (\resp $\Mod_{\coh}(\DXS(*Y))$),
\item\label{Prop:22b}
the category of regular holonomic $\DXS$-modules is closed under taking extensions in $\Mod(\DXS)$ and sub-quotients in $\Mod_{\coh}(\DXS)$.
\end{enumerate}
\end{proposition}

We note that \eqref{Prop:22b} follows from \eqref{Prop:21} and \eqref{Prop:22} together with Lemma \ref{Lem:1}, so that we will focus on the latter properties. The proof in both cases is the same because it is based on the coherence of the rings involved. We provide it in the localized case.

\begin{lemma}\label{d12}
Condition $(\mathrm{Reg}\,1)$ implies $(\mathrm{Reg}\,2)$.
\end{lemma}
\begin{proof}
We argue by induction on the amplitude of the complex~$\shm$.
We may assume that $\shm\in \rD^{\geq 0}_{\coh}(\DXS(*Y))$
and we consider the following distinguished triangle
\begin{equation}\label{eq:A}
\shh^0\shm\to\shm\to \tau^{\geq1}\shm\To{+1},
\end{equation}
where $\tau^{\geq 1}$ is the truncation functor with respect to the natural $t$-structure on
$\rD^\rb_{\coh}(\DXS(*Y))$.
Let us assume that $\shm$ satisfies $(\mathrm{Reg}\,1)$, hence
by definition and induction, both $\shh^0\shm$ and $\tau^{\geq 1}\shm$ satisfy $(\mathrm{Reg}\,1)$.
As remarked, $\shh^0\shm$ satisfies $(\mathrm{Reg}\,2)$ too and by induction on the amplitude of $\shm$,
$\tau^{\geq 1}\shm$ satisfies $(\mathrm{Reg}\,2)$.
\end{proof}

\begin{proof}[Proof of Proposition \ref{Prop:2}]
For $d\geq0$, we denote by $(\mathrm{Reg}\,1)_d$, \resp $(\mathrm{Reg}\,2)_d$, the corresponding condition for $\dim S\leq d$.
If $\dim S=0$, $\eqref{Prop:21}_{0}$ holds true, and $\eqref{Prop:22}_{0}$ is proved \eg in \cite[Th.\,5.3.4]{Bjork93}. We thus assume from now on that $d\geq1$ and we proceed by induction on $d:=\dim S$, denoting by \eqref{Prop:21}$_{d}$ and
\eqref{Prop:22}$_{d}$ the statements of the proposition restricted to $\dim S\leq d$. We will prove the following implications for $d\geq 1$:
\begin{enumeratea}
\item\label{Prop:2a}
$\eqref{Prop:21}_{d-1} \wedge \eqref{Prop:22}_{d-1}\Longrightarrow \eqref{Prop:21}_{d}$;
\item\label{Prop:2b}
$\eqref{Prop:21}_{d}\wedge \eqref{Prop:22}_{d-1} \Longrightarrow \eqref{Prop:22}_{d}$.
\end{enumeratea}
\noindent

Let us start with \eqref{Prop:2a}. Assuming that both $ \eqref{Prop:21}_{d-1}$ and $\eqref{Prop:22}_{d-1}$ hold, we have to prove that $(\mathrm{Reg}\,2)_d\Rightarrow(\mathrm{Reg}\,1)_d$.
Due to the induction hypothesis $\eqref{Prop:21}_{d-1}$ we simply write $(\mathrm{Reg})_{d-1}$ for either
$(\mathrm{Reg}\,1)_{d-1}$ or $(\mathrm{Reg}\,2)_{d-1}$.

We note that $\shm\in\rD^\rb_{\coh}(\DXS(*Y))$ satisfies $(\mathrm{Reg}\,2)_d$ if and only if for each smooth codimension-one germ $(H,s_o)\subset (S,s_o)$, $Li_H^*\shm$ satisfies $(\mathrm{Reg})_{d-1}$. It is then enough to prove that, for such an $\shm$, $Li_H^*\shh^j\shm$ satisfies $(\mathrm{Reg})_{d-1}$ for any $j$ and~$H$. We shall argue by induction on the amplitude of $\shm$.
We may assume that $\shm\in \rD^{\geq 0}_{\coh}(\DXS(*Y))$ and we consider the distinguished triangle \eqref{eq:A}.
We deduce an isomorphism
\begin{equation}\label{eq:isoH-1}
\shh^{-1}Li^*_{H}\shh^0\shm\simeq \shh^{-1}Li^*_{H}\shm
\end{equation}
and an exact sequence
\[
0\to \shh^0 L i^*_{H}\shh^0\shm\to
\shh^0 L i^*_{H}\shm\to \shh^0 L i^*_{H}\tau^{\geq 1}\shm\to 0.
\]
(Note that $\shh^kLi^*_{H}\shh^0\shm=0$ for $k\neq0, -1$.)
Since $L i^*_{H}\shm$ satisfies $(\mathrm{Reg})_{d-1}$, so does $\shh^k L i^*_{H}\shm$ ($k=-1,0$), and so does $\shh^{-1}Li^*_{H}\shh^0\shm$ by \eqref{eq:isoH-1}. Since
\eqref{Prop:22}$_{d-1}$ is assumed to hold,
any coherent sub-quotient of $\shh^0 L i^*_{H}\shm$
satisfies $(\mathrm{Reg})_{d-1}$, hence so does $\shh^0 L i^*_{H}\shh^0\shm$,
which proves, following Lemma~\ref{d12}, that $Li^*_{H}\shh^0\shm$ satisfies $(\mathrm{Reg})_{d-1}$.
Thus $\shh^0\shm$ satisfies
$(\mathrm{Reg}\,2)_d$
and, by the distinguished triangle \eqref{eq:A},
$\tau^{\geq 1}\shm$ satisfies $(\mathrm{Reg}\,2)_d$.
Induction on the cohomological length applied to $\tau^{\geq 1}\shm$ implies that $Li_H^*\shh^j\shm$ satisfies $(\mathrm{Reg})_{d-1}$ for any $j\geq1$, which concludes the proof of \eqref{Prop:2a} and we now simply write $(\mathrm{Reg})_d$.

Let us now prove \eqref{Prop:2b}. The extension property in $\eqref{Prop:22}_d$ is clear. Let us consider stability by sub-quotients in $\Mod_\coh(\DXS(*Y))$. Let $\shm\in \Mod_\coh(\DXS(*Y))$ satisfy $(\mathrm{Reg})_d$. Given any short exact sequence
\begin{equation}\label{E:exactseq}
0\to \shm_1\to \shm\to\shm_2\to 0
\end{equation}
of coherent $\DXS(*Y)$-modules, we wish to prove that $\shm_1$ and $\shm_2$ satisfy $(\mathrm{Reg})_d$. Owing to our assumption on $\shm$, $Li^*_H\shm$ satisfies $(\mathrm{Reg})_{d-1}$ for any smooth codimension-one germ $(H,s_o)\subset (S,s_o)$, and it is enough to prove that either $Li^*_H\shm_1$ or $Li^*_H\shm_2$ satisfies $(\mathrm{Reg})_{d-1}$. From \eqref{E:exactseq} we obtain the long exact sequence:
\begin{multline}\label{E:coseq}
0\to \shh^{-1}Li^*_H\shm_1\to \shh^{-1}Li^*_H\shm\to \shh^{-1}Li^*_H\shm_2\\
\to
\shh^{0}Li^*_H\shm_1\to \shh^{0}Li^*_H\shm\to
\shh^{0}Li^*_H\shm_2\to 0.
\end{multline}
Since $Li^*_H\shm$ satisfies $(\mathrm{Reg})_{d-1}$ then so do $\shh^{-1}Li^*_H\shm_1$ and $\shh^{0}Li^*_H\shm_2$ owing to $\eqref{Prop:22}_{d-1} $, and it remains to be proved that either $\shh^{-1}Li^*_H\shm_2$, or $\shh^{0}Li^*_H\shm_1$, satisfies $(\mathrm{Reg})_{d-1}$. Let $s$ be a local coordinate on $S$ vanishing on $H$.

Let us denote by~$\shm'$ the pullback of $\shm'_2:=\tors_H(\shm_2)$ in $\shm$ and by $\shm_{2,\tf}$ the quotient $\shm_2/\shm'_2$. The commutative diagram below is Cartesian and its columns and rows are short exact sequences:
\begin{equation}\label{E:cs}
\begin{array}{c}
\xymatrix@R=.5cm{
0\ar[r] & \shm_1\ar[r] \ar@<-2pt>@{=}[d]& \shm' \ar@{^{ (}->}[d] \ar[r]&\shm'_2\ar[r] \ar@{^{ (}->}[d] &0 \\
0\ar[r] & \shm_1\ar[r]& \shm\ar[r]\ar[d] & \shm_2\ar[r] \ar[d]&0 \\
& & \shm_{2,\tf}\ar@{=}[r] & \shm_{2,\tf}&
}
\end{array}
\end{equation}
Since $\shm_{2,\tf}$ is $i^*_{H}$-acyclic, the exact sequence \eqref{E:coseq} for the middle column splits as
\begin{gather*}
\shh^{-1}Li^*_H\shm'\isom \shh^{-1}Li^*_H\shm,\\
0\to
\shh^{0}Li^*_H\shm'\to \shh^{0}Li^*_H\shm\to
\shh^{0}Li^*_H\shm_{2,\tf}\to 0,
\end{gather*}
which implies that $Li_H^*\shm'$ and $Li_H^*\shm_{2,\tf}$ satisfy $(\mathrm{Reg})_{d-1}$ by $\eqref{Prop:22}_{d-1}$, hence $\shm'$ and $\shm_{2,\tf}$ satisfy $(\mathrm{Reg})_{d}$. We now prove that $Li_H^*\shm'_2$ satisfies $(\mathrm{Reg})_{d-1}$, that will conclude the proof. Since $\shm'_2$ is $\DXS(*Y)$-coherent, there exists locally an integer $k\geq1$ such that $s^k\shm'_2=\nobreak0$.
We will prove by induction on~$k$ that any $\DXS(*Y)$-coherent torsion quotient of a $\DXS(*Y)$-coherent module satisfying $(\mathrm{Reg})_{d}$,
satisfies $(\mathrm{Reg})_{d}$ too.

If $k=1$, we have $\shh^{-1}Li^*_{H}\shm'_2\simeq \shh^{0}Li^*_{H}\shm'_2$, and, by $\eqref{Prop:22}_{d-1}$ the latter satisfies $(\mathrm{Reg})_{d-1}$, being a quotient of $\shh^{0}Li^*_{H}\shm'$.

If $k>1$, we argue with the following Cartesian commutative diagram, analogous to \eqref{E:cs}:
\[
\begin{array}{c}
\xymatrix@R=.5cm{
0\ar[r] & \shm_1\ar[r] \ar@{=}[d]& \shm'' \ar@{^{ (}->}[d] \ar[r]& \shm''_2\ar[r] \ar@{^{ (}->}[d] &0 \\
0\ar[r] & \shm_1\ar[r] & \shm' \ar[d] \ar[r]& \shm'_2\ar[r] \ar[d] ^{s}&0 \\
& & s\cdot\shm'_2\ar@{=}[r] & s\cdot\shm'_2 &
}
\end{array}
\]
By the induction hypothesis on $k$, $Li_H^*(s\cdot\shm'_2)$ satisfies $(\mathrm{Reg})_{d-1}$ since $s\cdot\shm'_2$ is a $\DXS(*Y)$-coherent quotient of $\shm'$ which is annihilated by $s^{k-1}$. It~follows that, by~considering the middle vertical sequence, so does $Li_H^*\shm''$. Since $\shm''_2$ is a $\DXS(*Y)$-coherent quotient of $\shm''$ annihilated by $s$, $Li_H^*(\shm''_2)$ satisfies $(\mathrm{Reg})_{d-1}$. It~follows that, by~considering the first horizontal sequence, so does $Li_H^*\shm_1$, hence finally, by considering the middle horizontal sequence, so does $Li_H^*\shm'_2$, as~wanted.
\end{proof}

\begin{corollary}\label{cor:prop2}
\label{cor:prop22}
The category $\rD^\rb_{\rhol}(\DXS)$ is a full triangulated subcategory of $\rD^\rb_{\coh}(\DXS)$ stable by duality.
\end{corollary}

\begin{proof}
We note that $(\mathrm{Reg}\,2)$ implies that $\rD^\rb_{\rhol}(\DXS)$ is a full triangulated subcategory of $\rD^\rb_{\hol}(\DXS)$. Since the latter is a full triangulated subcategory of $\rD^\rb_{\coh}(\DXS)$, the first assertion follows. Stability by duality follows from the same property in the absolute case (\cf\eg\cite[Th.\,5.4.15\,(4)]{Bjork93}), together with the isomorphism $Li^*_{s_o}\bD\shm\simeq\bD Li^*_{s_o}\shm$, which follows from Proposition~\ref{prop:3.1}.
\end{proof}

\subsection{Stability of regular holonomicity under base pullback and base pushforward}\label{subsec:3c}
For a proper morphism $f:X\to Y$, it has been shown in \cite[Cor.\,2.4]{MFCS2} that if~$\shm$ is an object of $\rD^\rb_\rhol(\DXS)$ with $f$-good cohomology, then $\Df_*\shm$ belongs to $\rD^\rb_\rhol(\DYS)$. On the other hand, stability of regular holonomicity (and hence its coherence) by pullback $\Df^*$ has been shown in \cite{FMFS1} only if $\dim S=1$ as a consequence of the Riemann-Hilbert correspondence proved there. This will be obtained in general by the proof of Theorem \ref{th:inverseimage} in Section \ref{subsec:pfinv}. In this section, we consider on the other hand the behaviour with respect to base pullback and pushforward.

\begin{proposition}[Stability under base pullback]\label{LReginverse}
Let $\pi:S'\to S$ be a morphism of complex manifolds and let $\shm$ be an object of $\rD^\rb_{\rhol}(\DXS)$. Then $L\pi^*\shm$ belongs to $\rD^\rb_{\rhol}(\DXSp)$.
\end{proposition}

\begin{proof}
We already know that $L\pi^*\shm$ belongs to $\rD^\rb_{\hol}(\DXSp)$ by Lemma \ref{lem:basepullback}. Regularity follows from the isomorphism of functors $Li_{s'_o}^*L\pi^*\simeq Li_{\pi(s'_o)}^*$ for any $s'_o\in\nobreak S'$.
\end{proof}

Let $t_o$ be a \emph{fat point} of $S$, that~is, a complex subspace of $S$ supported on a reduced point $|t_o|\in S$. In other words, the ideal $\shi_{t_o}\subset\sho_S$, which satisfies $\shi_{t_o}\subset\mathfrak{m}_{|t_o|}$ (which is the ideal associated to
$|t_o|$ such that its fiber in $|t_o|$ is the maximal ideal of $\sho_{S,|t_o|}$),
contains some power $\mathfrak{m}_{|t_o|}^k$.
By~abuse of notation we will still denote by $\shi_{t_o}$ (\resp $\mathfrak{m}_{|t_o|}$) the ideal $\shi_{t_o,|t_o|}$ (\resp $\mathfrak{m}_{|t_o|, |t_o|}$) and also the sheaf $p^{-1}\shi_{t_o}$
(\resp $p^{-1}\mathfrak{m}_{|t_o|}$).
Let $i_{t_o}:t_o\to S$ denote the natural morphism of complex spaces defined by the surjective morphism $\sho_S\to\sho_S/\shi_{t_o}$. For an $\sho_{\XS}$-module, \resp a~$\DXS$-module~$\shm$, we~set
\[
i_{t_o}^*\shm:=(\pOS/p^{-1}\shi_{t_o})\otimes_{\pOS}\shm,
\]
that we regard in a natural way as a $\DX$-module since $\sho_{S,|t_o|}/\shi_{t_o}$ is a finite-dimensional vector space. We define thereby the pullback functor
\[
Li_{t_o}^*:\rD^\rb(\DXS)\to\rD^\rb(\DX).
\]
We note that, endowed with its natural structure of $\DX$-module, $i_{t_o}^*\DXS$ is coherent. As~a~consequence, if $\shm$ is $\DXS$-coherent, then $Li_{t_o}^*\shm$ has $\DX$-coherent cohomology. In other words, $Li_{t_o}^*$ induces a functor $\rD^\rb_\coh(\DXS)\to\rD^\rb_\coh(\DX)$.

\begin{corollary}\label{cor:fatpoint}
Let $\shm$ be an object of $\rD^\rb_\coh(\DXS)$. Then, for any fat point~$t_o$ of~$S$, $Li_{t_o}^*\shm$ belongs to $\rD^\rb_\rhol(\DX)$ if and only if $Li_{|t_o|}^*\shm$ does so.
\end{corollary}

\begin{proof}
For $k\geq0$ we set $\shi_k=\shi_{t_o}\cap\mathfrak{m}_{|t_o|}^k$ (with $\mathfrak{m}_{|t_o|}^0:=\sho_S$), so that $\shi_1=\shi_{t_o}$ and $\shi_k=\mathfrak{m}_{|t_o|}^k$ for~$k$ large enough. It is enough to prove that \hbox{$p^{-1}(\shi_k/\shi_{k+1})\otimes^L_{\pOS}\shm$} belongs to $\rD^\rb_\rhol(\DX)$ for any $k$ if and only if $Li_{|t_o|}^*\shm$ does so. Since the sheaf $\shi_k/\shi_{k+1}$ is an $\sho_S/\mathfrak{m}_{|t_o|}$-module
supported on $|t_o|$ whose fiber is a finite dimensional vector space, we have
\begin{align*}
p^{-1}(\shi_k/\shi_{k+1})\otimes^L_{\pOS}\shm
&\simeq p^{-1}(\shi_k/\shi_{k+1})\otimes^L_{p^{-1}(\sho_S/\mathfrak{m}_{|t_o|})}(p^{-1}(\sho_S/\mathfrak{m}_{|t_o|})\otimes^L_{\pOS}\shm)\\
&=(\shi_k/\shi_{k+1})\otimes Li^*_{|t_o|}\shm,
\end{align*}
and the conclusion follows.
\end{proof}

\begin{setting}\label{set:local}
We consider a local setting where $S$ is a polydisc $\Delta^d$ written as $\Delta^{d-1}\!\times\!\nobreak\Delta=S'\times\Delta$, and we denote by $q:S\to S'$ the projection $(s',t)\mto s'$, and we keep the same notation after taking the product with $X$. Recall (\cf\eg\cite[Prop.\,A.14]{Ka2}) that the sheaves of rings $q^{-1}\sho_{S'}$, $q^{-1}\sho_{\XS'}$ and $q^{-1}\DXSp$ are Noetherian.

Let $h(s',t)=t^k+\sum_{i=0}^{k-1}h_i(s')t^i$ be a Weierstrass polynomial, with $h_i$ holomorphic on $S'$ and let $T=h^{-1}(0)$. The equivalence between the categories of
\begin{itemize}
\item
coherent $\sho_S$-modules supported on $T$,
\item
$h$-nilpotent coherent $q^{-1}\sho_{S'}[t]$-modules,
\item
$h$-nilpotent $q^{-1}\sho_{S'}[t]$-modules which are $q^{-1}\sho_{S'}$-coherent,
\end{itemize}
extends in a natural way to $\sho_{\XS}$, to $\DXS$ and to $\DXS(*Y)$ for a given hypersurface $Y$ of $X$. For example, in one direction, if $\shm$ is $\DXS$-coherent and supported on $\XT$, then $\shm$ is $q^{-1}\DXSp$-coherent and each local section is annihilated by some power of~$h$. By taking local $\DXS$-generators of $\shm$, we conclude that there exists locally an integer $\ell\geq1$ such that $h^\ell\shm=0$. Conversely, if $\shm$ satisfies the latter property, we can regard it as an $h$-nilpotent coherent $q^{-1}\DXSp[t]$-module and the associated $\DXS$-module is $\DXS\otimes_{q^{-1}\DXSp[t]}\shm$.
\end{setting}

By definition, the characteristic variety of a coherent $q^{-1}\DXSp$-module is (locally) the support in $(T^*X)\times S$ of the graded coherent $q^{-1}\gr^F\DXSp$-module with respect to any local coherent $q^{-1}F_\sbullet\DXSp$-filtration. Such a module is said to be holonomic if this support is contained in $\Lambda\times S$ for some Lagrangian variety $\Lambda\subset T^*X$.

\begin{remark}[Holonomic and regular holonomic $q^{-1}\DXSp$-modules]\label{rem:TS}
Given a coherent $\DXS$-module $\shm$ supported on $T$ as above, one checks that a coherent $F_\sbullet\DXS$-filtration is also a coherent $q^{-1}F_\sbullet\DXSp$-filtration by the above correspondence. As a consequence, such a module is $\DXS$-holonomic if and only if it is $q^{-1}\DXSp$-holonomic.

On the other hand, we say that a holonomic $q^{-1}\DXSp$-module $\shm$ is regular if, for any $s'_o\in S'$ and any $s_o\in q^{-1}(s'_o)$, the holonomic $\DX$-module $i_{s_o}^{-1}(Li_{q^{-1}(s'_o)}^*\shm)$ is regular.
\end{remark}

\begin{corollary}\label{cor:TS}
A coherent $\DXS$-module $\shm$ supported on $X\times T$ is regular holonomic if and only if, when regarded as a $q^{-1}\DXSp$-module, it is regular holonomic.
\end{corollary}

\begin{proof}
By Remark \ref{rem:TS}, we only need to check regularity, and we can suppose from the start that~$\shm$ is $\DXS$-holonomic supported on $X\times T$. We assume first that $\shm$ is $q^{-1}\DXSp$-regular. We wish to prove that $Li_{s_o}^*\shm$ is $\DX$-regular holonomic for any $s_o\in S$. It is enough to prove this for $s_o\in T$. Let us choose coordinates $(s'_1,\dots,s'_{d-1},t)$ of $S'\times\Delta$ centered at $s_o$. There exists $k\geq1$ such that $\sho_S/(h^k)\otimes^L_{\sho_S}\shm\simeq\shm[1]\oplus\shm$, since $h^k\shm=0$ for some $k\geq1$, so, by the assumption,
\[
\sho_S/(s'_1,\dots,s'_{d-1})\otimes^L_{\sho_S}\bigl[\sho_S/(h^k)\otimes^L_{\sho_S}\shm\bigr]=\sho_S/(s'_1,\dots,s'_{d-1},h^k)\otimes^L_{\sho_S}\shm
\]
is a regular holonomic $\DX$-module. Since the support of $\sho_S/(s'_1,\dots,s'_{d-1},h^k)$ is a finite union of fat points in $S$ (defined by the ideal generated by $h^k(0,t)$), one of which is supported at $s_o$, we~conclude that $Li_{s_o}^*\shm$ is a regular holonomic $\DX$-module by applying Corollary \ref{cor:fatpoint}. The converse is proved similarly.
\end{proof}

\begin{theorem}[Stability under projective base pushforward]\label{th:Regdirect}
Let $\pi:S\to S'$ be a projective morphism of complex manifolds and let~$\shm$ be an object of $\rD^\rb_{\rhol}(\DXS)$. Assume that the cohomology of $\shm$ is $\pi$-good. Then $R\pi_*\shm$ belongs to $\rD^\rb_{\rhol}(\DXSp)$.
\end{theorem}

\begin{proof}
A standard consequence of Proposition~\ref{prop:basepushhol} is that $R\pi_*\shm$ belongs to $\rD^\rb_{\hol}(\DXSp)$. Furthermore, we note that the question is local with respect to $X$ and to $S'$.

\subsubsection*{Step 1: Reduction to the case where $S'$ is a point}
Since $\pi$ is projective, we can regard~$\pi$ (locally with respect to $S'$) as the composition of the inclusion $S\hto\PP^m\times S'$ and the projection $\PP^m\times S'\to S'$ for a suitable $m$. Moreover, we note that the result is easy if~$\pi$ is a closed embedding. We~can thus assume that $\pi$ is a projection $S=\PP^m\times S'\to S'$. For a complex $\shm$ in $\rD^\rb_{\rhol}(\DXS)$, proving the $\DXSp$-regularity of $R\pi_*\shm$ amounts to proving the $\DX$-regularity of $Li_{s'_o}^*R\pi_*\shm$ for any $s'_o\in S'$. Let $\fm_{s'_o}$ be the maximal ideal sheaf of~$\sho_{S'}$ at $s'_o$. Let us set $\PP^m_{s'_o}=\PP^m\times\{s'_o\}$ and consider the following Cartesian square
\[
\xymatrix{
X\times\PP^m_{s'_o}
\ar[r]^-{\pi'} \ar@{^{(}->}[d]_{i'}
&
X\times \{s'_o\} \ar@{^{ (}->}[d]^{i_{s'_o}} \\
\XS \ar[r]^-{\pi} & \XS'
}
\]
Then we have
\begin{align*}
Li_{s'_o}^*R\pi_*\shm&=i_{s'_0}^{-1}((\sho_{\XS'}/p^{-1}(\fm_{s'_o})\sho_{\XS'})\otimes_{\sho_{\XS'}}^LR\pi_*\shm)\\
&\simeq i^{-1}_{s'_o}R\pi_*\bigl(\pi^{-1}(\sho_{\XS'}/p^{-1}(\fm_{s'_o})\sho_{\XS'})\otimes^L_{\pi^{-1}\sho_{\XS'}}\shm\bigr)\\
&\simeq R\pi'_*{i'}^{-1}\bigl(\pi^{-1}(\sho_{\XS'}/p^{-1}(\fm_{s'_o})\sho_{\XS'})\otimes^L_{\pi^{-1}\sho_{\XS'}}\shm\bigr)\\
&\overset{(*)}{\simeq}R\pi'_*\bigl(\sho_{X\times \PP^m_{s'_o} }\otimes^L_{i'^{-1}\sho_{\XS}}i'^{-1}\shm\bigr)\\
&=R\pi'_*(Li^{\prime*}\shm),
\end{align*}
where the isomorphism $(*)$ follows by extension of scalars since $\shm$ is a complex of $\sho_{\XS}$-modules. By Proposition \ref{LReginverse}, $Li^{\prime*}\shm$ is a complex with regular holonomic cohomologies, so that if we know the theorem for~$\pi'$, we deduce it for $\pi$.

From now on, we assume that $S'$ is a point. In such a case, $S$ is a projective space~$\PP^m$. Since the question is local with respect to $X$ and since $S$ is compact, we can assume that the sets $I$ and $J_i$ occurring in \eqref{eq:chardec} are finite, so the $S$-support $T$ of~$\shm$ is a~closed analytic subset of $S$. We argue by induction on the dimension of the $S$-support of~$\shm$.
\subsubsection*{Step 2: Case $\dim \Supp_S\shm=0$}

If the $S$-support of $\shm$ has dimension zero, it consists of a finite number of points, and it is enough to consider the case where the support consists of one point $s_o\in S$. By a standard argument we may assume that $\shm$ is concentrated in degree zero and locally we can assume that there exists $k\geq1$ such that, denoting by $\fm_{s_o}$ the maximal ideal sheaf of~$s_o$ in~$S$, we have $\fm_{s_o}^k\shm=0$. If we denote by $t_o$ the fat point supported by~$s_o$ with ring $\sho_S/\fm_{s_o}^k\sho_S$, we conclude that $\shm=i_{s_o*}i_{t_o}^*\shm$. Therefore, $R\pi_*\shm=
R\pi_*i_{s_o*}i_{t_o}^*\shm=i_{t_o}^*\shm$. In~this case, the theorem follows from Corollary \ref{cor:fatpoint}.

\subsubsection*{Step 3: Case $\dim \Supp_S\shm\geq1$}
We recall that, in this step, $\pi$ is the constant map on $S=\PP^m$.

\subsubsection*{\textup{(i)} Proof of the regularity of $\pi_*\shm$}
Let $d\geq1$. We assume now that the statement holds true for any complex $\shm'$ whose $S$-support has dimension $<d$ and we aim at proving the same property for any complex $\shm$ with $S$-support $T$ of dimension $d$. We~may then reduce again to the case where $\shm$ is a single module.

We first prove that $\pi_*\shm$ is regular holonomic (instead of all modules $R^k\pi_*\shm$).

One checks that the $\DXS$-submodule~$\shm'$ of $\shm$ consisting of local sections $m$ such that the $S$\nobreakdash-support of $\DXS\cdot m$ has dimension $<d$ is holonomic (denoting by $T_{<d}$ the union of irreducible components of $T$ of dimension $<d$, $\shm'$ is locally defined as the sheaf of local sections annihilated by some power of the ideal $\cI_{T_{<d}}$ of $T_{<d}$ in $\sho_S$). By Proposition \ref{Prop:2}\eqref{Prop:22b}, $\shm'$ is regular holonomic and, by the induction hypothesis, $\pi_*\shm'$ is regular holonomic. It is thus enough to prove regularity of $\pi_*(\shm/\shm')$, and we can likewise assume that
\begin{enumerate}
\step\label{step:assumption}
$\shm$ has no nonzero coherent submodule with $S$-support of dimension $<d$.
\end{enumerate}
Let us denote by $\Lambda$ a Lagrangian subvariety of $T^*X$ such that the characteristic variety of $\shm$ is contained in $\Lambda\times S$.
Since $T$ is compact, it has a finite number of irreducible components and we index by $I_d$ those which are of dimension $d$. We choose a point~$s_i$ on the smooth part of each $T_i$ ($i\in I_d$). For the sake of simplicity, we~denote by $\bms_o$ the finite set $\{s_i\mid i\in I_d\}$.

There is a natural morphism of $\DX$-modules
\begin{equation}\label{eq:piiso}
\pi_*\shm\to i_{\bms_o}^{-1}\shm
\end{equation}
which associates to a section $m\in\Gamma(U;\pi_*\shm)=\Gamma(U\times S;\shm)$ ($U$ open in $X$) its germ along $U\times \bms_o$. Since $\shm$ has no $S$-torsion supported in dimension $<d$, this morphism is \emph{injective}. Indeed, let $x_o\in U\subseteq X$ and let $m\in\Gamma(U\times S;\shm)$ belong to the kernel, that~is, such that its germ at $(x_o,\bms_o)$ vanishes, \ie $m$ is zero in some neighborhood of $(x_o,\bms_o)$ that we write $U\times V$, up to shrinking $U$. Let us suppose by contradiction that $m\not=0$. Let us consider the coherent $\DUS$-submodule $\DUS\cdot m\subset\shm_{|U\times S}$ which is holonomic with $S$-support contained in $T$. By \eqref{step:assumption} $\DUS\cdot m$ has $S$-support of dimension $d$. Hence any irreducible $d$-dimensional component of its $S$-support is equal to some $T_i$ for a suitable $i\in I_d$, but since $m$ is zero on $U\times V$ this is not possible.

Let $\fm_{\bms_o}$ denote the ideal sheaf of $\bms_o$ in $\sho_S$. For each $k\geq1$, we consider the induced morphism
\[
\pi_*\shm\to i_{\bms_o}^{-1}\shm/i_{\bms_o}^{-1}\fm_{\bms_o}^k\shm.
\]
The proposition will be proved if we prove that this morphism is injective for $k$ large enough, since the right-hand side is regular holonomic by Corollary \ref{cor:fatpoint}.

Its kernel $\shn_k$ is a coherent, hence holonomic, $\DX$-submodule of $\pi_*\shm$ and the sequence $(\shn_k)$ is decreasing with characteristic variety contained in $\Lambda$. It is thus stationary. Let $\shn\subset\pi_*\shm$ denote this constant value. We conclude that the map
\[
\shn\to i_{\bms_o}^{-1}\shm/i_{\bms_o}^{-1}\fm_{\bms_o}^k\shm
\]
is zero for any $k$. Since $\shn=\shn_k$ for $k$ large enough, we aim at proving that $\shn=0$.

The image of $\shn$ by \eqref{eq:piiso} is contained in $\bigcap_ki_{\bms_o}^{-1}\fm_{\bms_o}^k\shm$, and since \eqref{eq:piiso} is injective, it suffices to prove that $\shm':=\bigcap_ki_{\bms_o}^{-1}\fm_{\bms_o}^k\shm=0$. For $x_o\in X$, let us denote by $\shd_{(x_o,\bms_o)}$ the germ at $(x_o,\bms_o)$ of $\DXS$ and similarly by $\shm_{(x_o,\bms_o)}$ that of $\shm$. It~suffices thus to prove that, for all $x_o\in X$, the germ $\shm'_{(x_o,\bms_o)}=\bigcap_k\fm_{\bms_o}^k\shm_{(x_o,\bms_o)}$ is zero. We note that $\shm'_{(x_o,\bms_o)}$ is of finite type over $\shd_{(x_o,\bms_o)}$ since the latter ring is Noetherian, hence $\shm'_{(x_o,\bms_o)}$ is holonomic with characteristic variety contained in $\Lambda_{x_o}\times(S,\bms_o)$, where~$\Lambda_{x_o}$ is the germ of $\Lambda$ along $T^*_{x_o}X$. Furthermore, it satisfies $\shm'_{(x_o,\bms_o)}=\fm_{\bms_o}\shm'_{(x_o,\bms_o)}$. The Nakayama-type argument given in the proof of \cite[Prop.\,1.9(1)]{MFCS2} shows that $\shm'_{(x_o,\bms_o)}=0$ for each $x_o$. This concludes the proof of Step~3(i).

\subsubsection*{\textup{(ii)} Regularity of $R^k\pi_*\shm$}
We use the result of (i) in order to prove holonomicity and regularity of $R^k\pi_*\shm$ for all $k$.

For any $\ell\in\ZZ$, we consider the line bundle $\sho_S(\ell)$ that we can realize as $\sho_S(\ell H)$ for any hyperplane $H$ of $S$. For a $\DXS$-module or an $\sho_{\XS}$-module $\shn$, we set $\shn(\ell)=\pOS(\ell)\otimes_{\pOS}\nobreak\shn$.

We apply the technical lemma~\ref{Lemma:l} below to $\shm$. We note that $\shm(\ell)$ is $\DXS$ regular holonomic: indeed, this is a local property on $S$, and locally $\sho_S(\ell)\simeq\sho_S$. Also, $\shm(\ell)$ is $\pi$-good. Furthermore, we choose a hyperplane $H\subset S$ such that $\dim H\cap T<\dim T$, and we realize $\sho_S(\ell)$ as $\sho_S(\ell H)$. We also consider $\shm(\ell H)$ with such a choice of $H$. From Step~3(i), we know that $\pi_*\shm(\ell H)$ is $\DX$-regular. Let us consider the sheaf $\sho_{\ell H}$ defined by the exact sequence
\[
0\to\sho_S\to\sho_S(\ell H)\to\sho_{\ell H}\to0
\]
yielding to the distinguished triangle
\[
\shm\to \shm(\ell H)\to\shn\To{+1}
\]
where $\shn$ is a complex in $\rD^\rb_{\rhol}(\DXS)$ with $S$-support contained in $T\cap H$. We deduce the long exact sequence
\[
0\to H^{-1} R\pi_*\shn\to \pi_*\shm\to\pi_*\shm(\ell H)\to H^0R\pi_*\shn\to R^1\pi_*\shm\to0
\]
and the isomorphisms $R^{k+1}\pi_*\shm\simeq H^kR\pi_*\shn$ for $k\geq1$. By the induction hypothesis, $R\pi_*\shn$ belongs to $\rD^\rb_{\rhol}(\DX)$. The latter isomorphism implies the regularity of $R^{k+1}\pi_*\shm$ for $k\geq1$, and, together with Corollary \ref{cor:prop2}, it~implies that of $\pi_*\shm$ and $R^1\pi_*\shm$.
\end{proof}

\begin{lemma}\label{Lemma:l}
Let us assume $S=\PP^ m$. Let $\shn$ be a coherent $\pi$-good $\DXS$-module. Then, for any $x_o\in X$, there exists a neighborhood $\nb(x_o)\subset X$ and an integer $\ell\geq0$ such that, for any $k\geq1$, $R^k\pi_*\shn(\ell)_{|\nb(x_o)}=0$.
\end{lemma}

A similar result is well-known to hold for a coherent $\sho_{\XS}$-module, as a consequence of Grauert-Remmert's Theorems A and B (\cf\cite{G-R58} and \cite[Th.\,IV.2.1]{B-S77b}).

\begin{proof}
For any $N\geq1$, by iteration of Lemma \ref{lem:cplxqgood}\eqref{lem:cplxqgood3} we obtain a distinguished triangle on $\nb(x_o)\times S$
\[
\shn'[N]\to\DXS\otimes_{\sho_{\XS}}\shl_\sbullet\to\shn\To{+1}.
\]
Let us choose $N=2m+1$ and let $\ell$ be an integer given by the Grauert-Remmert theorems for $\shl_i$ ($i=-N,\dots,0$) in $\nb(x_o)$. We~consider the same triangle obtained after tensoring with $\pOS(\ell)$. Since $\pi$ has cohomological dimension $2m$, we~find that $R^k\pi_*\shn'(\ell)[N]=0$ for any $k\geq0$. On~the other hand, on $\nb(x_o)$ (\cf \eqref{eq:DXS}),
\begin{align*}
R\pi_*(\DXS\otimes_{\sho_{\XS}}\shl_i(\ell))&\simeq R\pi_*(\pi^{-1}\DX\otimes_{\pi^{-1}\sho_X}\shl_i(\ell))\\
&\simeq\DX\otimes_{\sho_X}R\pi_*\shl_i(\ell)=\DX\otimes_{\sho_X}\pi_*\shl_i(\ell).
\end{align*}
Since $R\pi_*\shn(\ell)_{|\nb(x_o)}$ is in non negative degrees it follows that $R\pi_*\shn(\ell)_{|\nb(x_o)}$ is isomorphic to the complex $\tau^{\geq 0}(\DX\otimes_{\sho_X}\pi_*\shl_\sbullet(\ell))$ and since the latter is in non-positive degrees, we conclude that $R\pi_*\shn(\ell)_{|\nb(x_o)}\simeq\pi_*\shn(\ell)_{|\nb(x_o)}$.
\end{proof}

\subsection{Integrable regular holonomic \texorpdfstring{$\DXS$}{DXS}-modules}\label{subsec:3b}

The following proposition, answering a question of Lei Wu, shows that, under a suitable condition on the characteristic variety, a coherent $\DXS$-submodule of a regular holonomic $\shd_{\XS}$-module is relatively regular holonomic. We call such $\DXS$-modules \emph{integrable} since their relative connection can be lifted as an integrable connection, \ie a $\shd_{\XS}$-module structure.The condition on the characteristic variety is due to the restrictive definition of a holonomic $\DXS$-module (\cf after Lemma \ref{lem:cplxqgood}).

\begin{proposition}\label{prop:Wu}
Let $X$ and $S$ be complex manifolds and let $\shn$ be a regular holonomic $\shd_{\XS}$-module. Assume that the characteristic variety of $\shn$ is contained in \hbox{$\Lambda\!\times\! T^*S$} for some conic Lagrangian closed analytic subset $\Lambda$ in $T^*X$. Let $\shm$ be a coherent $\DXS$-submodule of~$\shn$. Then $\shm$ is a regular holonomic $\DXS$-module.
\end{proposition}

Let us first recall two properties (\eqref{step:relchar} and \eqref{step:reholregdim1} below) of regular holonomic $\shd_{\XS}$-modules that will be useful for the proof and do not depend on the assumption on $\Char\shn$ made in the proposition.

Let $\shn$ be any regular holonomic $\shd_{\XS}$-module and let $\shm$ be a coherent $\DXS$-submodule of~$\shn$. Any irreducible component of $\Char\shn$ projects to $\XS$ as an irreducible closed analytic subset $Z$ and, denoting by $Z^\circ$ the smooth part of $Z$, this irreducible component is the closure of the conormal bundle $T^*_{Z^\circ}(\XS)$, that we denote by $T^*_Z(\XS)$. We also denote by $Z^{\circ\circ}$ the open set of the smooth locus $Z^\circ$ on which $p_{Z^\circ}:Z^\circ\to S$ has maximal rank. The closure in the relative cotangent space $T^*(\XS/S)$ of the relative conormal bundle $T^*_{p|Z^{\circ\circ}}(\XS/S)$ is denoted by $T^*_{p|Z}(\XS/S)$, and is a conic analytic subset of $T^*(\XS/S)$.
\begin{enumerate}
\step\label{step:relchar}
According to \textup{\cite[Th.\,3.2]{Sabbah86II}} any irreducible component of $\Char\shm$ is equal to $T^*_{p|Z}(\XS/S)$ for some $Z$ such that $T^*_Z(\XS)$ is an irreducible component of $\Char\shn$.
\end{enumerate}

\smallskip

Assume furthermore that $\dim S=1$. For each $s_o\in S$, we denote by $s$ a local coordinate on~$S$ vanishing at $s_o$. We postpone after the proof of Proposition \ref{prop:Wu} that of the following classical result.

\begin{enumerate}
\step\label{step:reholregdim1}
Under the above assumptions, the kernel and cokernel of $s:\shm\to\shm$ are regular holonomic $\DX$-modules.
\end{enumerate}

\begin{proof}[Proof of Proposition \ref{prop:Wu}, Step 1]
We now add the assumption on $\Char\shn$. We first show that~$\shm$ is relative holonomic with characteristic variety contained in $\Lambda\times S$. We note that $T^*_Z(\XS)$ is contained in $\Lambda\times T^*S$ if and only if $Z$ decomposes as the product $Y\times T$ for some irreducible closed analytic subsets $Y\subset X$ and $T\subset S$ and $T^*_YX$ is an irreducible component of $\Lambda$ (this seen by considering first the smooth part~$Z^\circ$). It is then easily seen that
$
T^*_{p|Z}(\XS/S)=(T^*_YX)\times T,
$
hence is contained in $\Lambda\times S$. The conclusion follows from \eqref{step:relchar}.
\end{proof}

\begin{proof}[Proof of Proposition \ref{prop:Wu}, Step 2]
It remains to show the relative regularity of $\shm$. We argue by induction on $d=\dim S$. If $d=1$, relative regularity is provided by \eqref{step:reholregdim1}, since we already have relative holonomicity by Step~1. We thus assume that the statement of the proposition holds whenever $\dim S\leq d-1$ and we assume $\dim S=d\geq2$. The question being local, we fix a local coordinate $s$ in $S$, defining a smooth hypersurface $H:=\{s=0\}$. Locally, we can assume that $S=H\times\CC$. According to Proposition \ref{Prop:2}, we~are reduced to proving that
\begin{enumerate}
\step\label{step:kercokers}
the kernel and cokernel of $s\colon \shm\!\to\!\shm$ are regular holonomic $\DXH$\nobreakdash-modules.
\end{enumerate}

For the sake of clarity we set
\[
X'=\XH,\ S'=\CC,\ p':\XpS'\to S'\quad\text{so that}\quad X=\XpS',\quad S=H\times S',
\]
We denote by $\shm_1$ the $\shd_{\XpS'/S'}$-submodule of the regular holonomic $\shd_{\XpS'}$-mod\-ule~$\shn$ generated by~$\shm$. It is $\shd_{\XpS'/S'}$--coherent, so by \eqref{step:reholregdim1}, the kernel $\shn_1$ and the cokernel $\shn'_1$ of $s:\shm_1\to\shm_1$ are $\shd_{X'}$-regular holonomic.

\begin{claim*}
The characteristic varieties $\Char\shn_1$ and $\Char\shn'_1$ are contained in $\Lambda\times T^*H$.
\end{claim*}

\begin{proof}
We apply \eqref{step:relchar} to $\shm_1\!\subset\!\shn$ and to the map~\hbox{$\XpS'\!\to\!S'$}. Since any irreducible component of $\Char\shn$ takes the form $T^*_YX\times T^*_TS=T^*_Z(\XpS')$ with $Z=Y\times T$, any irreducible component of $\Char\shm_1$ takes the form \hbox{$T^*_{p'_{|Z}}(\XpS'/S')$} for some such~$Z$. Denoting by $q:S\to S'$ the projection $H\times S'\to S'$, the latter space reads
\[
(T^*_YX)\times T^*_{q_{|T}}(H\times S'/S').
\]
The fiber of the composition $T^*_{q|T}(H\times S'/S')\to T\to S'$ above $s=0$ ($s$ is the coordinate on $S'$) is contained in $T^*H$ and is Lagrangian: indeed, it has dimension $\dim H$ since $T^*_{s|T}(H\times S'/S')$ has dimension $\dim H+1$, and it is easily seen to be isotropic. We~denote it by $\Lambda_T$. The characteristic varieties of $\shn_1$ and $\shn'_1$ are contained in the fiber above $s=0$ of $\Char\shm_1$, hence in the union, over $Y,T$ occurring in $\Char\shn$, of~the Lagrangian subsets $T^*_YX\times\Lambda_T$, as claimed, since $\Lambda$ is nothing but $\bigcup_YT^*_YX$.
\end{proof}

As a consequence, $\shn_1$ is a regular holonomic $\shd_{\XH}$-module satisfying the assumption of the proposition, and $\ker (s:\shm\to\shm)$ is a coherent $\DXH$-submodule of it. We can thus apply the induction hypothesis to conclude that \eqref{step:kercokers} holds for $\ker s$.

On the other hand, for each $k\geq0$, by applying \eqref{step:reholregdim1} to $s^k\shm_1$, we obtain, according to the claim, that $s^k\shm_1/s^{k+1}\shm_1$ is a regular holonomic $\shd_{\XH}$-module with characteristic variety contained in $\Lambda\times T^*H$. Then
\[
\shm^{(k)}:=(\shm\cap s^k\shm_1)/(\shm\cap s^{k+1}\shm_1)
\]
is a coherent $\DXH$-submodule of $s^k\shm_1/s^{k+1}\shm_1$. The induction hypothesis implies that $\shm^{(k)}$ is $\DXH$-regular holonomic. It follows by Proposition \ref{Prop:2}(iii) that \hbox{$\shm/(\shm\cap s^{k+1}\shm_1)$} is also a regular holonomic $\DXH$-module. It remains to notice that coherence implies that, locally on~$X$, there exists $k$ such that $\shm\cap s^{k+1}\shm_1 \subset s\shm$, so that $\shm/s\shm$ is a coherent quotient of a regular holonomic $\DXH$-module.
Again by Proposition \ref{Prop:2}(iii) it is also a regular holonomic $\DXH$-module, concluding the proof of \eqref{step:kercokers}.
\end{proof}

\begin{proof}[Proof of \eqref{step:reholregdim1}]
We~consider the Kashiwara-Malgrange $V$-filtration $V_\sbullet\shn$ of $\shn$ relative to the function~$s$ (\cf \eg \cite{M-S86b,M-M04}). This is an increasing filtration indexed by~$\ZZ$ which satisfies, owing to the regularity of $\shn$, the following properties:
\begin{itemize}
\item
$V_k\shn$ is $\DXS$-coherent for any $k\in\ZZ$ (\cite[Th.\,(4.12.1)]{M-S86b}),
\item
for each $k\in \ZZ$, $s(V_k\shn)\subset V_{k-1}\shn$ and for $k\leq-1$, the multiplication by $s$ on $V_k\shn$ is injective with image $V_{k-1}\shn$ (\cf\eg\cite[Prop.\,(4.5.2)]{M-S86b}),
\item
each $\gr_k^V\shn$ is a regular holonomic $\DX$-module (\cf\eg\cite[Cor.\,4.7-5]{M-M04}).
\end{itemize}

From the second point we deduce that, for $k\leq-1$, $Li_{s_o}^*V_k\shn=i_{s_o}^*V_k\shn=\gr_k^V\shn$, and by the third point the latter is $\DX$-regular holonomic. For $k\geq0$, $V_k\shn$ is a successive extension of~$V_{-1}\shn$ by regular holonomic $\DX$\nobreakdash-modules $\gr_\ell^V\shn$. Therefore, for each $k$, the kernel and cokernel of $s:V_k\shn\to V_k\shn$ are $\DX$-regular holonomic: this is proved by induction on $k$
since it is clear for $k<0$, and the inductive step follows by considering the snake lemma applied to the following commutative diagram of exact sequences
\[
\xymatrix@R=.7cm{
0\ar[r]&V_{k-1}\shn\ar[r]\ar[d]_s&V_k\shn\ar[d]^s\ar[r]&\gr_k^V\shn\ar[r]\ar[d]^0&0\\
0\ar[r]&V_{k-1}\shn\ar[r]&V_k\shn\ar[r]&\gr_k^V\shn\ar[r]&0
}
\]
We deduce that $sV_k\shn/s^\ell V_k$ is regular holonomic for any $\ell \geq 1$. As any $\DXS$-coherent submodule of $\shn$, locally on $\XS$, $\shm$ is contained in $V_k\shn$ for some $k\gg0$ and, by the Artin-Rees lemma, $s\shm$ contains $\shm\cap s^\ell V_k\shn$ for some $\ell\geq1$. Since $\ker (s:\shm\to\shm)$ is contained in $\ker( s:V_k\shn\to V_k\shn)$, the former is regular holonomic since the latter is so. On the other hand, by considering the inclusion and the quotient maps
\[
sV_k\shn/s^\ell V_k\shn\hfrom (\shm\cap sV_k\shn)/(\shm\cap s^\ell V_k\shn)\to\hspace*{-5mm}\to(\shm\cap sV_k\shn)/s\shm
\]
we conclude that $(\shm\cap sV_k\shn)/s\shm$ is also $\DX$-regular holonomic. Furthermore, by considering the exact sequence
\[
0\to(\shm\cap sV_k\shn)/s\shm\to\shm/s\shm\to V_k\shn/sV_k\shn,
\]
we conclude that $\shm/s\shm$ is $\DX$-regular holonomic.
\end{proof}

\section{Holonomic \texorpdfstring{$\DXS$}{DXS}-modules of D-type and applications}\label{sec:4}

\subsection{\texorpdfstring{$S$}{S}-locally constant sheaves and their associated relative connections}
Let $X$ be a connected complex manifold and let $L$ be an $S$-locally constant sheaf of $\pOS$-modules on $\XS$ (\cf\cite[App.]{MFCS2}). For any $x_o\in X$, $L$ is uniquely deter\-mined from a monodromy representation $\pi_1(X,x_o)\to\Aut_{\sho_S}(G)$ with $G=i_{x_o}^{-1}L$. As a consequence, there exists an $\sho_S$-module $G$ such that, for any $1$-connected open subset~$U$ of $\XS$, there exists an isomorphism $L_{|U\times S}\simeq p_U^{-1}G$. Two choices of $G$ are isomorphic, but non-canonically in general. Furthermore, $L$ is $\pOS$-coherent if and only if $G$ is $\sho_S$-coherent.

To any $S$-locally constant sheaf $L$ one can associate an exact sequence of sheaves of $\pOS$-modules\vspace*{-3pt}
\begin{equation}\label{eq:ttf}
0\to L_\rt\to L\to L_\tf\to0
\end{equation}
where $L_\rt$ denotes the subsheaf of $\pOS$-torsion and $L_\tf$ the maximal $\pOS$-torsion-free quotient. Then $L_\rt$ and $L_\tf$ are $S$-locally constant and the previous exact sequence yields (for any choice of~$G$) to the exact sequence of $\sho_{S}$-modules
\[
0\to G_\rt\to G\to G_\tf\to0,
\]
with $G_\rt$ and $G_\tf$ defined similarly. The following is straightforward.

\pagebreak[2]
\begin{lemma}\label{lem:behaviourpi}
Let $\pi:S\to S'$ be a holomorphic map between complex manifolds.
\begin{enumerate}
\item\label{lem:behaviourpi1}
If $L$ is an $S$-locally constant sheaf on $\XS$ with associated $\sho_S$-module $G$, then $\pi_!L$ is $S'$-locally constant on $\XS'$ and $\pi_!G$ is an associated $\sho_{S'}$-module to~$\pi_!L$.

\item\label{lem:behaviourpi2}
If $L'$ is an $S'$-locally constant sheaf on $\XS'$ with associated $\sho_S'$-module~$G'$, then $\pi^*L'$ is $S$-locally constant on $\XS$ and $\pi^*G'$ is an associated $\sho_S$-module to~$\pi^*L'$.
\end{enumerate}
\end{lemma}
We shall denote by $\rd_{\XS/S}:\sho_{\XS}\to\Omega^1_{\XS/S}$ the relative differential associated to $p$.
Let us recall the Riemann-Hilbert correspondence for coherent $S$-local systems proved in \cite[Th.\,2.23 p.\,14]{De}, in the particular case of a projection $\XS\to S$, where~$X$ and $S$ are complex manifolds.

The functor $L\mto (E_L,\nabla)=(\sho_{\XS}\otimes_{\pOS}L,\rd_{\XS/S}\otimes\id)$ induces an equivalence between the category of coherent $S$-locally constant sheaves of $\pOS$-modules and the category of coherent $\sho_{\XS}$-modules $E$ equipped with an integrable relative connection $\nabla:E\to\Omega^1_{\XS/S}\otimes_{\sho_{\XS}}E$. A~quasi-inverse is given by $(E,\nabla)\mto E^\nabla=\ker\nabla$. The monodromy representation attached to the coherent $S$-locally constant sheaf $E^{\nabla}$ is also called the monodromy representation of $\nabla$ on~$E$. Let~us emphasize a direct consequence:

\begin{corollary}\label{cor:Enabla}
Let $(E,\nabla)$ be a coherent $\sho_{\XS}$-module equipped with an integrable relative connection. Then the natural $\sho_{\XS}$-linear morphism
\[
\sho_{\XS}\otimes_{\pOS}E^\nabla\to E
\]
is an isomorphism compatible with the integrable connections $\rd_{\XS/S}\otimes\id$ and~$\nabla$.
\end{corollary}

\begin{proposition}\label{prop:behaviourpi}
Notation as in Lemma \ref{lem:behaviourpi}.
\begin{enumerate}
\item\label{prop:behaviourpi1}
Let $L$ be a coherent $S$-locally constant sheaf. If $\pi$ is proper, then $\pi_*L$ is $S'$\nobreakdash-coherent and there exists a natural morphism\vspace*{-3pt}
$
E_{\pi_*L}\to\pi_*E_L.
$
\item\label{prop:behaviourpi2}
Let $L'$ be a coherent $S'$-locally constant sheaf. Then $E_{\pi^*L'}\simeq\pi^*(E_{L'})$.
\end{enumerate}
\end{proposition}

\begin{proof}\mbox{}
\begin{enumerate}
\item
Since $\pi$ is proper, $\pi_*L$ is a coherent $S'$-locally constant sheaf on $\XS'$ (Lemma \ref{lem:behaviourpi}\eqref{lem:behaviourpi1} and Grauert's theorem). The natural morphism $\sho_{\XS'}\to\pi_*\sho_{\XS}$ induces a composed morphism\vspace*{-3pt}
\[
\sho_{\XS'}\otimes_{\pOSp}\pi_* L\to\pi_*\sho_{\XS}\otimes_{\pOSp}\pi_*L\to\pi_*(\sho_{\XS}\otimes_{\pOS}L),
\]
which is clearly compatible with the relative differential $\rd\otimes\id$. We note that if~$\pi$ is surjective with connected fibers, the first morphism is an isomorphism since $\sho_{\XS'}\to\pi_*\sho_{\XS}$ is then an isomorphism. This is the case if for example $S'$ is a complex manifold and $\pi$ is a proper modification of $S'$.

\item
The second point is straightforward.\qedhere
\end{enumerate}
\end{proof}

\subsection{The Deligne extension of an \texorpdfstring{$S$}{S}-locally constant sheaf}\label{subsec:Deligneext}
We recall Theorem~2.6 and extend Corollary 2.8 of \cite{MFCS2} to the case where $\dim S>1$.

\begin{notation}\label{nota:Dtype}
Let $Y$ be a hypersurface in $X$ (assumed to be connected) and let us denote the inclusion by $j:X^*:=X\moins Y\hto X$. Let $L$ be a coherent $S$-locally constant sheaf on $\XsS$. Let $(E_L,\nabla)=(\sho_{\XsS}\otimes_{\pOS}\nobreak L,\rd_{\XsS/S}\otimes\id)$ be the associated cohe\-rent $\sho_{\XsS}$-module with flat relative connection $\nabla=\rd_{\XsS/S}\otimes\id$, so that $L=E_L^\nabla$. We~simply set $E=E_L$ when the context is clear, that we consider as a left $\DXsS$-module. We sometimes call $L$ a coherent $S$-local system.
\end{notation}

\begin{lemma}\label{lem:exactjstar}
The functor $j_*E_{\scriptscriptstyle\bullet}:L\mto j_*E_L$ with values in $\Mod(\DXS)$ is exact.
\end{lemma}

\begin{proof}
Any $(x_o,s_o)\!\in\!Y\times S$ has a fundamental system of open neighborhood $U\times\nobreak U(s_o)$ such that $(U\moins D)\times U(s_o)$ is Stein. Since $E_L$ is $\sho_{\XsS}$-coherent, the result follows.
\end{proof}

We assume from now on that $Y=D$ is a divisor with normal crossings in $X$. Let $\varpi:\wt X\to X$ denote the real blowing up of $X$ along the components of~$D$. We~denote by $\wtj:X^*\hto\wt X$ the inclusion, so that $j=\varpi\circ\wtj$. Let $x_o\in D$, $\wt x_o\in\varpi^{-1}(x_o)$ and let $s_o\in S$. Choose local coordinates $(x_1,\dots,x_n)$ at $x_o$ such that $D=\{x_1\cdots x_\ell=0\}$ and consider the associated polar coordinates $(\rhog,\thetag,\boldsymbol{x}'):=(\rho_1,\theta_1,\dots,\rho_\ell,\theta_\ell,x_{\ell+1},\dots,x_n)$ so that $\wt x_o$ has coordinates $\rhog_o=0$, $\thetag_o$, $\boldsymbol{x}'_o=0$. We~also denote by $\varpi$ the induced map $\wXS\to\XS$ and by $\wt p:\wXS\to S$ the projection.

\begin{definition}
The subsheaf $\sha^\rmod_{\wXS}$ of $\wtj_*\sho_{\XsS}$ of holomorphic functions having \emph{moderate growth} along $\varpi^{-1}(D)$ is defined by the following two requirements:
\begin{itemize}
\item
$\wtj^{-1}\sha^\rmod_{\wXS}=\sho_{\XsS}$.
\item
For any $x_o\in D$, $\wt x_o\in\varpi^{-1}(x_o)$ and $s_o\in S$, a germ $\wt h\in(\wtj_*\sho_{\XsS})_{(\wt x_o,s_o)}$ is said to have \emph{moderate growth}, \ie to belong to $\sha^\rmod_{\wXS,(\wt x_o,s_o)}$, if there exist open sets
\begin{equation}\label{eq:Ueps}
\wt U_\epsilon:=\{\Vert\rhog\Vert<\epsilon,\Vert\boldsymbol{x}'\Vert<\epsilon,\Vert\thetag-\thetag_o\Vert<\epsilon\}\quad\text{and}\quad U(s_o)\ni s_o
\end{equation}
($\epsilon$ small enough) such that, setting $U^*_\epsilon:=\wt U_\epsilon\moins\{\rho_1\cdots\rho_\ell=0\}$, $\wt h$ is defined on $U^*_\epsilon\times U(s_o)$ and $|\wt h|$ is bounded on this open set by $C\Vert\rhog\Vert^{-N}$, for some $C,N>0$.
\end{itemize}
\end{definition}

We~recall the following properties of the sheaf $\sha^\rmod_{\wXS}$:
\begin{enumerate}
\item\label{enum:Amod1}
$\sha^\rmod_{\wXS}$ is a $\varpi^{-1}\DXS(*D)$-module (\cf\cite[\S1.1.6]{Sabbah00}) which is $\varpi^{-1}\sho_{\XS}$-flat (\cf \cite[Th.\,4.6.1]{Mochizuki10}),
\item\label{enum:Amod1b}
for any coherent $\sho_{\XS}$-module $M$, the natural morphism
\[
\sha^\rmod_{\wXS}\otimes_{\varpi^{-1}\sho_{\XS}}\varpi^{-1}M\to\wtj_*j^{-1}M
\]
is injective (\cf\loccit),
\item\label{enum:Amod2}
$R\varpi_*\sha^\rmod_{\wXS}=\varpi_*\sha^\rmod_{\wXS}=\sho_{\XS}(*D)$ (\cf\cite[Cor.\,II.1.1.18]{Sabbah00}).
\end{enumerate}
Let us already notice for later use that these properties imply, for any $\sho_S$-module $G$,
\begin{equation}\label{eq:varpiAmod}
\varpi_*(\sha^\rmod_{\wXS}\otimes_{\wpOS}\wt p^{-1}G)\simeq\sho_{\XS}(*D)\otimes_{\pOS}p^{-1}G.
\end{equation}
Indeed, we have
\begin{align*}
R\varpi_*(\sha^\rmod_{\wXS}\otimes_{\wpOS}\wt p^{-1}G)
&\simeq R\varpi_*(\sha^\rmod_{\wXS}\otimes^L_{\wpOS}\wt p^{-1}G)\quad\text{by \eqref{enum:Amod1}}\\
&\simeq R\varpi_*(\sha^\rmod_{\wXS})\otimes^L_{\pOS} p^{-1}G\\
&\simeq\sho_{\XS}(*D)\otimes^L_{\pOS}p^{-1}G\quad\text{by \eqref{enum:Amod2}}\\
&\simeq\sho_{\XS}(*D)\otimes_{\pOS}p^{-1}G,
\end{align*}
where the latter isomorphism follows from flatness of $\sho_{\XS}(*D)$ over $\pOS$. Furthermore, \eqref{enum:Amod1b}~implies that, for any coherent $\sho_S$-module $G$, the natural morphism
\[
\sha^\rmod_{\wXS}\otimes_{\tilde{p}^{-1}\sho_S}\tilde{p}^{-1}G\to\wtj_*(\sho_{\XsS}\otimes_{\pOS}p^{-1}G)
\]
is injective.

Note that, $U^*_\epsilon$ being contractible, we have $L_{|U^*_\epsilon\times U(s_o)}\simeq p^{-1}_{U^*_\epsilon}G_{|U(s_o)}$ (\cf\cite[Prop.\,A.12]{MFCS2}). We thus have an identification
\begin{multline}\label{eq:jtilde}
(\wtj_*E)_{|\wt U_\epsilon\times U(s_o)}\simeq\wtj_*\Bigl(\sho_{U^*_\epsilon\times U(s_o)}\otimes_{p^{-1}\sho_{U(s_o)}} p^{-1}_{U^*_\epsilon}G_{|U(s_o)},\rd_{U^*_\epsilon\times U(s_o)/U(s_o)}\otimes \Id\Bigr)\\
\overset\sim\longleftarrow\Bigl((\wtj_*\sho_{U^*_\epsilon\times U(s_o)})\otimes_{\wt p^{-1}\sho_{U(s_o)}}\wt p^{-1}_{\wt U_\epsilon}G_{|U(s_o)},\rd_{\wt U_\epsilon\times U(s_o)/U(s_o)}\otimes \Id\Bigr).
\end{multline}
Indeed, for a polysector $\wt V_\epsilon\subset\wt U_\epsilon$ and $V\subset U(s_o)$, the natural morphism
\[
\Gamma(V_\epsilon^*\times V;\sho_{\XS})\otimes_{\sho(V)}\Gamma(\wt V_\epsilon\times V;\tilde{p}^{-1}G)\to\Gamma(V_\epsilon^*\times V;\sho_{\XS})\otimes_{\sho(V)}\Gamma(V_\epsilon^*\times V;p^{-1}G)
\]
is an isomorphism, since the adjunction morphism $p_{W*}p_W^{-1}G\to G$ is an isomorphism for both $W=V_\epsilon^*$ and $W=\wt V_\epsilon$ (both being connected, \cf\eg\cite[Prop.\,A.1(2)]{MFCS2}). The assertion is obtained by passing from the pre-sheaf isomorphism to a sheaf isomorphism.

\begin{definition}[The Deligne extension]\label{def:modgrowth}
Let $x_o\in X$, $\wt x_o\in\varpi^{-1}(x_o)$ and $s_o\in S$.
\begin{enumerate}
\item\label{def:modgrowth1}
A germ $\wt v$ of $(\wtj_*E_L)_{(\wt x_o,s_o)}$ is said to have \emph{moderate growth} if for some open set $\wt U_\epsilon\times U(s_o)$ as above, and for some (equivalently, any) identification $L_{|U^*_\epsilon\times U(s_o)}\simeq p^{-1}_{U^*_\epsilon}G_{|U(s_o)}$, $\wt v$ belongs to the image of the \emph{injective} morphism
\[
\Gamma\bigl(\wt U_\epsilon\times U(s_o);\sha^\rmod_{\wXS}\otimes_{\wpOS}\wt p^{-1}G\bigr)\to\Gamma(\wt U_\epsilon\times U(s_o);\wtj_*E_L).
\]

\item\label{def:modgrowth2}
A germ $v$ of $(j_*E_L)_{(x_o,s_o)}$ is said to have \emph{moderate growth} if for each~$\wt x_o$ in $\varpi^{-1}(x_o)$, the corresponding germ in $(\wtj_*E_L)_{(\wt x_o,s_o)}$ has moderate growth. In particular, this holds for any local section of $E_L$ at $(x_o,s_o)$ if $x_o\notin D$.

\item\label{def:modgrowth3}
The subsheaf of $j_*E_L$ consisting of local sections whose germs have moderate growth is denoted by $\wt E_L$. It satisfies $j^*\wt E_L=E_L$. It is called \emph{the Deligne extension} of $E_L$.
\end{enumerate}
\end{definition}

\begin{remark}
\label{RTMF}
By definition, with the previous notations, $v$ has moderate growth if and only if, on any such polysector $U_\epsilon\times U(s_o)$, for any isomorphism $L|_{U_\epsilon^*\times U(s_o)}\simeq p_{U_\epsilon^*}^{-1}G_{|U(s_o)}$, for any family of local generators $(g_j)$ of $G$ on $U(s_o)$, $v$ can be written as $\sum_jv_j\otimes g_j$ with $v_j$ being holomorphic functions on the corresponding polysector in $\XsS$ with moderate growth with respect to $D$.
\end{remark}

Since $\sha^\rmod_{\wXS}$ is stable by derivations with respect to $X$, $\wt E_L$ is a $\DXS(*D)$-submodule of $j_*E_L$.

\begin{proposition}[First properties of the Deligne extension $\wt E_L$]\label{prop:firsttildeE}\mbox{}
Let $L$ be a coherent $S$-locally constant sheaf on $\XsS$ and let $(E_L,\nabla)=(\sho_{\XsS}\otimes_{\pOS} L,\rd_{\XsS/S}\otimes\id)$ be the associated $\sho_{\XsS}$-module with flat relative connection.\begin{enumerate}
\item\label{prop:firsttildeE1}
The assignment $L\mto(\wt E_L,\nabla)$ with values in $\Mod(\DXS)$ is functorial.

\item\label{prop:firsttildeE2}
Let $\pi:S\to S'$ be a proper holomorphic map. Then the natural morphism $j_*(E_{\pi_*L},\nabla)\to \pi_*j_*(E_L,\nabla)$ sends the subsheaf $\wt E_{\pi_*L}$ to $\pi_*\wt E_L$.
\item\label{prop:firsttildeE3}
Let $\pi:S'\to S$ be a holomorphic map. Then the natural morphism $\pi^*j_*(E_L,\nabla)\to j_*(E_{\pi^*L},\nabla)$ sends isomorphically $\pi^*(\wt E_L,\nabla)$ to $(\wt E_{\pi^*L},\nabla)$.
\end{enumerate}
\end{proposition}

\begin{proof}
We only prove \eqref{prop:firsttildeE1} and \eqref{prop:firsttildeE2}, and \eqref{prop:firsttildeE3} will be a consequence of Theorem \ref{th:Deligneext} below.
\begin{enumerate}
\item
A morphism of $S$-locally constant sheaves $\varphi:L\to L'$ defines a morphism $\varphi:E_L\to E_{L'}$ compatible with~$\nabla$, hence $j_*E_L\to j_*E_{L'}$ compatible with~$\nabla$, and we only need to check that it sends $\wt E_L$ to $\wt E_{L'}$. This is straightforward from the definition.

\item
From the commutative diagram
\begin{equation}\label{eq:commutpij}
\begin{array}{c}
\xymatrix{
\XsS\ar@{^{ (}->}[r]^-j\ar[d]_\pi&\XS\ar[d]^\pi\\
\XsS'\ar@{^{ (}->}[r]^-j&\XS'
}
\end{array}
\end{equation}
together with the natural morphism of Proposition \ref{lem:behaviourpi}\eqref{lem:behaviourpi1}, we obtain the morphism $j_*(E_{\pi_*L},\nabla)\to \pi_*j_*(E_L,\nabla)$ and similarly $\wtj_*(E_{\pi_*L},\nabla)\to \pi_*\wtj_*(E_L,\nabla)$. Let us consider an open subset $\wt U_\epsilon\subset\wt X$ as in \eqref{eq:Ueps}. Then $U_\epsilon^*$ is contractible and for any $s'_o\in S'$, $L_{|U_\epsilon^*\times\pi^{-1}(U(s'_o))}\simeq p_{U_\epsilon^*}^{-1}G_{|\pi^{-1}(U(s'_o))}$. In order to simplify the notation, we set $S'=U(s'_o)$ and $S=\pi^{-1}(U(s'_o))$. We also set $p=p_{\wt U_\epsilon}$, so that $p_{U_\epsilon^*}=p\circ\wtj$. According to \eqref{eq:jtilde}, we are led to proving that the natural morphism
\[
(\wtj_*\sho_{U_\epsilon^*\times S'})\otimes_{\wpOSp}\wt p^{\prime-1}\pi_*G\to \pi_*\bigl((\wtj_*\sho_{U_\epsilon^*\times S})\otimes_{\pOS}p^{-1}G\bigr)
\]
sends $\sha^\rmod_{\wt U_\epsilon\times S'}\otimes_{\wpOSp}\wt p^{\prime-1}\pi_*G$ to $\pi_*(\sha^\rmod_{\wt U_\epsilon\times S}\otimes_{\pOS}p^{-1}G)$.
Following the proof of Proposition \ref{prop:behaviourpi}\eqref{prop:behaviourpi1}, we are led to proving that the natural morphism $\wtj_*\sho_{U_\epsilon^*\times S'}\to\pi_*(\wtj_*\sho_{U_\epsilon^*\times S})$ sends $\sha^\rmod_{\wt U_\epsilon\times S'}$ to $\pi_*\sha^\rmod_{\wt U_\epsilon\times S}$. This follows from the definition of moderate growth, owing to the properness of $\pi$.\qedhere
\end{enumerate}
\end{proof}

\subsection{Regular holonomicity of the Deligne extension of an \texorpdfstring{$S$}{S}-locally constant sheaf}
We continue to refer to Notation \ref{nota:Dtype} and assume that $Y=D$ has normal crossings.

\begin{theorem}\label{th:Deligneext}
Assume that $L$ is a coherent $S$-locally constant sheaf on $\XsS$. Then
\begin{enumerate}
\item\label{th:Deligneext1}
the subsheaf $\wt E_L$ of $j_*E_L$ is $\sho_{\XS}(*D)$-coherent;
\item\label{th:Deligneext2}
the functor $L\mto \wt E_L$ from the category of coherent $S$-locally constant sheaves on $\XsS$ to that of $\DXS$-modules is fully faithful;
\item\label{th:Deligneext3}
as a $\DXS$-module, $(\wt E_L,\nabla)$ is regular holonomic.
\end{enumerate}
\end{theorem}

We will make use of the following flattening result for a coherent $\sho_S$-module (here, we only need a local version, but the corresponding global one also holds true).

\begin{proposition}\label{prop:Rossi}
Near each $s_o\!\in\!S$, there exists a projective modification \hbox{$\pi:S'\to S$} with~$S'$ smooth such that the torsion-free quotient of $\pi^*G$ is $\sho_{S'}$-locally free.
\end{proposition}

\begin{proof}
We first apply the flattening theorem \cite[Th.\,3.5]{Rossi68} to the coherent sheaf $G$. There exists thus a projective modification $\pi'':S''\to S$ such that $\pi^{\prime\prime*}G$, when quotiented by its $\sho_{S''}$-torsion, is $\sho_{S''}$-locally free. We then apply resolution of singularities $\pi':S'\to S''$ of $S''$ in the neighborhood of the projective subset $\pi^{\prime\prime-1}(s_o)$ (\cf\cite[\S7, Main Th.\,I']{Hironaka64}) and denote by $\pi$ the morphism $\pi''\circ\pi'$, which answers the question.
\end{proof}

\begin{proof}[Proof of Theorem \ref{th:Deligneext}\eqref{th:Deligneext1}]
We recall the proof of \cite[Th.\,2.6]{MFCS2} for the $\sho_{\XS}(*D)$-coher\-ence. The problem is local on $\XS$. We~thus assume that $\XS$ is a small neighborhood of $(x_o,s_o)$ as above. In such a neighborhood, giving the local system is equivalent to giving $T_1,\dots,T_\ell\in\Aut(G)$ which pairwise commute. Let $U(s_o)$ be an open neighborhood of $s_o$ isomorphic to an open polydisc. The formula \cite[(2.11)]{Wang08} defining a logarithm of $T_i$ can be used to show that there exist $A_1(s),\dots,A_\ell(s)\in\End_{\sho_S}(G_{|U(s_o)})$ which pairwise commute, such that $T_i=\exp(-2\pi i A_i(s))$ for each $i$ on $U(s_o)$ (\cf \cite[Cor.\,2.3.10\,\&\,Chap.\,3]{Wang08}). Set $\wh E_G:=\sho_{\XS}(*D)\otimes_{\pOS}p^{-1}G$, equipped with the connection $\wh\nabla$ such that $\wh\nabla_{x_i\partial_{x_i}}$ acts on $1\otimes_{\pOS}p^{-1}G$ by $\id\otimes A_i(s)$ if~$i=1,\dots,\ell$, and zero otherwise. Then the monodromy representation of $\wh\nabla$ on~$\wh E_{G|\XsS}$ is given by $T_1,\dots,T_\ell$, hence an isomorphism $L\simeq(\wh E_{G|\XsS})^{\wh\nabla}$, from which one deduces, according to Corollary \ref{cor:Enabla}, an isomorphism
\begin{equation}\label{lftmf}
\begin{aligned}
(E_L,\nabla)&=(\sho_{\XsS}\otimes_{\pOS}L,\rd_{\XsS/S}\otimes\id)\\
&\simeq(\sho_{\XsS}\otimes_{\pOS}(\wh E_{G|\XsS})^{\wh\nabla},\rd_{\XsS/S}\otimes\id)\simeq(\wh E_{G|\XsS}, \wh\nabla).
\end{aligned}
\end{equation}

It~follows that $\wt E_L\!\simeq\!\wt{\wh E\,}_{\!G|\XsS}$ and we are thus reduced to proving that \hbox{$\wt{\wh E\,}_{\!G|\XsS}\!=\!\wh E_G$} (we~have trivialized the locally constant sheaf but the connection is not trivial anymore).

\begin{remark}\label{rem:pGOXsS}
Let $p_{X^*}^{-1}G$ be the constant $S$-local system on $\XsS$. Thus, locally on $\XS$, there exist $\sho_{\XsS}$-linear isomorphisms
$
E_L\simeq\wh E_{G|\XsS}\simeq E_{p_{X^*}^{-1}G},
$
but in general we do not have $(E_L,\nabla)\simeq(E_{p_{X^*}^{-1}G},\wh\nabla)$, \ie this isomorphism is not $\DXsS$-linear.
\end{remark}

Let us prove the inclusion $\wt{\wh E\,}_{\!G|\XsS}\subset\wh E_G$. Let $v\in (j_*\wh E_G)_{(x_o,s_o)}$ be locally defined on $U\times\nobreak U(s_o)$. On any polysector $U^*_\epsilon\times U(s_o)$, we can choose logarithms of $x_i$ \hbox{($i=1,\dots,\ell$)} and the automorphism $\prod_{i=1}^\ell x_i^{A_i(s)}$ of $\wh E_{G|\XsS}=\sho_{U^*_\epsilon\times U(s_o)}\otimes_{p^{-1}\sho_{U(s_o)}}p^{-1}G$ is well-defined by setting $x_i^{A_i(s)}=\sum_k[(\log x_i)^k/k!]\otimes A_i(s)^k$. By definition, $v$ has moderate growth, \ie is a local section of $\wt{\wh E\,}_{\!G|\XsS}$, if and only if, on any such polysector $\wt U_\epsilon\times U(s_o)$, the section $w:=(\prod_{i=1}^\ell x_i^{-A_i(s)})\cdot v$ is a section of $\sha^\rmod_{\wXS}\otimes_{\wpOS}\wt p^{-1}G$.

Let $(g_j)$ be a family of local generators of $G$ on $U(s_o)$. The entries of the matrices of $\prod_{i=1}^\ell x_i^{A_i(s)}$ and of $\prod_{i=1}^\ell x_i^{-A_i(s)}$ with respect to this family have moderate growth. If, on such a polysector, $w$ writes $\sum_jw_j\otimes g_j$, where $w_j$ are sections of $\sha^\rmod_{\wXS}$, then $v=\prod_{i=1}^\ell x_i^{A_i(s)}\cdot w$ writes similarly $\sum_jv_j\otimes g_j$, where $v_j$ are sections of $\sha^\rmod_{\wXS}$. Conversely, if $v$ has the previous expression, then so does $w=\prod_{i=1}^\ell x_i^{-A_i(s)}\cdot v$. We conclude that, for any such polysector, the moderate growth condition on $v$ is equivalent to $v\in\Gamma(\wt U_\epsilon\times U(s_o);\sha^\rmod_{\wXS}\otimes_{\wpOS}\wt p^{-1}G)$. In other words, if $v$ has moderate growth we have by \eqref{eq:varpiAmod} that
\begin{align*}
v\in\Gamma(U\times U(s_o);\varpi_*\sha^\rmod_{\wXS}\otimes_{\wpOS}\wt p^{-1}G)
&=\Gamma(U\times U(s_o);\sho_{\XS}(*D)\otimes_{\pOS} p^{-1}G)\\
&=\Gamma(U\times U(s_o);\wh E_G),
\end{align*}
hence $\wt{\wh E\,}_{\!G|\XsS}\subset\wh E_G$. The reverse inclusion also follows from the moderate growth of the entries of the matrices of $\prod_{i=1}^\ell x_i^{A_i(s)}$ and of $\prod_{i=1}^\ell x_i^{-A_i(s)}$.
\end{proof}

\begin{proof}[Proof of Theorem \ref{th:Deligneext}\eqref{th:Deligneext2}]
We already know, by the Riemann-Hilbert correspondence of \cite[Th.\,2.23]{De}, that the functor $L\mto E_L$ is fully faithful, and it is clearly exact, so it remains to prove that $E_L\mto\wt E_L$ is so. Faithfulness is clear. Let us check fullness. Let $\varphi:\wt E_L\to\wt E_{L'}$ be a morphism and set $\wt\varphi=\wt{j^*\varphi}:\wt E_L\to\wt E_{L'}$. The kernel and cokernel of $\varphi-\wt\varphi$ are coherent $\sho_{\XS}(*D)$-modules by \ref{th:Deligneext}\eqref{th:Deligneext1}, and are zero when restricted to $\XsS$. Therefore they are zero, and $\varphi=\wt\varphi$, proving the assertion.
\end{proof}

\begin{proof}[Proof of Theorem \ref{th:Deligneext}\eqref{th:Deligneext3}]
The question is local at a point $(x_o,s_o)\in\XS$. By the proof of \ref{th:Deligneext}\eqref{th:Deligneext1}, we can work with
\[
(\wh E_G,\wh\nabla)=(\sho_{\XS}(*D)\otimes_{\pOS}p^{-1}G,\wh\nabla)
\]
(with $\wh\nabla=\rd\otimes\Id+\sum_i\rd x_i/x_i\otimes\nobreak A_i$ and $A_i\in\End_{\sho_S}(G)$). Our method is to reduce to the case where $G$ is locally free of rank one and then prove~\ref{th:Deligneext}\eqref{th:Deligneext3} for it.

We argue by induction on the lexicographically ordered pair $(\dim S,\rk G)$, where $\rk G$ is the rank of $G$ at a general point of $S$. The case $\dim S=0$ and $\rk G$ arbitrary is well-\allowbreak known. We~thus assume that $d=\dim S\geq1$ and \hbox{$r=\rk G\geq0$}, and that \ref{th:Deligneext}\eqref{th:Deligneext3} holds for $(\wh E_{G'},\wh\nabla)$ for any $\pOSp$-module $G'$ with $(\dim S',\rk G')<(d,r)$.

\subsubsection*{Step 1: The case $(d,0)$}
If $\rk G=0$, $G$ is an $\sho_S$-torsion module. Since the question is local, we can assume that the support of $G$ is contained in a hypersurface $T\subset S$ locally presented as in the local setting~\ref{set:local}. Then~$G$ is $q^{-1}\sho_{S'}$\nobreakdash-coherent and is endowed with an endomorphism~$t$ (\ie multiplication by~$t$), so that the natural $\sho_S$\nobreakdash-linear morphism $\sho_S\otimes_{q^{-1}\sho_{S'}[t]}G\to G$ is an isomorphism. Furthermore, the endo\-mor\-phisms~$A_i$ are $q^{-1}\sho_{S'}$-linear (via the natural inclusion $q^{-1}\sho_{S'}\hto\sho_S$. Let us set $(\wh E'_G,\nabla)=(\wh E_G,\wh\nabla)$ when regarded as a $q^{-1}\DXSp$-module. By the induction hypothesis, it is regular holonomic. By Corollary \ref{cor:TS} we conclude that $(\wh E_G,\wh\nabla)$ is regular holonomic, so~that~\ref{th:Deligneext}\eqref{th:Deligneext3} holds for $(\wh E_G,\wh\nabla)$.

\subsubsection*{Step 2: Reduction after proper surjective generically finite base change}

Let~\hbox{$\pi\,{:}\,S'\!\to\!S$} be a proper surjective generically finite morphism of complex manifolds and let $p':\XS'\to S'$ denote the projection.\textit{ Let us assume that \ref{th:Deligneext}\eqref{th:Deligneext3} holds for $(\sho_{\XS'}(*D)\otimes_{p^{\prime-1}\sho_{S'}} p^{\prime-1}\pi^*G,\pi^*\wh\nabla)$.} The purpose of Step~2 is to show that, under such an assumption, \ref{th:Deligneext}\eqref{th:Deligneext3} holds for~$(\wh E_G,\wh\nabla)$.

From now on, we assume $\dim S=d$ and $\rk G\geq1$. Let $G'$ be the $\sho_S$-torsion subsheaf of $G$. We note that the endomorphisms $A_i$ of $G$ preserve $G'$, so the exact sequence $0\to G'\to G\to G''\to0$ with $G''$ being torsion-free gives rise to an exact sequence
\[
0\to(\wh E_{G'},\wh\nabla)\to(\wh E_G,\wh\nabla)\to(\wh E_{G''},\wh\nabla)\to0.
\]
Since \ref{th:Deligneext}\eqref{th:Deligneext3} holds for $(\wh E_{G'},\wh\nabla)$ by Step~1, it~is enough to prove \ref{th:Deligneext}\eqref{th:Deligneext3} for $(\wh E_{G''},\wh\nabla)$. In other words, we may assume that $G$ is torsion-free.

By assumption, there exists a closed analytic subset $T$ of codimension~$\geq1$ in $S$ such that, setting $T'=\pi^{-1}(T)$, the morphism $\pi:S'\moins T'\to S\moins T$ is finite étale. We~have $\dim S'=\dim S$ and $\dim T'<\dim S$.

Remark that we have, for any holomorphic map $\pi:S'\to S$, and denoting for clarity by $\wt\pi$ the map $\id\times\pi:\XS'\to\XS$, the following list of canonical isomorphisms:
\begin{equation}\label{lftmf2}
\begin{aligned}
L\wt\pi^*(\sho&_{\XS}(*D)\otimes_{\pOS}p^{-1}G)\\
&=\sho_{\XS'}(*D)\otimes^L_{\wt\pi^{-1}\sho_{\XS}(*D)}\wt\pi^{-1}(\sho_{\XS}(*D)\otimes_{\pOS}p^{-1}G)\\
&\overset{(*)}=\sho_{\XS'}(*D)\otimes^L_{\wt\pi^{-1}\pOS}\wt\pi^{-1}p^{-1}G\\
&=\sho_{\XS'}(*D)\otimes^L_{\pOSp}(\pOSp\otimes^L_{\wt\pi^{-1}\pOS}\wt\pi^{-1}p^{-1}G)\\
&=\sho_{\XS'}(*D)\otimes^L_{\pOSp}p^{\prime-1}L\pi^*G=\sho_{\XS'}(*D)\otimes_{\pOSp}p^{\prime-1}L\pi^*G,
\end{aligned}
\end{equation}
which reads
\begin{equation}\label{lftmf3}
L\pi^*\wh E_G\simeq\sho_{\XS'}(*D)\otimes_{p^{\prime-1}\sho_{S'}} p^{\prime-1}L\pi^*G.
\end{equation}
We conclude that, for each $j$, we have $L^j\pi^*\wh\nabla_{x_k\partial_{x_k}}=\rd\otimes\id+\id\otimes L^j\pi^*A_k$.

Furthermore, under the assumption on $\pi$ for this step, according to the projection formula for $R\wt\pi_*$ applied to $(*)$, we have
\[
R\pi_*(L\pi^*\wh E_G)\simeq R\pi_*\sho_{\XS'}(*D)\otimes^L_{\pOS}p^{-1}G.
\]

By induction, if $j\neq0$, \ref{th:Deligneext}\eqref{th:Deligneext3} holds for $(\sho_{\XS'}(*D)\otimes_{p^{\prime-1}\sho_{S'}} p^{\prime-1}L^j\pi^*G,L^j\pi^*\wh\nabla)$, according to the argument given in Step~1, since $L^j\pi^*G$ is supported on $T'$, hence is a torsion module. Thus ~\ref{th:Deligneext}\eqref{th:Deligneext3} holds for $L\pi^*\wh E_G$ (\ie for each of its cohomology modules $L^j\pi^*\wh E_G$) since for $j=0$ it is the initial assumption.

Since $\wh E_G$ is $\sho$-quasi-coherent (\cf Remark \ref{rem:localiz}\eqref{rem:localiz1}), one deduces from Lemma \ref{lem:qcoh} that each $L^j\pi^*\wh E_G$ is $\pi$-good. Then, by Theorem \ref{th:Regdirect}, \ref{th:Deligneext}\eqref{th:Deligneext3}
~holds for $R\pi_*(L\pi^*\wh E_G)$, hence for $\shh^0(R\pi_*(L\pi^*\wh E_G))$.

The natural morphism $\sho_{\XS}(*D)\to R\pi_*\sho_{\XS'}(*D)$ yields a morphism
\[
\wh E_G\to \shh^0(R\pi_*(L\pi^*\wh E_G)).
\]
Both modules are $\sho_{\XS}(*D)$-coherent, hence so is the kernel of this morphism.

On $S\moins T$, we claim that this morphism is injective: indeed, since $\pi:S'\moins T'\to S\moins T$ is finite étale, the trace morphism $\mathrm{tr}_\pi:\pi_*\sho_{S'\moins T'}\to\sho_{S\moins T}$ defined by $\mathrm{tr}_\pi(\varphi)(s)=(1/\deg\pi)\sum_{s'\in\pi^{-1}(s)}\varphi(s')$ satisfies $\mathrm{tr}_\pi\circ\iota=\id$, if $\iota$ denotes the natural morphism $\sho_{S\moins T}\to\pi_*\sho_{S'\moins T'}$; hence $\wh E_{G|S\moins T}$ is locally a direct summand of the right-hand side. As~$G$, hence~$\wh E_G$, is assumed to be $\sho_S$\nobreakdash-torsion-free, the kernel is $\sho_S$\nobreakdash-torsion-free, $\sho_{\XS}(*D)$-coherent and $S$-supported on $T$. It is thus zero and this morphism is injective.

Furthermore, $\wh E_G$ is $\DXS$-coherent: indeed, as \ref{th:Deligneext}\eqref{th:Deligneext3} holds for $\shh^0(R\pi_*(L\pi^*\wh E_G))$ and $\wh E_G$ is $\sho$-good, it follows that $\wh E_G$ is $\DXS$-coherent, and then regular holonomic by Proposition \ref{Prop:2}\eqref{Prop:22b}, so that \ref{th:Deligneext}\eqref{th:Deligneext3} holds for~$\wh E_G$.

\subsubsection*{Step 3: Reduction to the case where $G$ is $\sho_S$-locally free}

We choose $\pi$ as in Proposition~\ref{prop:Rossi} and by Step 2 we can assume from the start that the torsion-free quotient~$G''$ of~$G$ is $\sho_S$-locally free. By Step 1, it is enough to prove \ref{th:Deligneext}\eqref{th:Deligneext3} for $(\wh E_{G''},\wh\nabla)$, \ie we~can assume (and we do assume from now on) that $G$ is $\sho_S$-locally free.\enlargethispage{\baselineskip}%

\subsubsection*{Step 4: The case where $G$ is $\sho_S$-locally free}
We still work locally on $S$ and we assume that~$G$ is $\sho_S$-free of rank $r\geq1$. By the induction hypothesis, \ref{th:Deligneext}\eqref{th:Deligneext3} holds for any $(\wh E_{G'},\wh\nabla)$ with $\rk G'<r$. Locally let us fix an $\sho_S$-basis of~$G$ and let $\sfA_1(s)$ denote the matrix of~$A_1$ in this basis. Let $\Sigma\subset S\times\CC$ be the zero locus of $P:=\det(\alpha_1\id-\sfA_1)$ and let $\sigma:\Sigma\to S$ denote the projection. Since $P$ is a Weierstrass polynomial with respect to the variable~$\alpha_1$ (considered as a coordinate on $\CC$) with coefficients in $\sho_S$, $\sigma$ is a finite morphism of degree $\deg\sigma$ and $\Sigma$ is defined by the corresponding reduced Weierstrass polynomial. We note the following two properties related to~$\sigma$ and similarly to $\id_X\times\sigma$:
\begin{enumerate}
\step\label{step:sigma1}\upshape
The sheaf $\sigma_*\sho_\Sigma$ is $\sho_S$-free of degree $\deg\sigma$ (with basis $1,\alpha_1,\dots,\alpha_1^{\deg\sigma-1}$).

\step\label{step:sigma2}\upshape
There exist dense analytic open subsets $\Sigma^\circ\!\subset\!\Sigma$ and $S^\circ\!\subset\! S$ such that \hbox{$\sigma:\Sigma^\circ\to S^\circ$} is a finite covering of degree $\deg\sigma$. The corresponding trace morphism $\mathrm{tr}_\sigma:\sigma_*\sho_{\Sigma^\circ}\to\sho_{S^\circ}$ has been defined above. If $\varphi$ is a section of $\sigma_*\sho_{\Sigma}$, its trace on~$\Sigma^\circ$ is bounded on $S^\circ$, hence extends holomorphically to $S$, so that $\mathrm{tr}_\sigma$ extends as a morphism \hbox{$\sigma_*\sho_{\Sigma}\to\sho_S$}. If $\iota$ denotes the natural morphism $\sho_S\to\sigma_*\sho_\Sigma$, we clearly have $\mathrm{tr}_\sigma\circ\iota=\id$, making $\sho_S$ a direct summand of $\sigma_*\sho_{\Sigma}$.
\end{enumerate}

Let $\pi:S'\to\Sigma$ be a resolution of singularities of $\Sigma$, so that the natural composed map $(\sigma\circ\pi):S'\to S$ is projective and generically finite. Let us set $G'=\ker(\alpha_1\id-\sfA_1\circ\sigma\circ\pi)\subset(\sigma\circ\pi)^*G$, where we regard $\alpha_1$ as a function $S'\to\CC$. Noting that $\sigma\circ \pi$ is generically a local isomorphism, it follows by construction that $\rk G'\geq1$. We consider the exact sequence $0\to G'\to(\sigma\circ\pi)^*G\to G''\to0$ which satisfies $0<\rk G'$ and $\rk G''<\rk G$, and which is preserved by the endomorphisms $(\sigma\circ\pi)^*A_i$, so that it induces an exact sequence of $\DXS$-modules
\[
0\to(\wh E_{G'},\wh\nabla)\to(\sigma\circ\pi)^*(\wh E_G,\wh\nabla)\to(\wh E_{G''},\wh\nabla)\to0.
\]
If $\rk G'<\rk G$, we can apply induction to $(\wh E_{G'},\wh\nabla),(\wh E_{G''},\wh\nabla)$ and conclude that \ref{th:Deligneext}\eqref{th:Deligneext3} holds for \hbox{$(\sigma\circ\pi)^*(\wh E_G,\wh\nabla)$}, hence for $(\wh E_G,\wh\nabla)$ according to Step~2.

If $\rk G'=\rk G$, then $\rk G''=0$, so Step 1 applies to $(\wh E_{G''},\wh\nabla)$, and we are reduced to proving \ref{th:Deligneext}\eqref{th:Deligneext3} for $(\wh E_{G'},\wh\nabla)$, \ie we can assume that $A_1=\alpha_1\id$. Iterating the argument, we are reduced to the case where $A_i=\alpha_i\id$ for $i=1,\dots,\ell$, where $\alpha_1,\dots,\alpha_\ell$ are holomorphic functions on $S$, and $G$ is $\sho_S$-locally free. By considering a local basis of~$G$, it suffices to consider the case where $\rk G=1$.

\subsubsection*{Step 5: The case where $G$ is $\sho_S$-locally free of rank one}
We now consider $(\wh E_G,\wh\nabla)=(\sho_{\XS}(*D),\rd_{\XS/S}+\sum_{i=1}^\ell\alpha_i(s)\rd x_i/x_i)$. The argument for obtaining \ref{th:Deligneext}\eqref{th:Deligneext3} is then similar to that used in the proof of \cite[Cor.\,2.8]{MFCS2}. One can assume that, in the neighborhood of $s_o$ and for any $i$, $\alpha_i(s)\in\ZZ\implique\alpha_i(s)=0$. Then there exists a surjective morphism
\[
\DXS/\bigl((x_i\partial_{x_i}-\alpha_i(s)+1)_{i=1,\dots,\ell},(\partial_{x_j})_{j=\ell+1,\dots,n}\bigr)\to(\wh E_G,\wh\nabla)
\]
sending $1$ to $1/x_1\cdots x_\ell$, which is easily seen to be an isomorphism by the condition on $(\alpha_i)_{i=1,\dots,\ell}$.
\end{proof}

\begin{proof}[Proof of Proposition \ref{prop:firsttildeE}\eqref{prop:firsttildeE3}]
The statement is local, so, as a consequence of \eqref{lftmf} in the proof of Theorem \ref{th:Deligneext}\eqref{th:Deligneext1} and keeping the same notation, we are reduced to proving $\pi^*(\wh E_G,\wh\nabla)\simeq(\wh E_{\pi^*G},\wh\nabla)$, that~is,
\[
\pi^*(\sho_{\XS}(*D)\otimes_{\pOS}p^{-1}G,\wh\nabla)\simeq(\sho_{\XS'}(*D)\otimes_{\pOSp}p^{-1}\pi^*G,\pi^*{\wh\nabla}),
\]
which follows by taking the $0$-cohomology in \eqref{lftmf2}.
\end{proof}

\subsection{\texorpdfstring{$\DXS$}{DXS}-modules of D-type}
Recall Notation \ref{nota:Dtype}. In this section, we exhibit a family of regular holonomic $\DXS$-modules, that we call \emph{of D-type} and we prove in Proposition \ref{prop:equivDel} a special case of the analogue in the relative setting of \cite[Th.\,2.3.2]{K-K81} asserting that the restriction functor to the complement of the divisor is an equivalence of categories. The general case will be obtained in Theorem~\ref{Clftmf} below.\enlargethispage{\baselineskip}%

Let $Y$ be a closed hypersurface of $X$.

\begin{definition}\label{def:Dtype}
We say that \emph{a coherent $\DXS(*Y)$-module $\shl$ is of D-type}~if
\begin{enumeratea}
\item\label{def:Dtype1}
there exists a coherent $S$-locally constant sheaf $L$ on $\XsS$ such that $\shl_{|\XsS}\simeq E_L=(\sho_{\XsS}\otimes_{\pOS}L,\rd_{\XsS/S}\otimes\id)$, equivalently, $\shl_{|\XsS}$ is $\DXsS$-holonomic with characteristic variety contained in the zero section,
\item\label{def:Dtype2}
for each $s\in S$, the cohomology of $Li_s^*\shl$ is $\DX$-regular holonomic (in particular, $\DX$-coherent). In other words $\shl$ satisfies condition $(\mathrm{Reg}\,2)$ and thus $(\mathrm{Reg}\,1)$ (\cf Proposition \ref{Prop:2}).
\end{enumeratea}
We say that $\shl$ is \emph{strict} if $L$ is $\pOS$-locally free.
\end{definition}

We denote by $\Mod_Y(\DXS(*Y))$ the full subcategory of $\Mod(\DXS(*Y))$ whose objects are coherent $\DXS(*Y)$-modules of D-type.

\begin{lemma}\label{lem:Dtype}
For $\pi:S'\to S$ and $\shl\in\Mod_Y(\DXS(*Y))$, $L\pi^*\shl$ has cohomology in $\Mod_Y(\DXSp(*Y))$.
\end{lemma}

\begin{proof}
The good behaviour of \ref{def:Dtype}\eqref{def:Dtype1} by base pullback is clear. Let us check that of \ref{def:Dtype}\eqref{def:Dtype2}. Arguing as in Lemma \ref{lem:basepullback}, we see that $L\pi^*\shl$ is an object of $\rD^\rb_\coh(\DXSp(*D))$. For any $s'\in S'$ we have $Li_{s'}^*L\pi^*\shl\simeq Li_{\pi(s')}^*\shl$, so the complex $Li_{s'}^*L\pi^*\shl$ has $\DX$-regular holonomic cohomology. By Proposition \ref{Prop:2}\eqref{Prop:21}, each $Li_{s'}^*L^j\pi^*\shl$ also has $\DX$-regular holonomic cohomology.
\end{proof}

\begin{proposition}\label{prop:shltowtE}
Assume that $Y=D$ is a divisor with normal crossings in $X$.
\begin{enumerate}
\item\label{prop:shltowtE1}
If $\shl$ is $\DXS(*D)$-coherent of D-type and strict, then the natural morphism $\psi:\shl\to j_*E_L$ sends $\shl$ isomorphically to $\wt E_L$.
\item\label{prop:shltowtE2}
If $\shl$ is $\DXS(*D)$-coherent of D-type, then $\shl$ is $\DXS$-regular holonomic.
\end{enumerate}
\end{proposition}

\begin{proof}
We note that the question is local on $\XS$, so we may replace~$E_L$ with~$\wh E_{G|\XsS}$ as in \eqref{lftmf}. Namely, we have
\[
\shl_{|\XsS}\simeq(\sho_{\XsS}\otimes_{\pOS}p^{-1}G,\wh\nabla).
\]
We note that:
\begin{itemize}
\item
If $\shl$ is strict, the second assertion follows from the first one, because $\wt E_L$ is $\DXS$-regular holonomic (Theorem \ref{th:Deligneext}\eqref{th:Deligneext3}).

\item
For any coherent $\DXS(*D)$-module of D-type, the morphism $\psi$ is injective. Indeed, according to the first point of Definition~\ref{def:Dtype}, the restriction of $\psi$ to $\XsS$ is an isomorphism. The assertion follows from the $\sho$\nobreakdash-quasi-coherence of $\shl$ (\cf Remark \ref{rem:localiz}\eqref{rem:localiz2}). We thus identify $\shl$ with a $\DXS(*D)$-submodule of $j_*E_L$.

\item
Due to \ref{def:Dtype}\eqref{def:Dtype2}, \ref{prop:shltowtE}\eqref{prop:shltowtE2} amounts to holonomicity of $\shl$ (in particular, $\DXS$-coherence).
\end{itemize}

\subsubsection*{Proof of Proposition \ref{prop:shltowtE}\eqref{prop:shltowtE1}}
We assume that $G$ is $\sho_S$-locally free of finite rank.
We~can mimic the end of the proof of \cite[Prop.\,2.11]{MFCS2} to directly show that $\psi$ is an isomorphism $\shl\isom\wt E_L$ because, although in \loccit\ we assumed the $\DXS$-coherence of~$\shl$, that proof works under the weaker assumption of its $\DXS(*Y)$-coherence.
This shows \ref{prop:shltowtE}\eqref{prop:shltowtE1}.
In particular this implies the $\DXS$-regular holonomicity of~$\shl$.\enlargethispage{\baselineskip}%

\subsubsection*{Proof of Proposition \ref{prop:shltowtE}\eqref{prop:shltowtE2}}
We now prove the holonomicity of $\shl$ by assuming only that~$G$ is $\sho_S$-coherent. As in the proof of Theorem \ref{th:Deligneext}\eqref{th:Deligneext3}, we argue by induction on the dimension of $S$. The case where $S$ is a point is well-known (\cite[Th.\,2.3.2]{K-K81}). We~assume that the result holds if $\dim S\leq d-1$ ($d\geq1$) and that $\dim S=d$.

\subsubsection*{Step~1: The case where $G$ is a torsion $\sho_S$-module}
Since we work locally, we can assume that the support of $G$ is contained in a hypersurface $T$ of $S$, having equation $h=0$, endowed with a finite morphism $q$ to $S'$ of dimension $d-1$, and that $h^kG=0$ for some $k\geq1$. We claim that any local section $m$ of $\shl$ is annihilated by $h^k$. Indeed, for any such~$m$, $h^km$ is zero on $\XsS$ and we can apply the result of Remark \ref{rem:localiz}\eqref{rem:localiz2}.

Therefore, $\Supp_S\shl$ is contained in $T$. By the induction hypothesis and the equivalence recalled in Setting \ref{set:local}, we deduce that $\shl$ is $q^{-1}\DXSp$-holonomic, hence, arguing as in Remark \ref{rem:TS}, we~conclude that $\shl$ is $\DXS$-holonomic, as desired.

\subsubsection*{Step~2: The case where $G_\tf$ is $\sho_S$-locally free}
We consider the exact sequence \eqref{eq:ttf} and we assume that $G_\tf$ is $\sho_S$-locally free. We also consider the similar exact sequence $0\to\shl_\rt\to\shl\to\shl_\tf\to0$. Let us prove that $\shl_\tf$ is of D-type. First, $\shl_\rt$ is easily seen to be $\DXS(*D)$-coherent, hence so is $\shl_\tf$. Next, $\shl_\tf$ satisfies the first point of Definition \ref{def:Dtype}, with local system $p^{-1}G_\tf$.

For the second point, we note that the image of $i^*_s\shl_\rt\to i^*_s\shl$ is $\DX$-coherent since $i^*_s\shl$ is so and $i^*_s\shl_\rt$ is $\sho$-quasi-coherent. As a consequence, it is $\DX$-regular holonomic since $i^*_s\shl$ is so, and thus $i^*_s\shl_\tf$ is regular holonomic. On the other hand, for~\hbox{$j<0$}, $L^ji^*_s\shl_\tf$ is $\sho$-quasi-coherent and supported on $D$ by our assumption on $G_\tf$, so $L^ji^*_s\shl_\tf=0$ by Lemma \ref{lem:localization}\eqref{lem:localization2}.

In conclusion, $\shl_\tf$ is of D-type, and we also deduce that $\shl_\rt$ is of D-type. By Step~1 and \ref{prop:shltowtE}\eqref{prop:shltowtE1}, $\shl$ is holonomic if $G_\tf$ is $\sho_S$-locally free.

\subsubsection*{Step~3: The general case}
Let \hbox{$\pi:S'\to S$} be a projective modification as in Proposition \ref{prop:Rossi} such that the $S'$-torsion-free quotient of $\pi^*G_\tf$ is $\sho_{S'}$-locally free. The $S'$-torsion-free quotient of $\pi^*G$, being equal to it, is then $\sho_{S'}$-locally free.

By Lemma \ref{lem:Dtype}, $L\pi^*\shl$ has cohomology in $\Mod_D(\DXSp(*D))$. We can therefore apply Steps 1 and ~2 to deduce that $\pi^*\shl$ is $\DXSp$-holonomic, and it is regular by Definition \ref{def:Dtype}\eqref{def:Dtype2}. It is moreover $\pi$-good (\cf Lemma~\ref{lem:qcoh}). Hence, $\pi_*\pi^*\shl$ is $\DXS$-regular holonomic (Theorem \ref{th:Regdirect}).

Since $\shl$ is $\DXS(*D)$-coherent, the image $\shl'$ of the adjunction morphism $\shl\to\pi_*\pi^*\shl$ is $\DXS$-coherent, hence regular holonomic (Proposition \ref{Prop:2}\eqref{Prop:22b}) and its kernel is $\DXS(*D)$-coherent. It follows that the latter, which clearly satisfies the first point of Definition \ref{def:Dtype}, also satisfies the second point since $\shl$ and $\shl'$ do so (\cf Proposition \ref{Prop:2}). Since $\pi$ is biholomorphic above $S\moins T$ for some closed analytic subset $T$ of codimension $\geq 2$, this kernel satisfies the assumption of Step~1. It is thus $\DXS$-holonomic, hence $\shl$ is $\DXS$-holonomic, as desired.
\end{proof}

Note that the assignment $\shl\mto L=\shh^0\DR(\shl_{|\XsS})$ is a functor which takes values in the category of coherent $S$-locally constant sheaves, since the characteristic variety of $\shl_{|\XsS}$ is contained in the zero section.

\begin{proposition}\label{prop:equivDel}
The category of strict regular holonomic $\DXS$-modules of D-type with singularities along~$D$ is equivalent to the category of $S$-locally constant sheaves on \hbox{$\XsS$} which are $\pOS$-locally free of finite rank, under the correspondences $\shl\mto L=\shh^0\DR(\shl_{|\XsS})$ and \hbox{$L\mto \shl=\wt E_L$}.
\end{proposition}

\begin{proof}
Owing to the Riemann-Hilbert correspondence on $\XsS$ (\cf\cite[Rem.\,A.10]{MFCS2}), we~have $\shl_{|\XsS}\simeq(\sho_{\XsS}\otimes_{\pOS}L,\rd_{\XS/S}\otimes\id)=E_L$. Since the natural morphism $\shl\to j_*E_L$ sends isomorphically $\shl$ to $\wt E_L$ (Proposition \ref{prop:shltowtE}\eqref{prop:shltowtE1}), the functor $L\mto \wt E_L$ from the category of $S$-locally constant sheaves on $\XsS$ which are $\pOS$-locally free of finite rank to that of regular holonomic $\DXS$-modules of D-type is essentially surjective. That it is fully faithful has been proved in Theorem \ref{th:Deligneext}\eqref{th:Deligneext2}.
\end{proof}

\subsection{Proof of Theorem \ref{th:inverseimage}}\label{subsec:pfinv}

Although the next proposition is not general enough to prove Theorem \ref{th:inverseimage}, it will be one of the main tools for its proof.

Let $f:X\to Y$ be a morphism of real or complex analytic manifolds, we denote by $\shd_{{Y\leftarrow X}/S}$ and by $\shd_{{X\to Y}/S}$ the relative transfer bi-modules.

\begin{proposition}\label{prop:localization}
Let $\shm$ be a regular holonomic $\DXS$-module with $X$-support $Z=\bigcup_iZ_i$ (see \eqref{eq:strictperv2}). Let $Y\subset X$ be a hypersurface such that $Y\supset Z_i$ if $\dim Z_i<\dim Z$, and $Z_o:=Z\moins(Y\cap Z)$ is smooth of pure dimension $\dim Z$. Then the localized $\DXS$-module $\shm(*Y)$ is regular holonomic and locally isomorphic to the projective pushforward of a relative $\shd$-module of D\nobreakdash-type.
\end{proposition}

\begin{proof}
The question is local. The assumption on $Y$ implies that the characteristic variety of $\shm_{|(X\moins Y)\times S}$ is contained in $(T^*_{Z_o}X)\times S$. By Kashiwara's equivalence, $\shm_{|(X\moins Y)\times S}$ is the pushforward by the inclusion map of a coherent $\sho_{Z_o\times S}$-module with flat relative connection, which is thus of the form $(\sho_{Z_o\times S}\otimes_{\pOS} L,\rd_{Z_o\times S/S})$ for some coherent locally constant $p^{-1}_{Z_o}\sho_S$-module $L$.

One can find a complex manifold~$X'$ together with a divisor with normal crossings $Y'\subset\nobreak X'$ and a projective morphism $f:X'\to X$ which induces a biholomorphism \hbox{$X'\moins Y'\isom Z_o$} (\cf\eg\cite[Prologue, Th.\,4]{A-H-V18}). We~set $\delta=\dim Z-\dim X=\dim X'-\dim X\leq0$. For each $\ell$, we consider the $\DXpS$-module $\shm^{\prime\ell}:=\shh^\ell\Df^*\shm$. Although it is not yet known to be coherent, it is $f$-$\sho$-good in the sense of Definition \ref{def:qcoh} (Proposition \ref{prop:qcoh}\eqref{prop:qcohf1}). By~considering the filtration by the order of the pole along $Y'$, one checks that $\shm^{\prime\ell}(*Y')$ is also $f$-$\sho$-good (\cf Remark \ref{rem:localiz}\eqref{rem:localiz1}).

If $\ell\neq\delta$, the sheaf-theoretic restriction of $\shm^{\prime\ell}$ to $(X'\moins Y')\times S$ is zero, therefore $\shm^{\prime\ell}(*Y')=0$ owing to $\sho$-quasi-coherence (\cf Lemma \ref{lem:localization}\eqref{lem:localization2}). Since \hbox{$\sho_{\XpS}(*Y')$} is flat over $\sho_{\XpS}$, we conclude that
\begin{equation}\label{eq:DpiMloc}
\Df^*\bigl(\shm(*Y)\bigr)[\delta]\simeq\bigl(\Df^*\shm\bigr)(*Y')[\delta]\simeq\shm^{\prime\delta}(*Y').
\end{equation}
We can interpret $\Df^*\bigl(\shm(*Y)\bigr)$ as the pullback of $\shm(*Y)$ as a $\DXS(*Y)$-module. Since $f$ is a modification and $Y'=f^{-1}(Y)$, we have
\begin{equation}\label{E:new2}
\begin{split}
\DXpS(*Y')&\simeq\sho_{\XpS}(*Y')\otimes_{f^{-1}(\sho_{\XS}(*Y))}f^{-1}\DXS(*Y)\\[
-3pt]
&=\shd_{X'\to X/S}(*Y')
\end{split}
\end{equation}
(in local coordinates $x'$ in $X'$ and $x$ in $X$, the determinant of the matrix expressing $\partial_{x'_i}$ in terms of $\partial_{x_j}$ is invertible in $\sho_{\XpS}(*Y')$). It follows that $\shm^{\prime\delta}(*Y')=\DXpS(*Y')\otimes_{f^{-1}\DXS(*Y)}f^{-1}\shm(*Y)$ is $\DXpS(*Y')$-coherent.

Since $Li_s^*$ commutes with $\Df^*$, $\shm^{\prime\delta}(*Y')$ is $\DXpS(*Y')$-coherent of D-type (Definition \ref{def:Dtype}) so, by Proposition \ref{prop:shltowtE}\eqref{prop:shltowtE2}, it is $\DXpS$-regular holonomic. Since $\shm^{\prime\delta}(*Y')$ is $f$-$\sho$-good and $\DXpS$-coherent, it is $f$-good.

According to \cite[Cor.\,2.4]{MFCS2}, $\Df_*(\shm^{\prime\delta}(*Y'))$ has regular holonomic cohomology. Furthermore, since $\shh^j(\Df_*\shm^{\prime\delta}(*Y'))$ is supported on $\YS$ for $j\neq0$, and since $\Df_*(\shm^{\prime\delta}(*Y'))\simeq(\Df_*\shm^{\prime\delta})(*Y)$, we have
\[
\Df_*(\shm^{\prime\delta}(*Y'))\simeq \shh^0\Df_*(\shm^{\prime\delta}(*Y'))\simeq\shh^0(\Df_*\shm^{\prime\delta})(*Y).
\]
On the other hand, there is a natural adjunction morphism (\cf \cite[Lem.\,4.28 \& Prop.\,4.34]{Ka2})
\[
\Df_*\Df^*\shm[\delta]\to\shm,
\]
which induces a morphism of coherent $\DXS(*Y)$-modules
\begin{equation}\label{eq:MprimeM}
\shh^0(\Df_*\shm^{\prime\delta}(*Y'))\simeq(\shh^0\Df_*\shm^{\prime\delta})(*Y)\to\shm(*Y),
\end{equation}
where the left-hand side is $\DXS$-coherent and regular holonomic. Its cokernel is zero on \hbox{$(X\moins Y)\times S$} and $\DXS(*Y)$-coherent, hence it is zero according to~Lemma~\ref{lem:localization}\eqref{lem:localization2}, so that this morphism is an isomorphism. In conclusion, $\shm(*Y)$ is regular holonomic.
\end{proof}

\begin{corollary}\label{cor:localization}
Let $\shm$ be a regular holonomic $\DXS$-module and let $Y$ be any hypersurface in $X$. Then the localized $\DXS$-module $\shm(*Y)$ is regular holonomic (hence $\DXS$-coherent).
\end{corollary}

\begin{proof}
The question is local on $\XS$. Let $Z$ denote the $X$-support of $\shm$. We argue by induction on the dimension of $Z$. The case where $\dim Z=0$ is clear because either $Z\subset Y$ and then $\shm(*Y)=0$, or $Z\not\subset Y$ and then $\shm(*Y)=\shm$.

Let $Y'$ be a hypersurface satisfying the properties as in Proposition \ref{prop:localization}. Then $Y\cup Y'$ satisfies the same properties. We consider the following commutative diagram
\[
\xymatrix@=.5cm{
0\ar[r]&\Gamma_{[Y']}\shm\ar[d]\ar[r]&\shm\ar[d]\ar[r]&\shm(*Y')\ar[d]\ar[r]&R^1\Gamma_{[Y']}\shm\ar[d]\ar[r]&0\\
0\ar[r]&(\Gamma_{[Y']}\shm)(*Y)\ar[r]&\shm(*Y)\ar[r]&\shm(*(Y'\cup Y))\ar[r]&(R^1\Gamma_{[Y']}\shm)(*Y)\ar[r]&0
}
\]
By Proposition \ref{prop:localization}, the terms of the top horizontal line, together with $\shm(*(Y'\cup Y))$, are regular holonomic. On the other hand, the support of the regular holonomic modules $\Gamma_{[Y']}\shm$ and $R^1\Gamma_{[Y']}\shm$ is of dimension $<\dim Z$, hence the assertion holds for these modules by the induction hypothesis, and the extreme terms of the lower horizontal line are regular holonomic. It follows that the remaining term $\shm(*Y)$ is regular holonomic.
\end{proof}

\begin{corollary}\label{cor:gammaY}
Let $\shm$ be a regular holonomic $\DXS$-module and let $Y$ be any closed analytic subset in $X$. Then $R\Gamma_{[Y]}\shm$ belongs to $\rD^\rb_\rhol(\DXS)$.
\end{corollary}

\begin{proof}
The question is local. The case of a hypersurface follows from Corollary \ref{cor:localization}. The general case follows by writing locally $Y$ as the intersection of hypersurfaces.
\end{proof}

\begin{corollary}\label{cor:Dtype2}
Let $Y$ be a closed hypersurface of $X$ and let $\shm$ be a coherent $\DXS(*Y)$-module $\shl$ of D-type.
Then $\shm$ belongs to $\Mod_\rhol(\DXS)$.
\end{corollary}

\begin{proof}
This is obtained by the arguments used in the proof of Proposition \ref{prop:localization}.
\end{proof}

\begin{proof}[Proof of Theorem \ref{th:inverseimage}]
We regard the morphism $f$ as the composition of a closed inclusion and a projection. The latter case is clear, so we only consider the case where $f:Y\hto X$ is the inclusion of a closed submanifold. Then Corollary \ref{cor:gammaY} implies that $R\Gamma_{[Y]}\shm$ is regular holonomic. By Kashiwara's equivalence, we conclude that $\Df^*\shm$ belongs to $\rD^\rb_\rhol(\DYS)$.
\end{proof}

\begin{proposition}\label{tmf12}
If $\shm,\shn$ are objects of $\rD^\rb_\rhol(\DXS)$, then so is $\shm\otimes_{\sho_{\XS}}^L\shn$.
\end{proposition}

\begin{proof}
Recall that the tensor product has been defined in Remark \ref{rem:extprod}. According to Theorem~\ref{th:inverseimage} applied to the diagonal embedding $\delta$, it is enough to prove that the $S$-external tensor product $\shm\boxtimes^L_\shd\shn$ is an object of $\rD^\rb_\rhol(\shd_{(X\times X)\times S/S})$. Holonomicity has been observed in Remark \ref{rem:extprod}. Regularity follows from the isomorphism $Li_s^*(\shm\boxtimes^L_\shd\nobreak\shn)\simeq\shd_{X\times X}\otimes_{\DX\boxtimes_\CC\DX}(Li_s^*\shm\boxtimes_\CC Li_s^*\shn)$ and from the regular holonomicity of the latter as a complex in $\rD^\rb(\shd_{X\times X})$.
\end{proof}

Let $Y_i$ ($i=1,\dots,p$) be hypersurfaces of $X$ defined as the zero set of holomorphic functions $h_i:X\to \CC$, set $Y=\bigcup_iY_i$ and let $\shn$ be a $\DXS(*Y)$-module. We~regard~$\shn$ as an $\sho_{\XS}(*Y)$-module with flat relative connection $\nabla$, and for a \hbox{tuple} $\alpha=(\alpha_1,\dots,\alpha_p)$ of holomorphic functions $\alpha_i:S\to\CC$, we~denote by $\shn h^\alpha$ the $\sho_{\XS}$-module $\shn$, endowed with the flat relative connection $\nabla+\sum_i\alpha_i\id\otimes\rd h_i/h_i$. The functor $\shn\to\shn h^\alpha$ is an auto-equivalence of the category $\Mod(\DXS(*Y))$, as well as of $\Mod_\coh(\DXS(*Y))$. We have a functorial isomorphism $\shn h^\alpha\simeq\sho_{\XS}(*Y)h^\alpha\otimes_{\sho_{\XS}}\shn\simeq\sho_{\XS}(*Y)h^\alpha\otimes_{\sho_{\XS}}^L\shn$.

\begin{corollary}\label{cor:nabladf}
Assume that $\shm$ is a regular holonomic $\DXS$-module. Then so is $\shm(*Y)h^\alpha$.
\end{corollary}

\begin{proof}[Sketch of proof]
One checks that the coherent $\DXS(*Y)$-module $\sho_{\XS}(*Y)h^\alpha$ is of D\nobreakdash-type along~$Y$ by using that $\sho_X(*Y)h^{\alpha(s)}$ is regular holonomic for each \hbox{$s\!\in\! S$}. Hence $\sho_{\XS}(*Y)h^\alpha$ is regular holonomic by Corollary \ref{cor:Dtype2}. Then, on noting that $\shm(*Y)h^\alpha\simeq\sho_{\XS}(*Y)h^\alpha\otimes_{\sho_{\XS}}\shm$, we conclude by applying Proposition \ref{tmf12}.
\end{proof}

\begin{example}[Generalized Mellin transform]\label{ex:maisonobe}
Let $f_1,\dots,f_p$ be meromorphic functions on $X$, \ie locally each $f_i$ is the quotient of two holomorphic functions $h_i,g_i$ without common factor. Let~$Y$ be the union of the divisors of zeros and poles of the functions~$f_i$, \ie locally the divisors of zeros of $h_i,g_i$. We also set $S=\CC^p$ with its analytic topology. It is usual to denote by $\sho_{\XS}(*Y)f^s$ the $\sho_{\XS}$-module $\sho_{\XS}(*Y)$ equipped with the twisted connection $\rd+\sum_is_i\rd f_i/f_i$. The same argument as in Corollary \ref{cor:nabladf} shows that $\sho_{\XS}(*Y)f^s$ is regular holonomic and that, if $\shm$ is a regular holonomic $\DX$\nobreakdash-module, then the $\DXS$-module $q^*(\shm(*Y))f^s:=q^*\shm\otimes_{\sho_{\XS}}\sho_{\XS}(*Y)f^s$ (where $q:\XS\to X$ denotes the projection) is also regular holonomic.

Furthermore, the $\DXS$-submodule of $q^*(\shm(*Y))f^s$ generated by the image of $q^*\shm\otimes1\cdot f^s$ is also regular holonomic, according to Proposition \ref{Prop:2}\eqref{Prop:22b}, as it is clearly coherent (being locally of finite type in a coherent $\DXS$-module). This property is the $S$-analytic variant of \cite[Prop.\,13]{Maisonobe16}, with the regularity assumption however.

Since $\shm$ is regular holonomic, it is good, and so is $q^*(\shm(*Y))f^s$. Therefore, if $X$ is compact, the $\shd$-module pushforward $\Dp_*\bigl[q^*(\shm(*Y))f^s\bigr]$ is an object of $\rD^\rb_\coh(\sho_S)$. This is the generalized Mellin transform of $\shm$ with respect to $(f_1,\dots,f_p)$.
\end{example}

\subsection{Another characterization of regular holonomicity}\label{subsec:5b}
For a closed analytic subset $Y$ of $X$, we denote by $\sho_{\wh{\YS}}= \varprojlim_{k\in\Z} \sho_{\XS}/\shj^k$ the formal completion of $\sho_{\XS}$ along $\YS$, where $\shj$ denotes the defining ideal of $\YS$ in $\XS$. Let $i:Y\hto X$ denote the inclusion. We consider the exact sequence of sheaves supported on $\YS$:
\[
0\to i_*i^{-1}\sho_{\XS}\to\sho_{\wh{\YS}}\to\shq_{\YS}\to0.
\]

\begin{corollary}
If $\shm\in\rD^\rb_\rhol(\DXS)$, then the complexes
\[
\Rhom_{\DXS}(\shm, \sho_{\wh{\YS}})\quand\Rhom_{\DXS}(\shm, \shq_{\YS})
\]
belong to $\rD^\rb_\cc(\pOS)$.
\end{corollary}

\begin{proof}
If follows from Corollary \ref{cor:gammaY} that we have $R\Gamma_{[Y]}\shm\in\rD^\rb_\rhol(\DXS)$, so that the complex $\Rhom_{\DXS}(R\Gamma_{[Y]}\shm, \sho_{\XS})$ belongs to $\rD^\rb_\cc(\pOS)$, according to \cite[Th.\,3.7]{MFCS1}. On the one hand, by mimicking the proof when $S$ is reduced to a point (\cf\eg\cite[Cor.\,2.7-2]{Mebkhout04}), one finds a natural isomorphism
\[
\Rhom_{\DXS}(R\Gamma_{[Y]}\shm; \sho_{\XS})
\isom \Rhom_{\DXS}(\shm, \sho_{\wh{\YS}}),
\]
hence the $S$-$\CC$-constructibility of the latter complex. On the other hand, we have natural isomorphisms
\begin{align*}
Ri_*i^{-1}\Rhom_{\DXS}(\shm, \sho_{\XS})&\isom Ri_*\Rhom_{\DXS}(i^{-1}\shm, i^{-1}\sho_{\XS})\\
&\simeq \Rhom_{\DXS}(\shm,Ri_*i^{-1}\sho_{\XS}),
\end{align*}
showing $S$-$\CC$-constructibility of the latter complex, and therefore that of $\Rhom_{\DXS}(\shm, \shq_{\YS})$.
\end{proof}

\begin{theorem}\label{T:sreg}
Let $\shm$ belong to $\rD^\rb_\hol(\DXS)$. Then $\shm$ is regular holonomic if and only if for any germ of closed analytic subset $Y\subset X$, $\Rhom_{\DXS}(\shm, \shq_{\YS})$ is isomorphic to zero.
\end{theorem}

\begin{lemma}\label{lem:for}
For any closed analytic subset $Y\subset X$ we have
\[
Li^*_s \sho_{\wh{\YS}}\simeq \sho_{\wh Y}\quand Li^*_s \shq_{\YS}\simeq \shq_{Y}.
\]
\end{lemma}

\begin{proof}
Thanks to the properties of $Li^*_s$ (\cf Proposition \ref{prop:3.1}) the result follows from Mittag-Leffler's condition since the morphisms $\sho_{\XS}/\shj^{k+1}\to \sho_{\XS}/\shj^k$ are surjective.
\end{proof}

\begin{proof}[Proof of Theorem \ref{T:sreg}]
We first remark that the theorem holds if $S$ is reduced to a point, according to \cite[(6.4.6) \& (6.4.7)]{K-K81}.

If $\Rhom_{\DXS}(\shm, \shq_{\YS})=0$ for all germ $Y\subset X$, then for any $s\in S$
\begin{equation}\label{eq:reqQ}
\begin{split}
\Rhom_{\DX}(Li^*_s\shm,\shq_Y)&\simeq\Rhom_{\DX}(Li^*_s\shm,Li^*_s\shq_{\YS})\\
&\simeq Li_s^*\Rhom_{\DXS}(\shm, \shq_{\YS})=0.
\end{split}
\end{equation}
From the preliminary remark we conclude that $Li^*_s\shm\in\rD^\rb_{\rhol}(\DX)$ for any $s\in S$, hence $\shm\in\rD^\rb_{\rhol}(\DXS)$.

Conversely, if $\shm$ is regular holonomic, then $\Rhom_{\DXS}(\shm, \shq_{\YS})$ is $S$-$\CC$-con\-structible, and by the variant of Nakayama's lemma (\cf\cite[Prop.\,2.2]{MFCS1}), it is zero if (and only~if) $Li_s^*\Rhom_{\DXS}(\shm, \shq_{\YS})=0$ for any $s\in S$. Reading \eqref{eq:reqQ} backwards and according to the preliminary remark, we find that the latter property is satisfied, hence
\[
\Rhom_{\DXS}(\shm, \shq_{\YS})=0.\qedhere
\]
\end{proof}

\section{Construction of the relative Riemann-Hilbert functor \texorpdfstring{$\RH_X^S$}{RHS}}\label{S:RHS}

In this section, we extend the definition of the functor $\RH_X^S$ introduced in \cite{MFCS2} when $\dim S=\nobreak1$ to the case $\dim S\geq2$. We will check that it satisfies properties similar to those explained in \cite{MFCS2,FMFS1}.

\subsection{Reminder on the subanalytic site and complements}
We first recall the main results in \cite{MFP2}. We consider the site in the real analytic manifold $S_\RR$ given by the usual topology, that~is, where the family $\Op(S)$ consisting of all open sets. On~the other hand, we have the subanalytic site~$X_{\sa}$ underlying the real analytic manifold~$X_{\R}$ for which the family of open subsets $\Op(X_{\sa})$ consists of subanalytic open subsets in $X_{\R}$. Lastly, we let $X_{\sa}\times S$ be the subanalytic site underlying the real analytic manifold $X_{\R}\times S_{\R}$, for which the family of open sets $\Op(X_{\sa}\times S)$ consists of those which are finite unions of products $U\times V$ with $U\in\Op(X_{\sa})$ and $V$ is open in $S$. A subset $T \subset \Op(X_{\sa}\times S)$ is a covering of $W=U\times V \in \Op(X_{\sa}\times S)$ if and only if it admits a refinement $\{U_i\times V_j\}_{i \in I, j \in J}$ such that $\{U_i\}_{i \in\ I}$ is a covering of $U$ (in $X_{\sa}$) and $\{V_j\}_{j \in J}$ is a covering of $V$ (in $S$). In particular, when $U$ is relatively compact~$I$ is finite but $J$ needs not to be so even if $V$ is relatively compact.

We have the following commutative diagram, where the arrows are natural morphisms of sites induced by the inclusion of families of open subsets.
\begin{equation}\label{eq:diagram sites}
\begin{array}{c}
\xymatrix{
& X_{\sa}\times S \ar[dr]^a & \\
\XS \ar[rr]^{\rho'} \ar[ur]^{\rho_S} \ar[dr]^\rho & & X_{\sa}\times S_{\sa} \\
& (\XS)_{\sa} \ar[ur]^\eta &
}
\end{array}
\end{equation}

We recall that
\begin{enumerate}
\step\label{step:rhoS}
$\rho_S^{-1}$ commutes with tensor products (\cf \textup{\cite[Lem.\,18.3.1(ii)(c)]{KS3}}) and \hbox{$\rho_S^{-1}\rho_{S*}\!\simeq\!\Id$}. Furthermore, $\rho_S^{-1}$ admits a left adjoint $\rho_{S!}$ which is exact and commutes with tensor products (\cf \textup{\cite[\S3]{MFP2}}).
\end{enumerate}

The following proposition generalizes \cite[Prop.\,3.3]{MFCS2}. Its proof is completely similar.
\begin{proposition}\label{exact}
The category $\Mod_{\rc}(\pXOS)$ is acyclic for $\rho_S$:
\[
\forall F\in\Mod_{\rc}(\pXOS),\; \forall k\geq 1, \quad \shh^k R\rho_{S *}F=0.
\]
In particular $\rho_{S*}$ is exact on $\Mod_{\rc}(\pXOS)$.
\end{proposition}

The next statement corrects \cite[Prop.\,3.5]{MFCS2}. Its proof is given in the appendix.

\begin{proposition}\label{P:C1}
Let $F^\cbbullet$ be a bounded complex of $\pXOS$-modules with $S$-$\R$-constructible cohomology. Then there exists an isomorphism $K^\cbbullet\to F^\cbbullet$ in $\rD^-(\pXOS)$, where $K^\cbbullet$ is a complex in $\rC^-(\pXOS)$ whose terms are locally finite sums \hbox{$\bigoplus_{\alpha\in A} \C_{U_{\alpha}\times V_{\alpha}} \otimes p_X^{-1}\sho_S$}, where the $U_{\alpha}$ are open subanalytic relatively compact in $X$ and the $V_{\alpha}$ are open relatively compact in $S$.
\end{proposition}

\begin{remark}\label{R1}
Recall the diagram \eqref{eq:diagram sites}.
From the fact that $a_*$ is fully faithful and that $a^{-1}a_*=\Id$, we have $\rho_{S*}=a^{-1}\rho'_*$ (thus $R\rho_{S*}=a^{-1}R\rho'_*$) as explained in \cite{MFP2} before Prop.\,3.1, we deduce that, for any open subset $U\times V\in\Op^\rmc(X_{\sa}\times S)$, the constant sheaf $\C_{U\times V}$ on the site $X_{\sa}\times S$ coincides with $\rho_{S*}\C_{U\times V}$. Similarly,
\[
R\rho_{S*}(\C_{U\times V}\otimes p_X^{-1}\sho_S)\simeq \rho_{S*}(\C_{U\times V})\otimes \rho_{S*}(\pXOS)\simeq \rho_{S*}(\C_{U\times V}\otimes p_X^{-1}\sho_S)
\]
since the isomorphisms hold true with $\rho'$ instead of $\rho_S$ (\cf \cite[Lem.\,3.6(2)]{MFP1}).
\end{remark}

\begin{remark}\label{R2}
For any $S$, the site $X_{\sa}\times S$ is a ringed site both relatively to the sheaf $\rho_{S*}(\pXOS)$ and to the sheaf $\rho_{S!} \sho_{\XS}$ (\cf \cite[p.\,449]{KS3}), and $\rho_S$ is a morphism of ringed sites in both cases. Thus, according to \hbox{\cite[Lem.\,18.3.1, Th.\,18.6.9]{KS3}}, $\rho_S^{-1}$ commutes with $\otimes$, $\otimes_{\rho_{S*}\pXOS}$, $\otimes_{\rho_{S!}\sho_{\XS}}$, $\otimes^L_{\rho_{S*}\pXOS}$ and $\otimes^L_{\rho_{S!}\sho_{\XS}}$.
We recall that $R\rho_{S*}\pXOS\simeq \rho_{S!}\pXOS$ (\cf \cite[Prop.\,3.16]{MFP2}).
Let $\pi:S'\to S$ be a morphism of complex manifolds. Since $\sho_{S'}$ is a $\pi^{-1}\sho_S$-module, $\rho_{S'*}(p_X^{-1}\sho_{S'})$ is a $\rho_{S'*}(p_X^{-1}\pi^{-1}\sho_S)$-module hence a $\pi^{-1}\rho_{S*}(p_X^{-1}\sho_S)$-module. Similarly, $\rho_{S'!}\sho_{\XS'}$ is a $\pi^{-1}\rho_{S!}\sho_{\XS}$-module. In other words, $\pi$ induces a morphism of ringed sites with respect to both sheaves of rings. Consequently, according to \cite[Th.\,18.6.9(i)]{KS3}, the derived functors $L\pi^*: \rD(\rho_{S*}\pXOS)\to \rD(\rho_{S'*}p^{-1}_X\sho_{S'})$ \resp (keeping the same notation $\pi$ for the morphism $\Id\times \pi$), $L\pi^*: \rD(\rho_{S!}\sho_{\XS})\to \rD(\rho_{S'!}\sho_{\XS'})$ are well-defined.

We recall that $\rho_S^{-1}R\rho_{S*}\simeq\id$, that $R\rho_{S*}$ commutes with $R\pi_*$ and $\rho_S^{-1}$ commutes with $\pi^{-1}$ according to \cite[Prop.\,17.5.3]{KS3}. Furthermore $(\pi^{-1}, \pi_*)$ is a pair of adjoint functors according to \cite[Th.\,17.5.2(i)]{KS3} and $\rho_S^{-1}$ commutes with $R\pi_*$,
which follows by copying the proof of \cite[Prop.\,2.2.1(ii)]{P05} since $\rho_S^{-1}$ admits an exact left adjoint ($\rho_{S!}$) hence takes injective sheaves to injective sheaves. (Note that, in general, we do not have $\pi^{-1}R\rho_{S*}=R\rho_{S'*}\pi^{-1}$ while we have $\pi^{-1}\rho_{S!}=\rho_{S'!}\pi^{-1}$.) We thus have isomorphisms of functors
\[
\rho_{S'}^{-1}\pi^{-1}R\rho_{S*}\simeq \pi^{-1},\quad
R\rho_{S'*}\rho_{S'}^{-1}\pi^{-1}R\rho_{S*}\simeq R\rho_{S'*}\pi^{-1}
\]
and, composing with the natural morphism $\Id\to R\rho_{S'*}\rho_{S'}^{-1}$, we deduce a natural morphism
\begin{starequation}\label{E3b}
\pi^{-1}R\rho_{S*}\to R\rho_{S'*}\pi^{-1}.
\end{starequation}%
\end{remark}
We have
\[\pi^{-1}\rho_{S*}(\C_{U\times V})=\pi^{-1}(\C_{U\times V})= \C_{U\times \pi^{-1}V}=\rho_{S'*}(\C_{U\times \pi^{-1}V})=\rho_{S'*}\pi^{-1}(\C_{U\times V}).
\]

\begin{lemma}\label{Lr}
Let $\pi: S'\to S$ be a morphism. Then we have an isomorphism of functors on $\rD^\rb_{\rc}(\pOS)$ with values in $\rD^\rb_{\rc}(\rho_{S'*}\pOSp)$:
\[
L\pi^*R\rho_{S*}(\cbbullet)\simeq R\rho_{S'*}L\pi^*(\cbbullet).
\]
\end{lemma}

\begin{proof}

Let $F\in \rD^\rb_{\rc}(\pOS)$.
By \eqref{E3b} we have a morphism (recalling that $R\rho_{S*}\pOS\simeq \rho_{S*}\pOS$)
\begin{multline*}
\rho_{S'*}p^{-1}\sho_{S'}\otimes^L_{\pi^{-1}\rho_{S*}p^{-1}\sho_{S}}\pi^{-1}R\rho_{S*}F\to \rho_{S'*}p^{-1}\sho_{S'}\otimes^L_{\pi^{-1}\rho_{S*}p^{-1}\sho_{S}}R\rho_{S'*}\pi^{-1}F\\
\to R\rho_{S'*}(p^{-1}\sho_{S'}\otimes^L_{\pi^{-1}p^{-1}\sho_{S}}\pi^{-1}F).
\end{multline*}
This gives the desired morphism $\tau_F:L\pi^*R\rho_{S*}(F)\to R\rho_{S'*}L\pi^*(F)$.
By Proposition~\ref{PL}, \cite[Rem.\,1.8]{FMFS1} and Remark \ref{R58} we may assume that $F=\C_{U\times V}\otimes p^{-1}\sho_S$, for some relatively compact open subsets respectively of~$X$ and~$S$, with $U$ subanalytic.
We have
\begin{align*}
L\pi^*\rho_{S*}(F)&=\rho_{S'*}(\pOSp)\otimes^L_{\pi^{-1}\rho_{S*}(\pOS)}\pi^{-1}\rho_{S*}(\C_{U\times V}\otimes p^{-1}\sho_S)\\
&\simeq \rho_{S'*}(\pOSp)\otimes^L_{\pi^{-1}\rho_{S*}(\pOS)}\pi^{-1}\rho_{S*}(\C_{U\times V})\otimes \pi^{-1}\rho_{S*}(p^{-1}\sho_S)\\
&\simeq \rho_{S'*}(\pOSp)\otimes \pi^{-1}\rho_{S*}(\C_{U\times V})
\simeq\rho_{S'*}(\pOSp)\otimes \rho_{S'*}\pi^{-1}(\C_{U\times V})\\
&\simeq\rho_{S'*}(\pOSp\otimes \C_{U\times \pi^{-1}V})
\simeq \rho_{S'*}(L\pi^*F).
\end{align*}
\end{proof}

\subsection{The construction of \texorpdfstring{$\RH_X^S$}{RHS} and behaviour under pushforward}
For this section, we refer to the notation introduced in \cite[\S2.2]{FMFS1}. However, instead of making use of the morphism $\rho'$ of \eqref{eq:diagram sites} as in \loccit, we replace it with $\rho_S$. Most proofs do not need any change.

We define the triangulated functor $\TH_X^S:\rD^\rb(p^{-1}\sho_S)^{\op}\to\rD^\rb(\DXS)$ by
\begin{align*}
\TH_X^S(F)&:=\rho_{S}^{-1}\rh_{\rho_{S *}\pOS}(R\rho_{S *}F, \Db^{t,S}_{\XS}),\\
\RH_X^S(F)&:=\rho_{S}^{-1}\rh_{\rho_{S *}\pOS}(R\rho_{S *}F, \sho^{t,S}_{\XS})[d_X].
\end{align*}
If $\dim S=1$ we recover the definitions of \cite{MFCS2} (where we restricted to $\rD^\rb_{\rc}(p^{-1}\sho_S)^{\op}$):
\begin{align*}
\TH_X^S(F)&:=\rho'^{-1}\rh_{\rho'_*\pOS}(\rho'_*F, \Db^ {t,S,\sharp}_{\XS}),\\
\RH_X^S(F)&:=\rho'^{-1}\rh_{\rho'_{*}\pOS}(\rho'_*F, \sho^{t,S,\sharp}_{\XS})[d_X].
\end{align*}
This can be seen as follows:
noting that (\cf\eqref{eq:diagram sites}) $\rho'_*=a_*\circ\rho_{S*}$, $a^{-1}a_*\!=\!\Id$ and $a^{-1}\sho^{t,S,\sharp}_{\XS}=\sho^{t,S}_{\XS}$ (\cf \cite[\S\S 3.1\,\&\, 3.2]{MFP2}), for $F\in \rD^\rb_{\rc}(p^{-1}\sho_S)$,
one finds a natural morphism
\[
\rho'^{-1}\rh_{\rho'_{*}\pOS}(\rho'_*F, \sho^{t,S,\sharp}_{\XS})\to \rho_{S}^{-1}\rh_{\rho_{S *}\pOS}(\rho_{S *}F, \sho^{t,S}_{\XS}).
\]
Let us check that is an isomorphism. The question being local, we can reduce, accord\-ing to Proposition~\ref{P:C1}, to proving the isomorphism for sheaves of the form $F=\C_{U\times V}\otimes p^{-1}\sho_S$, for some open subanalytic relatively compact subset $U$ of~$X$ and $V$ open in $S$. Both objects become then isomorphic to $R\Gamma_{X\times V} T\shh om(\C_{U\times S}, \sho_{\XS})$, as follows from \cite[Props.\,4.1\,\&\,4.7]{MFP1} and \cite[Prop.\,3.24]{MFP2}.

Let $U$ be a subanalytic open subset in $X$ and let us denote by $j: U\times S\to \XS$ the inclusion.

\begin{lemma}[Extension in the case of an open subanalytic set]\label{L:1/7}
Let $F\in\rD^\rb_\rc(p^{-1}_{U}\sho_S)$. Then there are natural isomorphisms in $\rD^\rb(\DXS)$
\begin{align}\tag{\protect\ref{L:1/7}$\,*$}\label{eq:L:1/7*}
\TH_X^S(j_!F)&\simeq\rho_S^{-1}Rj_*R\shhom_{\rho_{S*}p_{U}^{-1}\sho_S}(R\rho_{S*}F, j^{-1}\Db^{t, S}_{\XS}),
\\
\tag{\protect\ref{L:1/7}$\,**$}\label{eq:L:1/7**}
\RH_X^S(j_!F)[-d_X]&\simeq
\rho_S^{-1}Rj_*R\shhom_{\rho_{S*}p_{U}^{-1}\sho_S}(R\rho_{S*}F, j^{-1}\sho^{t, S}_{\XS}).
\end{align}
\end{lemma}

\begin{proof}
This lemma is a variant of \cite[Lem.\,3.25]{MFCS2} and its proof is similar to that of \loccit
\end{proof}

We have the variant of \cite[Prop.\,2.1]{FMFS1} using \cite[Prop.\,3.30]{MFP2}:
\begin{proposition}\label{Prop:3.30}
For each $s\in S$ we have an isomorphism of functors
\[
Li^*_s\RH^S_X(\cbbullet)[-d_X]\simeq T\shhom(Li^*_s(\cbbullet), \sho_X).
\]
\end{proposition}

We have the relative version of~\hbox{\cite[Th.\,4.1]{Ka3}}:

\begin{theorem}\label{L:A3}
Let $f:X\to Y$ be a morphism of real analytic manifolds, let $F\in \rD^{\rb}_{\rc}(\pYOS)$. Then we have a canonical morphism in $\rD(\DYS)$, functorial with respect to $F$ and compatible in a natural way with the composition of morphisms
\[
\Df_!\TH_X^S(f^{-1}F)\to \TH_Y^S(F).
\]
\end{theorem}

The proof is stepwise similar to that with $\dim S=1$ in \cite[Lem.\,2.5]{FMFS1} using Proposition \ref{PL}.
We omit~it here and the detailed proof can be found in the arXiv version of this paper. Proposition 7.1 of \cite{Ka3} (see also \cite[Th.\,5.7 \,(5.12)]{KS4}) has a relative version already used in~\cite{MFCS2}, whose proof is given below.

\begin{theorem}\label{L:A4}
Let $f:X\to Y$ be a morphism of complex analytic manifolds, let $F\in \rD^{\rb}_{\rc}(\pXOS)$, and assume that $f$ is proper on $\Supp F$. Then there is a canonical isomorphism in $\rD^\rb(\DXS)$ which is compatible in a natural way with the composition of morphisms
\begin{starequation}\label{E021}
\Df_*\RH_X^S(F)[d_X]{\to} \RH_Y^S(Rf_*F)[d_Y].
\end{starequation}
\end{theorem}

\begin{proof}
We recall (\cf\cite[(4)]{FMFS1}) that
\begin{equation}\label{E:3n}
\RH_X^S(F)\simeq \rh_{\shd_{\overline{X}\times \overline{S}/\overline{S}}}(\sho_{\overline{X}\times \overline{S}}, \TH^S(F))[d_ X].
\end{equation}
In view of Theorem \ref{L:A3}, where we replace $F$ by $Rf_*F$, and by adjunction, we derive a natural morphism
\begin{equation}\label{E22}
\mu_F: \Df_*\TH_X^S(F)\to \TH_Y^S(Rf_*F).
\end{equation}

\begin{lemma}\label{Ldirim}
The morphism $\mu_F$ is an isomorphism.
\end{lemma}

The proof of this lemma, which is a relative version of \cite[Th.\,4.4]{KS4}, is given in the appendix and some more details are also given in the arXiv version of this article.

Applying $\rh_{\shd_{\overline{Y}\times \overline{S}/\overline{S}}}(\sho_{\overline{Y}\times \overline{S}}, \cbbullet)[d_Y]$ to both terms of \eqref{E22} the right term becomes $\RH_Y^S(Rf_*F)$. For the left side of~\eqref{E22} we obtain
\begin{align*}
&\rh_{\shd_{\overline{Y}\times \overline{S}/\overline{S}}}(\sho_{\overline{Y}\times \overline{S}}, \Df_*\TH_X^S(F) )[d_Y]\\
&\simeq
\rh_{\shd_{\overline{Y}\times \overline{S}/\overline{S}}}(\sho_{\overline{Y}\times \overline{S}}, Rf_*(\shd_{Y_{\R}\leftarrow X_{\R}/S_{\R}}\otimes^L_{\shd_{X_{\R}\times S_{\R}/S_{\R}}}\hspace*{-1mm}\TH_X^S(F)))[d_Y]\\
&\simeq Rf_*(\rh_{f^{-1}\shd_{\overline{Y}\times \overline{S}/\overline{S}}}(f^{-1}\sho_{\overline{Y}\times \overline{S}}, \shd_{Y_{\R}\leftarrow X_{\R}/S_{\R}}\otimes^L_{\shd_{X_{\R}\times S_{\R}/S_{\R}}}\hspace*{-1mm}\TH_X^S(F)))[d_Y]\\
&\underset{\textup{(a)}}{\simeq} Rf_*(\rh_{f^{-1}\shd_{\overline{Y}\times \overline{S}/\overline{S}}}(f^{-1}\sho_{\overline{Y}\times \overline{S}}, \shd_{Y_{\R}\leftarrow X_{\R}/S_{\R}})\otimes^L_{\shd_{Y_{\R}\times S_{\R}/S_{\R}}}\hspace*{-1mm}\TH_X^S(F))[d_Y]\\\
&\underset{\textup{(b)}}{\simeq}Rf_*(\shd_{Y\leftarrow X/S}\otimes^L_{\DXS}(\rh_{\shd_{\overline{X}\times \overline{S}/\overline{S}}}(\sho_{\overline{X}\times \overline{S}}, \shd_{ X_{\R}\times S_{\R}/S_{\R}})\otimes^L_{\shd_{X_{\R}\times S_{\R}/S_{\R}}}\hspace*{-1mm}\TH_X^S(F)))[d_Y]\\
&\simeq Rf_*(\shd_{Y\leftarrow X/S}\otimes^L_{\DXS}(\rh_{\shd_{\overline{X}\times \overline{S}/\overline{S}}}(\sho_{\overline{X}\times \overline{S}}, \TH_X^S(F)[d_X])).
\end{align*}
Here (a) follows from \cite[(A10)]{Ka2} and (b) follows from the relative version of \cite[Lem.\,7.2]{Ka3} which asserts that
\begin{multline*}
\rh_{\shd_{\overline{Y}\times \overline{S}/\overline{S}}}(\sho_{\overline{Y}\times \overline{S}}, \shd_{Y_{\R}\leftarrow X_{\R}/S_{\R}})\\
\simeq
\shd_{Y\leftarrow X/S}\otimes ^L_{\DXS}\rh_{\shd_{\overline{X}\times \overline{S}/\overline{S}}}(\sho_{\overline{X}\times \overline{S}}, \shd_{X_{\R}\times S_{\R}/S_{\R}})[d_X-d_Y].
\end{multline*}
Therefore, by applying $\rh_{\shd_{\overline{Y}\times \overline{S}/\overline{S}}}(\sho_{\overline{Y}\times \overline{S}}, \cbbullet)[d_Y]$ to the left-hand side of \eqref{E22} we obtain an isomorphism with $\Df_!\RH^S_X(F)$
which concludes the construction of the morphism \eqref{E021}. Lemma \ref{Ldirim} shows that it is an isomorphism.
\end{proof}

\subsection{Riemann-Hilbert correspondence for Deligne's extensions}\label{subsec:CDT}
We recall that, for $F\in\rD^\rb(\pOS)$ one defines $\bD'F:=\Rhom_{\pOS}(F,\pOS)$ and $\bD F:=\Rhom_{\pOS}(F,\pOS)[2d_X]$.

Let $L$ be a coherent $S$-locally constant sheaf on $\XsS$.
We consider the setting of Notation \ref{nota:Dtype} and assume that $Y=D$ has normal crossings in~$X$.

\begin{proposition}\label{CDT}
Let $L$ be a coherent $S$-locally constant sheaf on $\XsS$ and let $\wt E_L$ be the associated Deligne extension. Then
\begin{itemize}
\item
the complex of $\DXS$-modules $\RH_X^S(j_!\bD'L)[-d_X]$ is isomorphic to $\wt E_L$ and thus it is regular holonomic,
\item
$\pSol \wt E_L\simeq j_!\bD'L$.
\end{itemize}
\end{proposition}

\begin{proof}
We adapt the idea of proof of \cite[Lem.\,4.2]{MFCS2}. Let us prove the second statement assuming the first one holds true. Firstly, the following lemma is similar to \cite[Lem.\,3.19]{MFCS2}:

\begin{lemma}\label{lem:DRRHS}
There exists an isomorphism of functors in $\rD^\rb_\rc(\pOS)$:
\[
\pDR_X(\RH_X^S(\cbbullet))\isom\bD(\cbbullet).
\]
\end{lemma}

Since $\RH_X^S(j!\bD'L)[-d_X]$ is holonomic by the first point, we have
\[
\pSol_X \RH_X^S(j_!\bD'L)[-d_X]\simeq \bD\pDR\RH_X^S(j_!\bD'L)[-d_X]
\overset{(*)}\simeq \bD \bD(j_!\bD'L)=j_!\bD' L,
\]
where $(*)$ follows from Lemma \ref{lem:DRRHS}.

Let us now prove the first statement. We will set
\[
\wt E^S_L=\RH_X^S(j_!\bD'L)[-d_X].
\]
Then, according to \eqref{eq:L:1/7**} and to \cite[Lem.\,3.22]{MFCS2}, we have
\begin{equation}\label{eq:ELStilde}
\begin{aligned}
\wt E^S_L&=\rho_{S}^{-1}Rj_*R\shhom_{\rho_{S*}p_{X^*}^{-1}\sho_S}(\rho_{S*}\bD'L, j^{-1}\sho^{t, S}_{\XS})\\
&\simeq \rho_S^{-1}Rj_*(\rho_{S*}L\otimes^L_{\rho_{S*}p_{X^*}^{-1}\sho_S}j^{-1}\sho_{\XS}^{t,S}).
\end{aligned}
\end{equation}

We make use of the following result of \cite{MFP2}, the proof of which we recall with details as it is used in an essential way below.

\begin{lemma}[{\cite[Prop.\,3.32]{MFP2}}]\label{LS}
The complex $\wt E^S_L$ is concentrated in degree zero.
\end{lemma}

\begin{proof}
Assume first that $L$ is the constant local system $p_{X^*}^{-1}G$ with $G$ being $\sho_S$-coherent. Since the question is local on $\XS$, we can assume that $G$ admits a finite resolution $\sho_S^\cbbullet\to G$ by free $\sho_S$-modules of finite rank. Then we have
\begin{multline}\label{eq:RHSG}
\rho_S^{-1}Rj_*(\rho_{S*}p_{X^*}^{-1}G\otimes_{p_{X^*}^{-1}\sho_S}^L j^{-1}\sho^{t,S}_{\XS})\\
\isofrom\rho_S^{-1}Rj_*(\rho_{S*}p_{X^*}^{-1}\sho_S^\cbbullet\otimes_{p_{X^*}^{-1}\sho_S} j^{-1}\sho^{t,S}_{\XS})\\
\simeq p_{X}^{-1}\sho_S^\cbbullet\otimes_{p_{X}^{-1}\sho_S} \rho_S^{-1}Rj_*j^{-1}\sho^{t,S}_{\XS}.
\end{multline}
We recall that, according to \cite[Prop.\,2.4.4\,(2.4.4)]{KS5}, we have isomorphisms of functors on $\rD^\rb(\C_{X_{\sa}\times S})$
\[
\Rhom (\rho_{S*}j_!\C_{\XsS}, (\cbbullet))\isom Rj_*\Rhom(\rho_{S*}\C_{\XsS}, j^{-1}(\cbbullet))\simeq Rj_*j^{-1}(\cbbullet).
\]
As a consequence, we obtain an isomorphism
\[
\rho^{-1}_SRj_*j^{-1}\sho^{t,S}_{\XS}\simeq \rho_S^{-1} \Rhom(\C_{\XsS}, \sho^{t,S}_{\XS}).
\]
It follows then from \cite[Prop.\,3.24(1)]{MFP2} applied to $\C_{\XsS}=\C_{X^*}\boxtimes\C_S$, that the latter complex is isomorphic to $\tho(\C_{\XsS}, \sho_{\XS})$, so that, according to \cite[Lem.\,7.5]{Ka3}
\begin{equation}\label{eq:rhojs}
\rho^{-1}_SRj_*j^{-1}\sho^{t,S}_{\XS}\simeq\sho_{\XS}(*D).
\end{equation}
We thus deduce from \eqref{eq:RHSG} and \eqref{eq:rhojs}:
\begin{equation}\label{eq:RHSGD}
\begin{aligned}
\rho_S^{-1}Rj_*(\rho_{S*}p_{X^*}^{-1}G\otimes_{p_{X^*}^{-1}\sho_S}^L j^{-1}&\sho^{t,S}_{\XS})\\[
-3pt]
&\simeq \rho_S^{-1}Rj_*(\rho_{S*}p_{X^*}^{-1}\sho_S^\cbbullet\otimes_{p_{X^*}^{-1}\sho_S} j^{-1}\sho^{t,S}_{\XS})\\
&\simeq p_{X}^{-1}\sho_S^\cbbullet\otimes_{p_{X}^{-1}\sho_S} \rho_S^{-1}Rj_*j^{-1}\sho^{t,S}_{\XS}\\
&\simeq p_{X}^{-1}\sho_S^\cbbullet\otimes_{p_{X}^{-1}\sho_S} \sho_{\XS}(*D)\\
&\simeq p_{X}^{-1}G\otimes_{p_{X}^{-1}\sho_S} \sho_{\XS}(*D),
\end{aligned}
\end{equation}
where the latter isomorphism follows from the $p_{X}^{-1}\sho_S$-flatness of $\sho_{\XS}(*D)$. We~conclude that $\wt E^S_{p_{X^*}^{-1}G}$ is concentrated in degree zero.

Let us now consider the case of a possibly non constant $S$-local system $L$ locally isomorphic to $p^{-1}G$, with $G$ being $\sho_S$-coherent. According to Remark \ref{rem:pGOXsS}, locally on $\XS$ we~have an isomorphism of $\sho_{\XsS}$-modules
\begin{equation}\label{ED}
L\otimes_{p_{X^*}^{-1}\sho_S}\sho_{\XsS}\simeq p_{X^*}^{-1}G\otimes_{p_{X^*}^{-1}\sho_S}\sho_{\XsS},
\end{equation}
and by flatness of $\sho_{\XsS}$ over $p_{X^*}^{-1}\sho_S$, the same equality holds with the derived tensor product $\otimes^L$. According to Proposition \ref{exact},
the natural morphism $\rho_{S*}L\to R\rho_{S*}L$ is an isomorphism and \cite[Prop.\,3.16]{MFP2} implies that $\rho_{S!}L\simeq R\rho_{S*}L$. Recall also that, by~\cite[Prop.\,3.13]{MFP2}, the functor $\rho_{S!}$ is exact and commutes with $\otimes$. From \eqref{ED} we~then derive $\sho_{\XsS}$-linear isomorphisms (they are a priori not $\DXsS$-linear since we use a tensor product over $\sho_{\XsS}$):
\begin{align*}
\rho_{S*}L\otimes^L_{\rho_{S*}p_{X*}^{-1}\sho_S}j^{-1}\sho_{\XS}^{t,S}&\simeq(\rho_{S!}L\otimes^L_{\rho_{S!}p_{X^*}^{-1}\sho_S}\rho_{S!}\sho_{\XsS})\otimes^L_{\rho_{S!}\sho_{\XsS}}j^{-1}\sho_{\XS}^{t,S}\\
&\simeq(\rho_{S!}(L\otimes^L_{p_{X^*}^{-1}\sho_S}\sho_{\XsS}))\otimes^L_{\rho_{S!}\sho_{\XsS}}j^{-1}\sho_{\XS}^{t,S}\\
&\simeq(\rho_{S!}(p_{X^*}^{-1}G\otimes^L_{p_{X^*}^{-1}\sho_S}\sho_{\XsS}))\otimes^L_{\rho_{S!}\sho_{\XsS}}j^{-1}\sho_{\XS}^{t,S}\\
&\simeq(\rho_{S!}p_{X^*}^{-1}G\otimes^L_{\rho_{S!}p_{X^*}^{-1}\sho_S}\rho_{S!}\sho_{\XsS})\otimes^L_{\rho_{S!}\sho_{\XsS}}j^{-1}\sho_{\XS}^{t,S}\\
&\simeq \rho_{S*}p_{_X^{*}}^{-1}G\otimes^L_{\rho_{S*}p_{X^*}^{-1}\sho_S}j^{-1}\sho_{\XS}^{t,S}.
\end{align*}
Then, due to \eqref{eq:ELStilde}, we deduce an $\sho_{\XS}$-linear isomorphism
\begin{equation}\label{eq:ESLGtilde}
\wt E^S_L\simeq \wt E^S_{p_{X^*}^{-1}G}\simeq p_{X}^{-1}G\otimes_{\pXOS} \sho_{\XS}(*D).
\end{equation}
By the first part of the proof, we deduce that $\wt E^S_L$ is concentrated in degree zero. This concludes the proof of Proposition \ref{LS}.
\end{proof}

\begin{lemma}\label{claim:5}
Both $\wt E_L$ and $\wt E^S_L$ are naturally $\DXS$-submodules of $j_*E_L$.
\end{lemma}

\begin{proof}
Firstly, applying the commutation of $\rho_{S}^{-1}$ with $j^{-1}$ together with the analogue of \cite[Cor.\,3.24]{MFCS2} entails that $\wt E^S_L$ and $\wt E_L$ coincide with $E_L$ when restricted to \hbox{$\XsS$}.

On the one hand, by construction, $\wt E_L$ is naturally a $\DXS$-submodule of $j_*E_L$. Let us check on the other hand that $\wt E^S_L$ is also naturally a $\DXS$-submodule of $j_*E_L$. From the natural morphism of functors $\id\to R\rho_{S*}\rho_S^{-1}$ on $\Mod(\C_{X_{\sa}\times S})$ we derive a natural morphism $\rho_S^{-1}Rj_*\to Rj_*\rho_{S}^{-1}$: denoting for a moment by $j_*^S$ the morphism in the subanalytic site, we have an isomorphism of functors $Rj_*^S\circ R\rho_{S*}\isom R\rho_{S*}\circ Rj_*$ (\cf \cite[Prop.\,2.2.1(i)]{P05}, by replacing there $\rho$ by $\rho_S$ and using the same argument), hence a morphism
\[
Rj_*^S\to Rj_*^S\circ R\rho_{S*}\rho_S^{-1}\isom R\rho_{S*} Rj_*^S\rho_S^{-1},
\]
and, by applying $\rho_S^{-1}$ on the left, we obtain the desired morphism. Recall also, as~already used, that $\rho_{S*}L\isom R\rho_{S*}L$. We then deduce a $\DXS$-linear morphism
\begin{align*}
\Psi:\wt E^S_L\simeq\rho_S^{-1}Rj_*(\rho_{S*}L&\otimes^L_{\rho_{S*}p_{X^*}^{-1}\sho_S} j^{-1}\sho^{t, S}_{\XS})\\
&\to Rj_*\rho_{S}^{-1}(\rho_{S*}L\otimes^L_{\rho_{S*}p_{X*}^{-1}\sho_S} j^{-1}\sho^{t, S}_{\XS})\\
&\overset{\textup{(a)}}\simeq Rj_*(L\otimes^L_{p_{X^*}^{-1}\sho_S}\rho_S^{-1} j^{-1}\sho^{t, S}_{\XS})\\
&\simeq Rj_*(L\otimes^L_{p_{X^*}^{-1}\sho_S}\sho_{\XsS})\\
&\overset{\textup{(b)}}\simeq Rj_*(L\otimes_{p_{X^*}^{-1}\sho_S}\sho_{\XsS})=Rj_*E_L\overset{\textup{(c)}}\simeq j_*E_L,
\end{align*}
where (a) follows from \eqref{step:rhoS}, (b) follows from the $p_{X^*}^{-1}\sho_S$-flatness of $\sho_{\XsS}$ and (c) from the Steinness of $j$.

In order to check that $\Psi$ is injective, it is enough to consider it as an $\sho_{\XS}$-linear morphism, and by \eqref{eq:ESLGtilde}, it is enough to check injectivity when $L=p_{X^*}^{-1}G$. Furthermore, the question is local on $S$.

In such a case, we consider a resolution $\sho_S^\cbbullet\to G$ as in the proof of Proposition~\ref{LS}. Then, locally on $S$, \eqref{eq:RHSGD} identifies $\wt E^S_{p_{X^*}^{-1}G}$ with $p_{X}^{-1}G\otimes_{\pXOS} \sho_{\XS}(*D)$ and the morphism~$\Psi$ with the natural morphism
\[
p_{X}^{-1}G\otimes_{\pXOS} \sho_{\XS}(*D)\to j_*(p_{X^*}^{-1}G\otimes_{p_{X^*}^{-1}\sho_S} \sho_{\XsS}),
\]
which is injective since $G$ is $\sho_S$-coherent, as wanted.
\end{proof}

\subsubsection*{End of the proof of Proposition \ref{CDT}}
By Lemma \ref{claim:5}, we are reduced to showing that both $\DXS$-submodules $\wt E^S_L$ and $\wt E_L$ of $j_*E_L$ coincide, and since the $\DXS$-structure is induced by that $j_*E_L$ for both, it is enough to check that they coin\-cide as $\sho_{\XS}$-submodules. The question is local on $\XS$. By Remark \ref{rem:pGOXsS}, there exists a \hbox{local} $\sho_{\XS}$-linear isomorphism $j_*E_L\simeq j_*\wh E_{p_{X^*}^{-1}G}$ under which $\wt E_L$ is identified with $\wh E_G=p_X^{-1}G\otimes_{\pXOS}\sho_{\XS}(*D)$. Besides, \eqref{eq:ESLGtilde} also identifies $\wt E^S_L$ with $p_X^{-1}G\otimes_{\pXOS}\sho_{\XS}(*D)$ as an $\sho_{\XS}$-module under the same local isomorphism, so~the $\sho_{\XS}$-submodules $\wt E_L$ and~$\wt E^S_L$ of $j_*E_L$ locally coincide, as desired.
\end{proof}

\section{Proof of Theorem \ref{RHH}}\label{sec:synopsis}
The proof of Theorem \ref{RHH} is similar to that of \cite[Th.\,1]{FMFS1}
where $\dim S=1$. However, various improvements are necessary in order to handle the case $\dim S\geq2$.

\subsection{First part of the proof of Theorem \ref{RHH}}
We prove that $\RH_X^S$ is a right quasi-inverse of the functor $\pSol_X:\rD^\rb_\rhol(\DXS)\to\rD^\rb_\cc(\pOS)$ by exhibiting an isomorphism of functors $\alpha:\id_{\rD^\rb_\cc(\pOS)}\isom\pSol_X\circ\RH_X^S$. This is the analogue in possibly higher dimension for $S$ of \cite[Th.\,3]{MFCS2}:
\begin{enumerate}
\step\label{TM}
For $F$ in $\rD^\rb_\cc(\pOS)$, $\RH_X^S(F)$ is an object of $\rD^\rb_\rhol(\DXS)$ and there exists a functorial isomorphism $\alpha_F:F\simeq\pSol_X(\RH_X^S(F))$ in $\rD^\rb_\cc(\pOS)$.
\end{enumerate}

\begin{proof}[Proof of \eqref{TM}]
We recall the notation for the duality functor (\cf\cite[Prop.\,2.23]{MFCS1}): for $F$ in $\rD^\rb_\rc(\pOS)$, we set $\bD'F=\Rhom_{\pOS}(F,\pOS)\in\rD^\rb_\rc(\pOS)$ and $\bD F=\bD'F[2\dim X]$. Functoriality in \eqref{TM} is obtained by means of Lemma \ref{lem:DRRHS}: once we know that, for $F$ in $\rD^\rb_\cc(\pOS)$, $\RH_X^S(F)$ belongs to $\rD^\rb_\rhol(\DXS)$, we can apply \cite[Cor.\,3.9]{MFCS1} together with bi-duality in $\rD^\rb_\cc(\pOS)$ (\cf \cite[Prop.\,2.23]{MFCS1}) to obtain a functorial isomorphism
\[
\alpha_F:F\isom\pSol_X(\RH_X^S(F)).
\]

For the first part of \eqref{TM}, the main step is provided by Proposition \ref{CDT}, where we proved the case $F=j_!\bD'L$ in the setting of Section~\ref{subsec:Deligneext} (from which we keep the notation), that~is, $D$ is a normal crossing divisor in~$X$ and $L$ is a coherent $S$-local system on $\XsS:=(X\setminus D)\times S$.

We now conclude the proof of \eqref{TM} by a standard induction on the dimension of the $X$-support of $F$, based on Theorem \ref{L:A4}, analogous to that of \cite[\S7.3]{Ka3} (\cf\cite[Th.\,3]{MFCS2} for the case $\dim S=1$).

We~thus assume that \eqref{TM} holds if $\dim\Supp_XF<k$ ($k\geq1$) and we prove that it holds for any~$F$ with $\dim\Supp_XF\leq k$. By functoriality, it is enough to prove the first part of \eqref{TM}, so that the question is local. By induction on the amplitude of the complex $F$, we can also assume that $F$ is an $S$-$\CC$-constructible sheaf. We can find a projective morphism $f:X'\to\nobreak X$ with $\dim X'=k$, which is biholomorphic from the complement $X^{\prime*}$ of a normal crossing divisor~$D'$ in~$X'$ to the smooth locus of dimension $k$ of $\Supp_XF$, so that $f^{-1}F$ is a coherent $S$-local system on $X^{\prime*}$. Let $j':X^{\prime*}\hto X'$ denote the inclusion. Proposition \ref{CDT} implies the regular holonomicity of $\RH^S(j'_!\bD'j^{\prime-1}f^{-1}F)$ and Theorem \ref{L:A4}, together with \cite[Cor.\,2.4]{MFCS2}, implies regular holonomicity of
\[
\RH_X^S(R f_*j'_!\bD'j^{\prime-1}f^{-1}F)\simeq\RH_X^S(j_!\bD'j^{-1}F),
\]
where $j$ denotes the inclusion of $f(X^{\prime*})$ in $X$. Hence, if $F$ is an $S$-$\CC$-constructible sheaf with $X$-support of dimension $\leq k$, $\RH_X^S(\bD'F)$ is regular holonomic since it fits in a distinguished triangle whose third term is regular holonomic by the induction hypothesis. The same holds for $\bD'F$ for any $F\in\rD^\rb_\cc(\pOS)$ with $\dim\Supp_XF\leq k$, and replacing $F$ with $\bD'F$, which has the same $X$-support, we~conclude that $\RH_X^S(F)$ is regular holonomic.
\end{proof}

\subsection{Second part of the proof of Theorem \ref{RHH}}
We prove that $\RH_X^S$ is a left quasi-inverse of the functor $\pSol_X:\rD^\rb_\rhol(\DXS)\to\rD^\rb_\cc(\pOS)$ by exhibiting an isomorphism of functors $\beta:\id_{\rD^\rb_\rhol(\DXS)}\isom\RH_X^S\circ\pSol_X$.

\begin{enumerate}
\step\label{fullyfaithful}
For each object $\shm$ of $\rD^\rb_\rhol(\DXS)$, there exists an isomorphism
\[
\beta_\shm:\shm\isom\RH_X^S(\pSol_X(\shm)),
\]
functorial in~$\shm$.
\end{enumerate}

\subsubsection*{Proof of \eqref{fullyfaithful}}
As in \cite[\S3.b]{FMFS1}, we construct a bi-functorial isomorphism (with respect to $\shm,\shn\in\rD^\rb_\rhol(\DXS)$) which, by applying $Li^*_s$ for any $s\in S$, yields that of \cite[Cor.\,8.6]{Ka3}:
\begin{equation}\label{eq:HomRSol}
\Hom_{\DXS}(\shm, \RH^S_X(\pSol_X(\shn)))\isom\Hom_{\pOS}(\pSol_X(\shn),\pSol_X(\shm)),
\end{equation}
that we also denote as $\eqref{eq:HomRSol}_{\shm,\shn}$. We then define the morphism
\[
\beta_\shm:\shm\to\RH^S_X(\pSol_X(\shm))
\]
as the unique morphism such that $\eqref{eq:HomRSol}_{\shm,\shm}(\beta_\shm)=\id_{\pSol(\shm)}$. One classically deduces from the full faithfulness of $\pSol$ in the absolute case (a consequence of the Riemann-Hilbert correspondence of \cite{Ka3}\,\& \cite{Mebkhout84}) that, for each $s\in S$, $\beta_{Li^*_s\shm}$ is an isomorphism. Therefore, by a Nakayama-type argument \cite[Prop.\,1.9\,\&\,Cor.\,1.10]{MFCS2}, $\beta_{\shm}$ is an isomorphism.

In order to check that $\beta_\shm$ is functorial with respect to $\shm$, that~is, it defines a morphism of functors $\id_{\rD^\rb_\rhol(\DXS)}\to\RH^S_X\circ\pSol_X$, we consider as in \cite[p.\,668]{MFCS2}, for a morphism $\varphi:\shm\to\shn$, the commutative diagram
\[
\xymatrix@C=1.8cm{
\Hom_{\DXS}(\shm, \RH^S_X(\pSol_X(\shm)))\ar[r]^-{\eqref{eq:HomRSol}_{\shm,\shm}}\ar@<4ex>[d]_{\RH^S(\pSol_X(\varphi))\circ\cbbullet}&\Hom_{\pOS}(\pSol_X(\shm),\pSol(\shm))\ar@<-1ex>[d]^{\cbbullet\circ\pSol_X(\varphi)}\\
\Hom_{\DXS}(\shm, \RH^S_X(\pSol_X(\shn)))\ar[r]^-{\eqref{eq:HomRSol}_{\shm,\shn}}&\Hom_{\pOS}(\pSol_X(\shn),\pSol_X(\shm))\\
\Hom_{\DXS}(\shn, \RH^S_X(\pSol_X(\shn)))\ar[r]^-{\eqref{eq:HomRSol}_{\shn,\shn}}\ar@<3ex>[u]_{\cbbullet\circ\varphi}&\Hom_{\pOS}(\pSol_X(\shn),\pSol_X(\shn))\ar@<-9ex>[u]^{\pSol_X(\varphi)\circ\cbbullet}
}
\]
Since $\id_{\pSol_X(\shm)}\circ\pSol_X(\varphi)=\pSol_X(\varphi)\circ\id_{\pSol_X(\shn)}$, we obtain from the commutativity of the diagram that
\[
\RH_X^S(\pSol_X(\varphi))\circ\beta_\shm=\beta_\shn\circ\varphi,
\]
which is the desired functoriality.

In order to obtain \eqref{eq:HomRSol}, it is enough to construct a bi-functorial morphism in $\rD^\rb(\pOS)$
\begin{equation}\label{eq:HomRSolloc}
\Rhom_{\DXS}(\shm, \RH^S_X(F))\to\Rhom_{\pOS}(F,\pSol_X(\shm))
\end{equation}
for $\shm\in\rD^\rb_\rhol(\DXS)$ and $F\in\rD^\rb_{\cc}(\pOS)$, and to show that it is an isomorphism. Then $\eqref{eq:HomRSol}_{\shm,\shn}$ is obtained by taking global sections of $\shh^0\eqref{eq:HomRSolloc}$ applied to $F=\pSol(\shn)$.

Such a morphism \eqref{eq:HomRSolloc} is given by \cite[(14)]{FMFS1}. In view of Proposition \ref{Prop:3.30} we can argue as in \loccit, where it is also shown that, for each $s_o\in S$, $Li_{s_o}^*\eqref{eq:HomRSolloc}$ can be identified with the morphism constructed in the absolute case by Kashiwara (\cite[Cor.\,8.6]{Ka3}), hence it is an isomorphism. To conclude that \eqref{eq:HomRSolloc} is an isomorphism, we~apply a Nakayama-type argument \cite[Prop.\,2.2]{MFCS1}.

To construct the morphism \eqref{eq:HomRSolloc} we need to check a finiteness property. The proof of \eqref{fullyfaithful} will thus be concluded with the proof of the following assertion.

\begin{enumerate}
\step\label{Tequivtf}
For any $\shm \in \rD^\rb_{\rhol}(\DXS)$ and for any $F\in\rD^\rb_\cc(\pOS)$, $\Rhom_{\DXS}(\shm, \RH^S_X(F))$ belongs to $\rD^\rb_{\cc}(\pOS)$.
\end{enumerate}

We also consider the statement:
\begin{enumerate}
\step\label{Rhom}
For any $\shm,\shn\in\rD^\rb_{\rhol}(\DXS)$, the complex $\Rhom_{\DXS}(\shm, \shn)$ belongs to $\rD^\rb_{\cc}(\pOS)$.
\end{enumerate}

We shall argue by induction on the pair $(\dim S,\dim\supp_X\shm)$ ordered lexicographically. For that purpose we introduce the following notations:
\begin{itemize}
\item
For $d\geq0$, we denote by $\eqref{Tequivtf}_d$, \resp $\eqref{Rhom}_d$ the corresponding statement concerning any $S$ satisfying $\dim S\leq d$.
\item
For $k\geq0$ we introduce the assertion $\eqref{Rhom}_{d,k}$ by requiring that the property holds for $\dim S\!\leq\!d-1$ or $\dim S\!=\!d$ and $\dim\supp_X\shm\!\leq\!k$.
\end{itemize}

Because of the first part of \eqref{TM} already proved, we have
\[
\eqref{Rhom}_d\implies\eqref{Tequivtf}_d.
\]
Conversely, the reverse implication also holds:
\[
\eqref{Tequivtf}_d \implies\eqref{Rhom}_d.
\]
Indeed we deduce from $\eqref{Tequivtf}_d$ an isomorphism $\beta_\shn:\shn\isom\RH^S_X(\pSol_X(\shn))$, so applying $\eqref{Tequivtf}_d$ with $F=\pSol_X(\shn)$ gives $\eqref{Rhom}_d$. In the following, we will focus on $\eqref{Rhom}_d$. Since belonging to $\rD^\rb_{\cc}(\pOS)$ is a local property (\cf\cite[\S2.5]{MFCS1}), we will allow restriction to an open neighborhood of a point of $\XS$ when needed.

\begin{proof}[Proof of $\eqref{Rhom}_d$]
The proof of $\eqref{Rhom}_d$ is done by induction on $d$, in order to reduce the assertion to the particular case of Lemma \ref{lem:D-type} below. Let us recall that $\eqref{Rhom}_0$ holds, according to \cite{Ka3}. So~we assume that $\eqref{Rhom}_{d-1}$ holds ($d\geq1$) and we consider $S$ with $\dim S=d$, together with any $\shm,\shn \in \rD^\rb_{\rhol}(\DXS)$.

\subsubsection*{Step 1: The $S$-torsion case}
\begin{lemma}\label{lem:Storsion}
Let $\dim S=d$ and let us assume that $\eqref{Rhom}_{d-1}$ holds. Let $\shm,\shn\in \rD^\rb_{\rhol}(\DXS)$ and assume that the cohomology of $\shm$ or that of $\shn$ is of $S$-torsion. Then $\eqref{Rhom}_d$ holds for $\shm,\shn$, that~is,
\[
\Rhom_{\DXS}(\shm, \shn)\in\rD^\rb_{\cc}(\pOS).
\]
\end{lemma}

\begin{proof}
By the biduality isomorphism of \cite[(3)]{MFCS1}, we have
\begin{equation}\label{eq:biduality}
\Rhom_{\DXS}(\shm,\shn)\simeq\Rhom_{\DXS}(\bD\shn,\bD\shm),
\end{equation}
and we recall that both $\bD\shn$ and $\bD\shm$ are regular holonomic according to Corollary~\ref{cor:prop2}. It is thus enough to prove Lemma \ref{lem:Storsion} when the cohomology of $\shn$ is of $S$-torsion. On the other hand, owing to the definition of $\rD^\rb_{\cc}(\pOS)$ in terms of micro-support (\cf\cite[\S2.5]{MFCS1}), if two terms of a distinguished triangle in $\rD^\rb(\pOS)$ are objects of $\rD^\rb_{\cc}(\pOS)$, then so does the third one. This property allows us to assume that $\shm,\shn$ are regular holonomic
$\DXS$-modules instead of complexes. Lastly, the same argument implies that belonging to $\rD^\rb_{\cc}(\pOS)$ is a local question on $\XS$.

The assumption on $\shn$ entails that $\Supp_S \shn$ is, locally with respect to $\XS$, contained in a closed analytic subset~$T$ of~$S$ such that $\dim T<\dim S$. We will argue by induction on $\dim T$.

If $\dim T=0$, as we consider a local situation, we may assume $T=\{s_o\}$ with maximal ideal sheaf $\fm_{s_o}$. By considering the (locally) finite filtration of $\shn$ by the $\DXS$-regular holonomic submodules $\fm_{s_o}^k\shn$ ($k\in\N$), we may reduce to assume $\fm_{s_o}\shn=0$. We then have $\shn\simeq i_{s_0 *} i^*_{s_0}\shn$ and by adjunction
\[
\Rhom_{\DXS}(\shm,i_{s_o*}i^{*}_{s_0}\shn)\simeq Ri_{s_0 *}\Rhom_{\DX}(Li^*_{s_0}\shm, i^*_{s_0}\shn),
\]
so that the proof is reduced to applying $\eqref{Rhom}_0$ or, equivalently, the absolute case proved by Kashiwara in \cite{Ka3}.

Assume now that $\dim T\geq1$ and let $T_0\subset T$ be a closed analytic subset of dimension $<\dim T$ such that $T\setminus T_0$ is non singular of dimension $\dim T$. Let $\pi: S'\to S$ be a projective morphism of complex manifolds satisfying the two conditions: $\pi(S')=T$ and, setting $T'_0:=\pi^{-1}(T_0)$, $\pi':=\pi|_{S'\setminus T'_0}: S'\setminus T'_0\to T\setminus T_0$ is a biholomorphism. In~particular $\dim S'=\dim T<d$.

Note that, according to Proposition \ref{LReginverse}, $L\pi^{*}\shm$, $L\pi^{*}\shn$ are regular holonomic on $\XS'$ and, by Theorem \ref{th:Regdirect}, $R\pi_*L\pi^*\shn$ belongs to $\rD^\rb_\rhol(\DXS)$. By $\eqref{Rhom}_{d-1}$, $\Rhom_{\DXSp}( L\pi^*\shm, L\pi^{*}\shn)\in \rD^\rb_{\cc}(\pOSp)$ thus, by the properness of $\pi$ and Lemma \ref{lem:adjunctionpi},
\[
\Rhom_{\DXS}(\shm, R\pi_*L\pi^{*}\shn)\in \rD^\rb_{\cc}(\pOS).
\]
The cones of the natural morphisms $R\pi_*L\pi^*\shn\to R\pi_*\pi^*\shn$ and $\shn\to R\pi_*\pi^*\shn$ are supported in $T_0$, so that, by the induction hypothesis we deduce that
\begin{align*}
&\Rhom_{\DXS}(\shm, R\pi_*\pi^{*}\shn)\in \rD^\rb_{\cc}(\pOS),\\\text{and thus}\quad&\Rhom_{\DXS}(\shm, \shn)\in \rD^\rb_{\cc}(\pOS).
\end{align*}
\end{proof}

\subsubsection*{Step 2: Induction on $\dim\supp_X\shm$}

\begin{lemma}
Let $\dim S=d\geq1$ and let us assume that $\eqref{Rhom}_{d-1}$ holds. Then $\eqref{Rhom}_{d,0}$ holds.
\end{lemma}

\begin{proof}
This is a consequence of Theorem \ref{th:inverseimage} and adjunction. Indeed, we can assume that $\supp_X\shm=\{x\}$. Denoting by $i:\{x\}\times S\hto\XS$ the inclusion, there exists a coherent $\sho_S$-module $\shm_0$ such that $\shm=\Di_*\shm_0$ by Kashiwara's equivalence, and we have (\cf\cite[Th.\,4.33]{Ka2}, which can be proved in a simple way in the present setting)
\[
Ri_*\Rhom_{\sho_S}(\shm_0,\Di^*\shn)\isom\Rhom_{\DXS}(\shm,\shn)[\dim X].
\]
By Theorem \ref{th:inverseimage}, $\Di^*\shn$ has $\sho_S$-coherent cohomology, and the assertion follows.
\end{proof}

We are thus reduced to proving:
\begin{enumerate}
\step\label{step:dk}
Let $\dim S=d\geq1$. Assume that $\eqref{Rhom}_{d,k-1}$ holds (with $k\geq1$). Then~$\eqref{Rhom}_{d,k}$ holds.
\end{enumerate}

We can reduce to the case $\shm,\shn$ are regular holonomic $\DXS$-modules. Let~$Z$ denote the $X$-support of $\shm$. Since the assertion is local, we~can assume that there exists a hypersurface $Y$ of $X$ such that $Y$ contains the singular locus of $Z$ and $\dim Z\cap Y<k$. Recall that localization along $\YS$ preserves regular holonomicity (\cf Corollary \ref{cor:localization}).

\begin{lemma}\label{lft}
It is enough to prove the assertion \eqref{step:dk} for $\shm,\shn$ such that $\shm=\shm(*Y)$ and $\shn=\shn(*Y)$.
\end{lemma}

\begin{proof}
The assertion for $\shm$ follows from the property that $R\Gamma_{[Y]}\shm$ belongs to $\rD^\rb_{\rhol}(\DXS)$ (Corollary~\ref{cor:gammaY}) and has $X$-support of dimension $<k$.

For the assertion concerning $\shn$, we recall the argument given at the end of the proof of \cite[Th.\,3]{FMFS1}. It is enough to prove that \eqref{step:dk} holds if $\shn=R\Gamma_{[Y]}\shn$. We have, by \cite[(3)]{MFCS1},
\[
\Rhom_{\DXS}(\shm,\shn)\simeq \Rhom_{\DXS}(\bD \shn, \bD \shm).
\]
Since $\shn$ has $\DXS$-coherent cohomology and is supported on $\YS$, we have
\[
\Rhom_{\DXS}(\bD\shn, (\bD\shm)(*Y))=0.
\]
Furthermore, $\bD\shm$ being regular holonomic, so is $R\Gamma_{[Y]}(\bD\shm)$ by Corollary \ref{cor:gammaY}, as well as $\shm':=\bD R\Gamma_{[Y]}(\bD\shm)$. Finally, applying once more \cite[(3) \& (1)]{MFCS1}, we obtain
\[
\Rhom_{\DXS}(\shm,\shn)\simeq\Rhom_{\DXS}(\shm',\shn),
\]
with $\dim \Supp_X\shh^j\shm'<k$ for any $j$, so the latter complex is $S$-$\CC$-constructible by $\eqref{Rhom}_{d,k-1}$.
\end{proof}

\subsubsection*{Step 3: Reduction to the case where $\shm$ is of D-type}
We take up the notation of the proof of Proposition \ref{prop:localization}, so that $f:X'\to X$ is a projective morphism inducing a biholomorphism $X'\moins D\isom Z\moins Z\cap Y$, and we set $\delta=\dim X'-\dim X$.

\begin{lemma}
Let $\shm,\shn$ be regular holonomic $\DXS$-modules such that $\shm=\shm(*Y)$. Then
\[
Rf_*\Rhom_{\DXpS}(\Df^*\shm,\Df^*\shn)\simeq\Rhom_{\DXS}(\shm,\shn).
\]
\end{lemma}

\begin{proof}
There is a natural morphism of adjunction (\cf \cite[Th.\,4.33(1)]{Ka2})
\[
Rf_*\Rhom_{\DXpS}(\Df^*\shm,\Df^*\shn)\to\Rhom_{\DXS}(\Df_*\Df^*\shm[\delta],\shn).
\]
We note that $\Df^*\shm,\Df^*\shn$ are $f$-good. Therefore, this morphism is a quasi-isomorphism. Furthermore, the natural adjunction morphism $\Df_*\Df^*\shm[\delta]\to\shm$ already used in the proof of Proposition \ref{prop:localization} is an isomorphism if $\shm=\shm(*Y)$.
\end{proof}

\subsubsection*{Step 4: The case where $\shm$ is of D-type along a normal crossing divisor}
Recall Definition \ref{def:Dtype} and Proposition \ref{prop:equivDel}. From Step 3, we are reduced to proving \eqref{step:dk} in the following setting:
\begin{enumerate}
\step\label{step:newsetting}
$\dim S=d$, $\dim X=k$, $\shm{}\simeq\wt E_L$ is of D-type along a normal crossing divisor $D\subset X$, and $\shn=\shn(*D)$.
\end{enumerate}

Since the assertion is local, we can assume that~$(E_L,\nabla)=(\wh E_{G|\XsS},\wh\nabla)$ as in \eqref{lftmf} for some $\sho_S$-coherent module $G$. We denote by $r\geq0$ the generic rank of $G$ as an $\sho_S$-module. The case $r=0$ means that $G$ is a torsion $\sho_S$-module.

Let us consider $X,S,D$ with $\dim S=d\geq1$, $\dim X=k\geq1$ and $D$ is a divisor with normal crossings in $X$. We shall argue by induction on $r$. For that purpose we consider, for $r\geq0$, the assertion:
\begin{enumerate}\itemindent0pt
\item[$\eqref{Rhom}_{d,k,r}$]
For any such $X,S,D$, the property $\eqref{Rhom}_{d,k-1}$ holds, as well as $\eqref{Rhom}_{d,k}$ if $\shm$ is of D-type with $E_L=\wh E_{G|\XsS}$ and~$G$ being $\sho_S$-coherent of generic rank $\leq r$, and~$\shn$ being regular holonomic satisfying $\shn=\shn(*D)$.
\end{enumerate}

Then, proving \eqref{step:newsetting} amounts to proving $\eqref{Rhom}_{d,k,r}$ for any $r\geq0$, according to Lemma \ref{lft}.

\pagebreak[2]
\begin{lemma}\label{lem:dkr-Dtype}
If $\eqref{Rhom}_{d-1}$ holds, then
\begin{enumerate}
\item\label{lem:dkr-Dtype1}
$\eqref{Rhom}_{d,k,0}$ holds;
\item\label{lem:dkr-Dtype2}
if $r\geq1$ and if $\eqref{Rhom}_{d,k,r}$ holds for any $G$ which is $\sho_S$-locally free and any $\shn=\shn(*D)$, then $\eqref{Rhom}_{d,k,r}$ holds.
\end{enumerate}
\end{lemma}

\begin{proof}
The first point follows from Lemma \ref{lem:Storsion}. In order to prove the second point, we choose (locally on $S$) a projective modification $\pi:S'\to S$ as in Proposition \ref{prop:Rossi}. Since the cones of the natural morphisms $\shn\to R\pi_*\pi^*\shn$ and $R\pi_*L\pi^*\shn\to R\pi_*\pi^*\shn$ are S-supported in dimension $<d$, Lemmas \ref{lem:Storsion} and \ref{lem:adjunctionpi} imply that \eqref{Rhom} holds for $\shm,\shn$ on $\XS$ if it holds for $L\pi^*\shm,L\pi^*\shn$ on $\XS'$. Furthermore, $L^j\pi^*\shm,L^j\pi^*\shn$ are also $S'$-supported in dimension $<d$ if $j\neq0$, so, by the same argument, \eqref{Rhom} holds for $\shm,\shn$ if it holds $\pi^*\shm$ and $\pi^*\shn$. Lastly, by Lemma \ref{lem:Storsion}, we can replace $\pi^*\shm$ and $\pi^*\shn$ with their quotient~$\shm',\shn'$ by their $S'$-torsion subsheaves. Note that we still have $\shn'=\shn'(*D)$ and $\shm'$ is of D-type with corresponding $E_{L'}$ isomorphic to $\wh E_{G'|\XsS'}$, where $G'$ is $\pOSp$-locally free by the choice of $\pi$. The assumption in \eqref{lem:dkr-Dtype2} therefore implies that $\eqref{Rhom}$ holds for any $G$ of generic rank $r$, as wanted.
\end{proof}

\subsubsection*{Step 5: Induction on the generic rank of $\shm$}
We use the same notation as in Step~4 of the proof of Theorem \ref{th:Deligneext}.

\begin{lemma}\label{lem:D-type}
Assume that $\eqref{Rhom}_{d,k-1}$ holds. Then $\eqref{Rhom}_{d,k,1}$ holds.
\end{lemma}

\begin{proof}
Assume $\shm$ is of D-type with $E_L=\wh E_{G|\XsS}$ and $G$ being $\sho_S$-coherent and of generic rank one, and let $\shn$ be regular holonomic with $\shn=\shn(*D)$. By a projective modification
we can assume that~$G$ is $\sho_S$-locally free of rank one. Then, by~Propo\-si\-tion~\ref{prop:shltowtE}\eqref{prop:shltowtE1}, we have $\shm=\wh E_G$.

The matrix of the relative connection $\wh\nabla$ on~$\wh E_G$ in some local $\sho_S$-basis of $G$ writes $\rd_{\XS/S}\otimes \Id+\sum_{i=1}^\ell\alpha_i(s)\rd x_i/x_i\otimes \Id$. Let us denote by $\shn x^{-\alpha}$ the $\sho_{\XS}$-module $\shn(*D)$ for which we add to the connection the relative $1$-form $-\sum_{i=1}^\ell\alpha_i(s)\rd x_i/x_i\otimes \Id$. Then $\shn x^{-\alpha}$ is also regular holonomic, according to Corollary~\ref{cor:nabladf}. Since $\shm=\shm(*D)$ and $\shn=\shn(*D)$, they also are $\DXS(*D)$\nobreakdash-mod\-ules
and, as recalled before Corollary~\ref{cor:nabladf}, the functor
$(\cbbullet)x^\alpha$ is an auto-equivalence of $\Mod(\DXS(*D))$. Therefore, we~have
\begin{align*}
\Rhom_{\DXS}(\shm,\shn)&=\Rhom_{\DXS(*D)}(\shm,\shn)\\
&\simeq\Rhom_{\DXS(*D)}(\shm x^{-\alpha},\shn x^{-\alpha})\\
&=\Rhom_{\DXS}(\shm x^{-\alpha},\shn x^{-\alpha}).
\end{align*}
Furthermore, $\shm x^{-\alpha}\simeq(\sho_{\XS}(*D),\rd_{\XS/S})$. Since $\DR(\shn x^{-\alpha})$ is $S$-$\C$-constructible (\cf\cite[Th.\,3.7]{MFCS1}) and $\Rhom_{\DXS}([\sho_{\XS}(*D)/\sho_{\XS}],\shn x^{-\alpha})$ also by $\eqref{Rhom}_{d,k-1}$, we~conclude that $\Rhom_{\DXS}(\shm,\shn)$ belongs to $\rD^\rb_{\cc}(\pOS)$.
\end{proof}

\begin{lemma}
Let us assume that $\eqref{Rhom}_{d,k,r-1}$ holds ($r\geq2$). Then $\eqref{Rhom}_{d,k,r}$ holds.
\end{lemma}

\begin{proof}
Let us assume that $r\geq2$. Together with the assumptions of \eqref{step:newsetting} in Step~4, we can now assume that $G$ is $\sho_S$-locally free of finite rank~$r$ (Lemma \ref{lem:dkr-Dtype}) and thus, by Proposition~\ref{prop:shltowtE}\eqref{prop:shltowtE1}, $\shm\simeq(\wh E_G,\wh\nabla)$. We argue in a way similar to that in Step 4 in the proof of Theorem \ref{th:Deligneext}\eqref{th:Deligneext3}.

If each endomorphism $A_i(s)$ occurring in the connection matrix of $\nabla$ is scalar, \ie of the form $\alpha_i(s)\id$, then Lemma \ref{lem:D-type} allows us to conclude, by obvious reduction to rank one, the proof of $\eqref{Rhom}$ for $\shm=(\wh E_G,\wh\nabla)$ and $\shn$ as above.

We can thus assume that some endomorphism $A_i(s)$, say $A_1(s)$, is not scalar. We~use the notation of Step 4 in the proof of Theorem \ref{th:Deligneext}\eqref{th:Deligneext3}, in particular \eqref{step:sigma1} and \eqref{step:sigma2}. Recall that we~consider the composition $S'\To{\pi}\Sigma\To{\sigma} S$, with $\Sigma:=\{\det(\alpha_1\Id-A_1(s))=0\}$ and $\pi:S'\to\Sigma$ is a resolution of singularities.

Recall also that we set $G'=\ker(\alpha_1\id-\sfA_1\circ\sigma\circ\pi)\subset(\sigma\circ\pi)^*G$ and we consider the exact sequence $0\to G'\to(\sigma\circ\pi)^*G\to G''\to0$. By construction and assumption,
we have $1\leq\rk G'<\rk G$ and thus $\rk G''<\rk G$.

By~applying the induction hypothesis $\eqref{Rhom}_{d,k,r-1}$ to $G',G''$, we deduce that
\[
\Rhom_{\DXSp}(L(\sigma\circ\pi)^*\shm,L(\sigma\circ\pi)^*\shn)
\]
belongs to $\rD^\rb_{\cc}(\pOS)$, hence so does its pushforward by $\sigma\circ\pi$. Thus
\[
\Rhom_{\DXS}(\shm,R(\sigma\circ\pi)_*L(\sigma\circ\pi)^*\shn),
\]
belongs to $\rD^\rb_{\cc}(\pOS)$ by Lemma \ref{lem:adjunctionpi}.

The cones of the natural morphisms
$
R(\sigma\circ\pi)_*L(\sigma\circ\pi)^*\shn\to R(\sigma\circ\pi)_*(\sigma\circ\pi)^*\shn
$,
$
(\sigma\circ\pi)_*(\sigma\circ\pi)^*\shn\to R(\sigma\circ\pi)_*(\sigma\circ\pi)^*\shn
$ and
$\sigma_*\sigma^*\shn\to (\sigma\circ\pi)_*(\sigma\circ\pi)^*\shn$
are supported on $X\times \sigma(\mathrm{Sing}(\Sigma))$. Therefore,
$\Rhom_{\DXS}(\shm, (\sigma\circ\pi)_*(\sigma\circ\pi)^*\shn)
\in\rD^\rb_{\cc}(\pOS).
$
Recalling (\cf \eqref{step:sigma1}) that locally $\sigma_*\sho_{X\times\Sigma}$ is $\sho_{\XS}$-free of rank $\deg\sigma$, $\sigma_*\sigma^*\shn$, which is isomorphic to $\sigma_*\sho_{X\times\Sigma}\otimes_{\sho_{\XS}}\shn$, is regular holonomic. Besides, according to \eqref{step:sigma2}, $\shn$ is a direct summand of $\sigma_*\sigma^*\shn$. Therefore $\Rhom_{\DXS}(\shm, \shn)$ is a direct summand of $\Rhom_{\DXS}(\shm, \sigma_*\sigma^*\shn)$, thus again by $\eqref{Rhom}_{d-1}$, $\Rhom_{\DXS}(\shm, \sigma_*\sigma^*\shn)$ is an object of $\rD^\rb_{\cc}(\pOS)$, hence the same holds true for $\Rhom_{\DXS}(\shm, \shn)$.
\end{proof}

This lemma concludes Step~5, therefore the proof of $\eqref{Rhom}_d$, and thereby that of Theorem~\ref{RHH}.
\end{proof}

As a consequence of Theorem \ref{RHH} we can now generalize Proposition \ref{prop:equivDel}:

\begin{theorem}\label{Clftmf}
Let $Y$ be a hypersurface in $X$. Then the category of $\DXS$-modules of D\nobreakdash-type along~$Y$ is equivalent to the category of $S$-locally constant sheaves on \hbox{$\XsS$} with finite rank, under the functor $\shl\mto L=\shh^0\DR(\shl_{|\XsS})$ with quasi-inverse $L\to \RH_X^S(j_!\bD'L)[-d_X]$.
\end{theorem}

\begin{proof}
Let us prove that the second functor takes values in the category of modules of $D$\nobreakdash-type along $Y$. Let $L$ be an $S$-locally constant sheaf on \hbox{$\XsS$} of finite rank. Since, by \cite[Prop.\,2.6]{FMFS1} (valid without any restriction on~$\dim S$), $\RH_X^S(j_!\bD'L)$ is localized along $Y$, and regular holonomic by Theorem \ref{RHH}, hence $\DXS(*Y)$-coherent, it remains to prove that $\RH_X^S(j_!\bD'L)[-d_X]$ is in degree zero and that $\RH_X^S(j_!\bD'L)[-d_X]|_{X^*\times S}\simeq
\sho_{X^*\times S}\otimes _{p_{X^*}^{-1}\sho_S}L$. The first assertion follows by the usual Nakayama-type argument.
The second assertion follows from
\cite[Cor.\,3.24]{MFCS2} (valid without any restriction on~$\dim S$), which shows that $\RH_{X^*}^S(\bD'L)[-d_X]\simeq \sho_{X^*\times S}\otimes_{p^{-1}_{X^*}\sho_S}L$.

Let $\shl$ be a $\DXS$-module of $D$-type along $Y$ with associated $S$-local system $L$, that is $\DR(\shl)|_{X^*\times S}\simeq L$; we have a natural morphism $\DR(\shl)\to j_*L$.
By Corollary \ref{cor:Dtype2}, $\shl$ is regular holonomic.
According to Lemma \ref{lem:DRRHS}, $\DR(\RH^S_X(j_!\bD'L)[-d_X])\simeq j_*L$ thus, by Theorem \ref{RHH}, there exists a natural morphism of localized regular holonomic $\DXS$-modules $\shl\to \RH_X^S(j_!\bD'L)[-d_X]$, which has to be an isomorphism since that is so on $X^*\times S$.
\end{proof}

\section*{Appendix: Complements}
\setcounter{equation}{0}
\renewcommand{\thesection}{\Alph{section}}
\setcounter{section}{1}\label{A1}
\addcontentsline{toc}{section}{Appendix: Complements}

The aim of this appendix is twofold:
\begin{itemize}
\item
to prove Proposition \ref{P:C1}, correcting thus the statement of \cite[Prop.\,3.5]{MFCS2} (namely, contrary to \loccit, the resolution in \ref{P:C1} may be not bounded below),
\item
and to correct and complete the contents of \cite[\S1.2.2]{FMFS1}.
\end{itemize}
Taking into account these corrections, the main results in both articles remain valid.

Furthermore, for the sake of completeness, we will give a detailed proof of Lemma \ref{Ldirim}, by following the same strategy as in \cite[Th.\,4.4]{KS4}.

In the following, we use the notation $\sho_V$ for $\CC_V\otimes\sho_S$, for any open set $V\subset S$. Recall that, for $\Omega, \Omega'$ open in $X$ and $V$, $V'$ open in $S$, for any $\sho_S$\nobreakdash-module~$G$, we have:
\begin{itemize}
\item{
$\Hom_{\pXOS}(\C_{\Omega}\boxtimes \sho_V, \C_{\Omega'\times V'}\otimes p_X^{-1}G)
\simeq \Gamma (\Omega \times V; p_X^{-1}G)$ if $\Omega\subset \Omega'$ and $ V\subset V'$.}
\item{If $\Omega\times V\cap \ov{\Omega'\times V'}=\emptyset$, then $\Hom_{\pXOS}(\C_{\Omega}\boxtimes \sho_V, \C_{\Omega'\times V'}\otimes p_X^{-1}G)=0$.}
\end{itemize}

We~consider the family $\shscr=\shscr_X\times\shscr_S$ of open sets $W=U\times V$ of $\XS$, where
\begin{enumeratea}

\item
$U$ is connected open subanalytic relatively compact in $X$.
\item
$V$ is open and relatively compact in $S$.
\end{enumeratea}

This family satisfies the conditions of \cite[(A.7) \& (A.8)]{KS4}, since each point of $\XS$ has a fundamental system of neighborhoods consisting of open sets in $\shscr$, and the intersection of two open sets of $\shscr$ is a finite union of elements in $\shscr$.

Notice that for any $W_1=U_1\times V_1, W_2=U_2\times V_2\in \shscr$ we have
\refstepcounter{equation}
\begin{starequation}\label{eq:UVW}
\Gamma(W_1;\CC_{W_2}\otimes\pXOS)=
\begin{cases}
\Gamma(W_1; p_X^{-1}\sho_{V_2})=\Gamma(V_1; \sho_{V_2})&\text{if }U_1\subset U_2,\\
0&\text{otherwise}.
\end{cases}
\end{starequation}%

\begin{definition}\label{DVTMF}
In a way similar to \cite[\S A.2]{KS4}, we define:
\begin{itemize}
\item
$\bA_X=\Mod_\wrc(\pXOS)$;
\item
$\bP_X$ the category whose objects $\fW=(W_i=U_i\times V_i)_{i\in I}$
consist of locally finite families
of open subsets of~$\shscr$,
and whose morphisms are described in a way similar to that of \cite[\S A.2]{KS4}:
\[
\Hom_{\bP_X}(\fW,\fW')=\prod_{i\in I}\Bigl(\bigoplus_{j\in J\mid W'_j\supset W_i}\Gamma(U_i\times\ov V_i;\pXOS)\Bigr),
\]
where $\fW'=(W'_j=U'_j\times V'_j)_{j\in J}$.
\item
In the following we will write simply $\bA, \bP$ in order to keep the notation clean.
The additive functor $L:\bP \to \bA$ given by
\[
L(\fW)=\bigoplus_{i\in I} \C_{U_i}\boxtimes {\sho}_{V_i}=\bigoplus_{i\in I}(\CC_{W_i}\otimes\pXOS),
\]
and, for $\fW,\, \fW'\in \bP$, for $\varphi\in \Hom_{\bP}(\fW, \fW')$, $L(\varphi)$ is given by the natural morphism
\[
\Gamma(U_i\times \overline{V_i}; \pXOS)\to \Gamma(U_i\times V_i; \C_{U_j}\boxtimes \sho_{V_j})=\Gamma(W_i; \pXOS)
\]
for each $i\in I,\, j\in J$ such that $W_i\subset W_j$. We remark that $L$ is a faithfull functor.
\item
The additive bifunctor $H:\bP^{\op}\times \bA\to \Ab$ given, for $\fW\in \bP$ and $F\in \bA$, by
\[
H(\fW, F)=\prod_{i\in I}\Gamma(U_i\times \overline{V_i}; F),
\]
where, following the notations of \cite[App.\,A.1]{KS4}, $\Ab$ denotes the category of abelian groups.

\item
$\alpha_{\fW, F}: H(\fW, F)\to \Hom_{\pXOS}(L(\fW), F)$ given by the restrictions
\[
\prod_{i\in I}\Gamma(U_i\times \overline{V_i}; F)\to\prod_{i\in I}\Gamma(U_i\times V_i; F)=
\Hom_{\pXOS}(\bigoplus_{i\in I} \C_{U_i}\boxtimes {\sho}_{V_i}; F).
\]
\end{itemize}
\end{definition}

For a morphism $u$ in $\bP$ let us denote by $\ker u$ (\resp $\Ima u$) the kernel (\resp the image) of~$L(u)$ as a morphism in $\bA$.

An object $F$ of $\bA$ is called $\bP$-coherent (\hbox{\cite[App. A.1]{KS4}}) if there exist $\fW\in\bP$ and an epimorphism $L(\fW)\to F$ in $\bA$ and if, for any $f\in H(\fW', F)$, there exist~$\fW''$ in $\bP$ and a morphism $g:\fW''\to \fW'$ such that $\fW''\overset{g}\to \fW'\overset{f}\to F$ is exact; that is
$L(\fW'')\overset{L(g)}{\to}L(\fW')\overset{\alpha(f)}{\to} F$ is exact in $\bA$ (and hence $f\circ g = 0$ since $L$ is faithfull).

\begin{proposition}\label{prop:KS4}\mbox{}
\begin{enumerate}
\item\label{prop:KS41}
The objects $(\bA,\bP,L,H,\alpha)$ as defined above satisfy Properties \textup{(A.1)--(A.4)} of \cite[\S A.1]{KS4}.
\item\label{prop:KS42}
An object of $\bA$ is $\bP$-coherent if and only if it is $S$-$\R$-con\-struc\-tible.
\end{enumerate}
\end{proposition}

The proof of \ref{prop:KS4}\eqref{prop:KS41}, which concerns the ring $\pXOS$, essentially reduces that of \cite[Prop.\,A.6]{KS4}, which concerns the ring $\sho_S$. However, the new ingredient with respect to \cite{KS4} is the property that any $S$-weakly $\R$-constructible sheaf $F$ admits an epimorphism $u:\bigoplus_k(\CC_{U_k}\boxtimes G_k)\to F$ for some locally finite family $(U_k)$ in $\shscr_X$ and some $\sho_S$-modules $G_k$, which can be assumed to be $\sho_S$-coherent if $F$ is $S$-$\R$-constructible. Details are given in the arXiv version of this paper.

\begin{corollary}\label{cor:KS4}
The category of $\bP$-coherent objects is closed under kernels, cokernels and extensions in $\bA$.
\end{corollary}

\begin{proof}
This is \cite[Prop.\,A.1]{KS4}, that we can apply according to Proposition~\ref{prop:KS4}.
\end{proof}

As in \cite[App.]{KS4}, we denote by $\rD^-(\bP)$ (resp. $\rD^\rb(\bP)$) the triangulated category obtained by taking the quotient of $\rK^-(\bP)$ (resp. $\rK^\rb(\bP)$) with respect to the null system of complexes $\fW^{\cbbullet}$ in $\bP$ such that $L(\fW^{\cbbullet})$ is acyclic in $\bA$. We denote by $\rD^-_\coh(\bP)$ the full subcategory of $\rD^-(\bP)$ of objects $\fW^{\cbbullet}$ such that $L(\fW^{\cbbullet})$ has $\bP$-coherent cohomologies.

On the other hand, we denote by $\rD^-_\coh(\bA)$ (\resp $\rD^\rb_\coh(\bA)$) the full subcategory of $\rD^-(\bA)$
(\resp $\rD^\rb(\bA)$) whose objects have $\bP$-coherent cohomologies, \ie according to Proposition \ref{prop:KS4}\eqref{prop:KS42}, are $S$-$\R$-constructible.
We also denote by $\rD^-_{\rb, \coh}(\bP)$ the full subcategory of $\rD^-(\bP)$ of objects $\fW^{\cbbullet}$ such that $L(\fW^{\cbbullet})\in \rD^\rb_\coh(\bA)$.

In this context, \cite[Prop.\,8.1.4]{KS1} holds with a similar proof and provides functorial isomorphisms, for any $F$ in $\bA$, any $\sigma\in\Delta$ and any $x\in|\sigma|$:
\begin{equation}\label{eq:VFx}
G_\sigma(F):=p_{U(\sigma) \ast}(F|_{U(\sigma)\times S})\simeq F|_{\{x\}\times S},\quad R^kp_{U(\sigma)*}(F|_{U(\sigma)\times S})=0\quad \forall k\geq1.
\end{equation}
The proofs of \cite[Th.\,8.1.10 \& 8.4.5(i)]{KS1} can then be adapted to the relative setting
by replacing everywhere the functor $\Gamma(U(\sigma),\cbbullet)$ with $p_{U(\sigma)*}(\cbbullet|_{U(\sigma)})$,
yielding that the natural inclusion functor $\rD^\rb(\bA)\to\rD^\rb_\wrc(\pXOS)$ is an equivalence by providing an explicit quasi-inverse denoted there by $R\beta$. We deduce an equivalence
\[
\rD^\rb_\coh(\bA)=\rD^\rb_\rc(\Mod_\wrc(\pXOS))\simeq\rD^\rb_\rc(\pXOS).
\]

The following proposition, which is a consequence of \cite[Th.\,A.5]{KS4} and the properties explained above, corrects \cite[Prop.\,1.6]{FMFS1}.

\begin{proposition}\label{PL}
The natural functor
\[
L: \rD^-_\coh({\bP})\to \rD^-_\coh(\bA)
\]
is an equivalence of categories which induces an equivalence of categories (still denoted by $L$ for short)
\[
L:\rD^-_{\rb, \coh}(\bP)\isom\rD^\rb_{\rc}(\pXOS).
\]
\end{proposition}

\begin{proof}[Proof of Proposition \ref{P:C1}]
We identify $\rD^-_{\rb, \coh}(\bP)$ with its essential image in $\rD^\rb_\coh(\bA)$, which is thus equivalent to $\rD^\rb_{\rc}(\pXOS)$, according to Proposition \ref{PL}. By definition, the objects $K^\cbbullet$ in this essential image satisfy the property described in Proposition~\ref{P:C1}.
\end{proof}

\begin{remark}\label{rem:muP}
In \eqref{E22} we defined the natural transformation
$\mu:\Df_*\TH_X^S(\cbbullet)\to \TH_Y^S(Rf_*(\cbbullet))$
between functors from $\rD^{\rb}_{\rc}(\pXOS)$ to $\rD^\rb(\DYS)$.
Due to Proposition~\ref{PL} we get (by composing with $L$) a natural transformation
\[
\mu\circ L:\Df_*\TH_X^S(L(\cbbullet))\to \TH_Y^S(Rf_*(L(\cbbullet)))
\]
between functors from $\rD^-_{\rb, \coh}(\bP)$ to $\rD^\rb(\DYS)$. We now extend it, by relaxing the coherency condition, to a natural transformation between functors from $\rD^\rb(\bP)$ to $\rD^\rb(\DYS)$. Recall that we denote by $f$ both maps $X\to Y$ and $\XS\to\YS$. Let $W=U\times V\in\shscr$ and consider the family $\{W\}$. By adjunction we have
\[
f^{-1}Rf_*(\CC_W)\simeq f^{-1}Rf_*(\CC_U)\boxtimes\CC_V,
\]
and thus
\[
f^{-1}Rf_*(\CC_W\otimes\pXOS)\simeq (f^{-1}Rf_*(\CC_U)\boxtimes\CC_V)\otimes\pXOS.
\]
Since $f^{-1}Rf_*(\C_{U})$ is a bounded complex with $\R$-constructible cohomologies, it is quasi-isomorphic to a bounded complex whose entries are locally finite sums $\bigoplus_{i\in I}\C_{U_i}$ where each $U_i\in \shscr_X$ (\cf\cite[Prop.\,A.1]{D-G-S11}). Therefore the complex $f^{-1}Rf_*L(\{W\})$ belongs to the essential image by $L$ of a complex in $\rD^\rb(\bP)$. The same property then holds for any $\fW\in\bP$ instead of $\{W\}$ and then for any bounded complex $\fW^\cbbullet\in\rD^\rb(\bP)$.

It is then meaningful to define $\mu_{L(\fW^\cbbullet)}:\Df_*\TH_X^S(L(\fW^\cbbullet))\to \TH_Y^S(Rf_*(L(\fW^\cbbullet))$ as the composition of
\[
\Df_*\TH_X^S(L(\fW^\cbbullet))\to\Df_*\TH_X^S(f^{-1}Rf_*L(\fW^\cbbullet))
\]
and
\[
\eta_{Rf_*L(\fW^\cbbullet)}:=\Df_*\TH_X^S(f^{-1}Rf_*L(\fW^\cbbullet))\to \TH_Y^S(Rf_*L(\fW^\cbbullet)),
\]
since all complexes on which $\TH^S$ is applied are bounded complexes.
\end{remark}

\begin{proof}[Proof of Lemma~\ref{Ldirim}] \label{PLAp}
The morphism $\mu_F$ being defined, checking that it is an isomorphism is a local question on $Y\times S$.
If $F\in \rD^{\rb}_{\rc}(\pXOS)$ has
non zero cohomologies only in the interval $[a,b]$,
the complexes $\Df_*\TH_X^S(F)$ and $\TH_Y^S(Rf_*F)$ can have non zero cohomologies only
for indices belonging to a finite interval $I=[-b-m, -a+n]$ for suitable $m,n\in \NN$ only depending on $\dim X$ and $\dim Y$.
Thus it is enough to show that $\shh^j(\mu_F)$ is an isomorphism for any $j\in I$.
In view of Proposition~\ref{P:C1}, we can replace $F$ with a bounded
complex $K^{\prime\cbbullet}=0\to K^{-N}\to K^\cbbullet$ with $N>-a+m+n$ such that
$K^\cbbullet=L(\fW^\cbbullet)$ for some $\fW^\cbbullet\in\rK^\rb(\bP)$.
Thus we obtain a distinguished triangle $K^{-N}[N]\to L(\fW^\cbbullet)\To{\tau} F\To{+1}$, and both
$\shh^j\Df_*\TH_X^S(K^{-N}[N])$ and $\shh^j\TH_Y^S(K^{-N}[N])$ vanish for $j\in I$. Hence, for any
$j\in I$,
\begin{gather*}
\shh^j\Df_*\TH_X^S(F)\isom\shh^j\Df_*\TH_X^S(L(\fW^\cbbullet)),\\ \shh^j\TH_Y^S(Rf_*F)\isom\shh^j\TH_Y^S(Rf_*L(\fW^\cbbullet)),
\end{gather*}
hence we are reduced to prove
that, for any $j\in \Z$, the morphism
\[
\shh^j(\mu_{L(\fW^\cbbullet)}):\shh^j\Df_*\TH_X^S(L(\fW^\cbbullet))\to \shh^j\TH_Y^S(Rf_*L(\fW^\cbbullet))
\]
(defined in Remark~\ref{rem:muP}) is an isomorphism.

Let $(y, s)\in Y\times S$. Let $\Omega_Y$ be a relatively compact subanalytic neighbourhood of $y$. By~the assumption on $f$ we can find an open relatively compact subanalytic neighbourhood $\Omega_X$ of $f^{-1}\Omega_Y\cap \Supp_X F$.
By restricting to $\Omega_X\times S$, we can reduce to study the case of $\C_{W}\otimes p_X^{-1}\sho_S$, for some $W=U\times V\in \shscr$.
By the assumption on $f$, we~have an isomorphism of functors
\[
R\Gamma_{Y\times V}(Rf_*(\cbbullet))\simeq Rf_*(
R\Gamma_{X\times V}(\cbbullet)).
\]
Thus we have isomorphisms
\[
Rf_*\TH_X^S(\C_U\boxtimes (\C_V\otimes\sho_S))\isom\R f_*R\Gamma_{X\times V}\TH_X^S(\C_U\boxtimes\sho_S)
\isom R\Gamma_{Y\times V}Rf_*\TH_X^S(\C_U\boxtimes\sho_S).
\]
Similarly
\begin{multline*}
\TH_Y^S(Rf_*(\C_U\boxtimes (\C_V\otimes\sho_S)))\isom\TH_Y^S(Rf_*(\C_U)\boxtimes (\C_V\otimes\sho_S)))\\
\isom\R\Gamma_{Y\times V}\TH_Y^S(Rf_*(\C_U)\boxtimes\sho_S).
\end{multline*}
Therefore we are led to prove the statement for $F=\C_U\boxtimes \sho_S$.

We will now follow the proof of \cite[Th.\,4.4]{KS4} which contains the statement in the absolute case.
We decompose $f$ by the graph embedding so that we first assume that $f:X\to Y$ is a closed embedding and next we treat the case of a smooth morphism.

\subsubsection*{Step 1}
Let us assume that $f:X\!\to\!Y$ is the embedding of a closed manifold. As~in the proof of \cite[Prop.\,2.4]{FMFS1}, let us start by proving the statement for \hbox{$F'=\C_{Z\times S}\otimes \pOS$}, with $Z=X\setminus U$. The conclusion for $U$ will easily follow by functoriality, by considering the exact sequence
\[
0\to \C_{U\times S}\to \C_{\XS}\to \C_{Z\times S}\to 0.
\]
We have $\TH^S_X(F')\simeq \Gamma_{Z\times S}(\Db_{\XS})$ regarded as a $\shd_{Y_{\R}\times S_{\R}/S}$-module. We note that the local structure of distributions supported by a submanifold entail that
\[
\Gamma_{\XS}(\shd_{\YS})\simeq f_*(\shd_{Y\leftarrow X/S}\otimes_{\DXS}\Db_{\XS}).
\]
Then
\begin{align*}
\Df_*\Gamma_{Z\times S}(\Db_{\XS})&\simeq f_*(\shd_{Y\leftarrow X/S}\otimes_{\DXS}\Gamma_{Z\times S}(\Db_{\XS}))\\
&\simeq \Gamma_{Z\times S}(\Db_{\YS})\simeq \tho(f_*\C_{Z\times X}, \Db_{\YS})\simeq \TH^S_Y(f_*(\C_{Z\times S}\otimes p^{-1}\sho_S)).
\end{align*}

\subsubsection*{Step 2}
Let us now assume that $f$ is smooth and $f\times \Id_S$ is proper on the support of~$F$.
The question being local we may assume as in (ii) of the proof of \cite[Th.\,3.5]{KS4} that $f:X=Y\times \R\to\nobreak \R$ is the projection.
Recall that $F\simeq \C_{U\times S}\otimes p_X^{-1}\sho_S$. Let $\ov{U}=Z$. Then $Z$ is closed subanalytic in $X$ and the assumption entails that $f|_{Z}$ is proper.

As in Step 1, we prove first the statement for $\C_{Z'\times S}\otimes p_X^{-1}\sho_S$ where $Z'$ is an arbitrary closed analytic subset of $X$ such that $f|_{Z'}$ is proper. By (3.9) of Lemma 3.6 in \loccit\ we may assume that for any $x\in Y$, $Z'\cap f^{-1}(x)$ is a closed interval in $\R$ containing $0$. With this assumption we have
\[
\tho(Rf_*\C_{Z'\times S}, \Db_{\YS})\simeq \tho(\C_{f(Z')\times S}, \Db_{\YS})\simeq \Gamma_{f(Z')\times S}(\Db_{\YS}).
\]
Noting that
$\shd_{Y\leftarrow X/S}$ is $\DYS$-isomorphic to $\DYS/\DYS \partial_t$, where $t$ is the coordinate in~$\R$, the result follows by the exact sequence
\[
0\to f_*\Gamma_{Z'\times S}(\Db_{\XS})\To{\partial_t}f_*\Gamma_{Z'\times S}(\Db_{\XS})\To{\int {}\cdot \rd t}\Gamma_{f(Z')\times S}(\Db_{\YS})\to 0
\]
proved in \cite[Lem.\,4.5]{Ka3}. The isomorphism for $U$ follows by functoriality considering the exact sequence
\[
0\to \C_{U\times S}\to \C_{Z\times S}\to \C_{\delta(U)\times S}\to 0.
\]
\end{proof}

\begin{remark}\label{R58}
In order to prove that the morphism
$\tau_F:L\pi^*R\rho_{S*}(F)\to R\rho_{S'*}L\pi^*(F)$
of Lemma~\ref{Lr} is an isomorphism
for any $F\in \rD^\rb_{\rc}(\pXOS)$, we may use the same argument of Lemma~\ref{Ldirim}
to reduce to the case $F\!=\!\C_{U\times V}\!\otimes\! p_X^{-1}\sho_S$.
\end{remark}

\section*{Added to the final arXiv version}
\addcontentsline{toc}{section}{Added to the final arXiv version}

In this addition, written for the convenience of the reader, we explain the proofs of Proposition~\ref{prop:KS4} and of Theorem \ref{L:A3}, and we justify the property used for the proof of Lemma \ref{Ldirim}, namely that the morphism $\mu_f$ in this lemma is compatible with composition of morphisms (we used it for the composition of a graph embedding and a projection).

We will make use of the next lemma.

\begin{lemma}\label{lem:epiconst}
Let $F$ be an $S$-weakly $\R$-constructible sheaf on $\XS$. Then there exist
\begin{itemize}
\item
a locally finite covering $(U(\sigma))_{\sigma\in\Delta}$ of $X$ by elements of $\shscr_X$,
\item
for each $\sigma\in\Delta$ an $\sho_S$-module $G_\sigma(F)$ on $S$,
\item
and an epimorphism $\bigoplus_{\sigma\in\Delta}\CC_{U(\sigma)}\boxtimes G_\sigma(F)\to F$.
\end{itemize}
If $F$ is $S$-$\R$-constructible, then all $\sho_S$-modules $G_\sigma(F)$ can be chosen coherent.
\end{lemma}

\begin{proof}
We will use the notations of \cite[\S8.1]{KS1}. The subanalytic triangulation theorem (\cite[Th.\,8.2.5]{KS1}) allows us to consider a simplicial complex $\bT=(T,\Delta)$ on $X$ with respect to which $F$ is
$S$-weakly $\R$-constructible. Recall (\cf \eqref{eq:VFx}) that, for any $\sigma\in\Delta$ and any $x\in|\sigma|$, the natural morphism $G_\sigma(F)\to F|_{\{x\}\times S}$ is an isomorphism and that $F|_{\{x\}\times S}$ is a coherent $\sho_S$-module, as explained in \cite[\S2.2]{MFCS1}. By~adjunction, there exists a canonical morphism
$
p_{U(\sigma)}^{-1}G_\sigma(F)\to F_{|U(\sigma)\times S},
$
and by extension a canonical morphism
$
\C_{U(\sigma)\times S}\otimes p_{U(\sigma)}^{-1}G_\sigma(F)\to F_{|U(\sigma)\times S},
$
thus a canonical morphism
\[
u_{G_\sigma(F)}: \C_{U(\sigma)\times S}\otimes p^{-1}G_\sigma(F)\to F.
\]
For any $\sigma$, for any $x\in |\sigma|$, and any $s\in S$, the germ
\[
u_{G_\sigma(F),\,(x, s)}: (\C_{U(\sigma)}\boxtimes G_\sigma( F))_{(x,s)}\to F_{(x,s)}
\]
is an isomorphism.
We now define $K:=\bigoplus_{\sigma\in \Delta} \C_{U(\sigma)}\boxtimes G_\sigma(F)$. Then
\begin{equation}\label{Estar}
u:=\sum_{\sigma\in\Delta} u_{G_{\sigma(F)}}: K\to F
\end{equation}
is an epimorphism.
\end{proof}

\subsection*{Proof of Proposition~\ref{prop:KS4}\eqref{prop:KS41}}

We explain how to reduce to the arguments given for the proof of \cite[Prop.\,A.6]{KS4}, noting that the proof of Property (A.1) is obvious.

\begin{proof}[Proof of \textup{(A.2)} in \cite{KS4}]
Given a morphism $\psi: \fW=(W_i)_{i\in I}\to \fW'=(W'_j)_{j\in J}$ in $\bP$, (A.2) asserts the existence of a morphism $\varphi: \fW''\to \fW$ in $\bP$ such that $\Ima L(\varphi)=\ker L(\psi)$ (which implies $\psi\circ\varphi=0$ since $L$ is faithful).

Recall Definition \ref{DVTMF}. The morphism~$\psi$ is given by a family $(\psi_{i,j})_{i\in I,j\in J}$ with $\psi_{i,j}\in\Gamma(U_i\times\ov{V_i}; \pXOS)$. Since $U_i$ is connected, \cite[Prop.\,A.1(2)]{MFCS2}) implies that $\Gamma(U_i\times\nobreak\ov{V_i}; \pXOS)=\Gamma(\ov{V_i}; \sho_S)$ and each~$\psi_{i,j}$ corresponds to some $g_{i,j}\in\Gamma(\ov{V_i}; \sho_S)$. As a consequence, for any $(x,s)\in\XS$, there exists a neighborhood $W(x,s)\in\shscr$ of $(x,s)$ such that each $\psi_{i,j}$ is the restriction to $W(x,s)\cap(U_i\cap\ov V_i)$ of an element $\wt\psi_{i,j}\in\Gamma(W(x,s);\pOS)$.

From this point, the proof is identical to that of \cite[p.\,66--67]{KS4}.\end{proof}

\begin{proof}[Proof of \textup{(A.3)} in \cite{KS4}]
Property (A.3) asserts that, given an epimorphism $u: F'\to F$ in $\Mod_\wrc(\pXOS)$, $\fW=(W_i)_{i\in I}$ in $\bP$ and $\varphi\in H(\fW, F)$, there exists $\fW'$ in $\bP$, a~cover $g:\fW'\to \fW$ (\ie $L(g)$ is an epimorphism) and a morphism $\psi: \fW'\to F'$, all data being such that the following diagram is commutative:
\[
\xymatrix{
\fW'\ar[r]^-{g}\ar[d]_{\psi}&\fW\ar[d]^{\varphi}\\
F'\ar[r]^-{u}&F
}
\]
Let us set $\varphi=(\varphi_i)_{i\in I}$, $\varphi_i\in \Gamma(U_i\times \ov{V}_i; F)$. We claim that, for each $(x,s)\in \XS$, there exists an open neighborhood $W(x,s)=U_x\times V_s\in\shscr$ of $(x,s)$ such that, for any $i\in I$, $U_x\cap U_i$ is the finite union of open sets $U_{x,i,j}\in\shscr_X$, each of which satisfying the following property, where we set $W_{i,j}(x,s)=U_{x,i,j}\times V_s$:
\begin{multline}\label{eq:phii}
\varphi_i|_{W_{i,j}(x,s)\cap (U_i\times \ov V_i)}=u(\varphi'_{i,j}(x,s))\\ \text{for some } \varphi'_{i,j}(x,s)\in\Gamma(W_{i,j}(x,s)\cap (U_i\times \ov V_i),F').
\end{multline}

Indeed, for any $(x,s)\in\XS$, let us choose $W=U\times V\in\shscr$ such that $(x,s)\in W$. Then the set $I_W=\{i\in I\mid W_i\cap W\neq\emptyset\}$ is finite. As in Lemma \ref{lem:epiconst}, we consider a simplicial complex $\bT=(T,\Delta)$ on $X$, compatible with $U$ and the finite families $(U\cap U_i)_{i\in I_W}$, and with respect to which both $F$ and $F'$ are $S$\nobreakdash-weakly $\R$-constructible. Let $i\in I$.

\begin{itemize}
\item
We first note that \eqref{eq:phii} is empty for $i\notin I_W$.
\item
If $(x,s)\!\notin\!\ov W_i$, we can find a neighborhood $W_i(x,s)\!\in\!\shscr$ such that \hbox{$W_i\cap W_i(x,s)\!=\!\emptyset$} and \eqref{eq:phii} for such an~$i$ and $W_i(x,s)$ is empty.

\item
If $(x,s)\in\ov W_i$, let $\sigma\in\Delta$ be such that $x\in U(\sigma)$. Then $U(\sigma)\cap U_i$ is a union of some open sets of the form $U(\tau_j)$. Then, for each such $\tau_j$, $\varphi_i|_{U(\tau_j)\times\ov V_i}$ corresponds to $\varphi_i|_{\{x'\}\times\ov V_i}$ for some $x'\in |\tau_j|$, by means of \eqref{eq:VFx}.

There exists an open neighborhood $V_{i,s}\subset V$ of $s$ such that $\varphi_i|_{\{x'\}\times\ov V_i}$ extends to a section $\wt\varphi_i|_{\{x'\}\times V_{i,s}}$ and that
$\wt\varphi_i|_{\{x'\}\times V_{i,s}}=u|_{\{x'\}\times V_{i,s}}(\varphi'_{i,j}(x',s))$ for some $\varphi'_{i,j}(x',s)\in\Gamma(\{x'\}\times V_{i,s}; F')$. Thus $\varphi'_{i,j}(x',s)$ uniquely defines $\varphi'_{i,j}(x,s)\in\Gamma(U(\tau_j)\times V_{i,s};F')$ such that $u(\varphi'_{i,j}(x,s))=\wt\varphi_i|_{U(\tau_j)\times V_{i,s}}$ by using \eqref{eq:VFx} once more. The finiteness of the set of indices $j$ allows us to choose $V_{i,s}$ independent of $j$.

We set
$U_x=U(\sigma)$,
$V_s =\bigcap_{i\in I_W}V_{i,s}$ and
$W_{i,j}(x,s)=U(\tau_j)\times V_{i,s}$.
Then
\eqref{eq:phii} is satisfied for such choices.
\end{itemize}

From this point, the proof is identical to that in \cite[p.\,67]{KS4}.\end{proof}

\begin{proof}[Proof of \textup{(A.4)} in \cite{KS4}]
Assume we are given $\fW=(W_i)_{i\in I},\wt\fW=(\wt W_j)_{j\in J}$ in $\bP$ and $\eta\in H(\fW,L(\wt\fW))$. Property (A.4) amounts to finding $\fW'$ in $\bP$, together with a cover $\psi:\fW'\to\fW$ in $\Hom_\bP(\fW',\fW)$ (\ie $L(\varphi)$ is an epimorphism) and a morphism $\varphi\in\Hom_\bP(\fW',\wt\fW)$ such that $L(\varphi)=\alpha(\eta\circ\psi)$ in $\Hom(L(\fW'),L(\wt\fW))$. By definition, $\eta=(\eta_{i,j})$ with $\eta_{i,j}\in\Gamma(U_i\times\ov V_i,\CC_{\wt W_j}\otimes\pOS)$.

According to \eqref{eq:UVW}, we can regard $\eta_{i,j}$ as an element of $\Gamma(\ov V_i;\sho_{\wt V_j})$, which is zero if $U_i\not\subset\wt U_j$. In particular, the $S$-support of $\eta_{i,j}$ (\ie the projection to $S$ of its support) is a compact subset of $\wt V_j$. The argument of \cite[p.\,67--68]{KS4} yields open subsets $V_{i,j,n}$ and $V'_{i,j,n}$ of $V_i$ (in our present notation), and we argue with $W_{i,j,n}:=U_i\times V_{i,j,n}$ and $W'_{i,j,n}:=U_i\times V'_{i,j,n}$ exactly as in \loccit\ to conclude the proof that Property (A.4) holds true.\end{proof}

\subsection*{Proof of Proposition~\ref{prop:KS4}\eqref{prop:KS42}}

The proof will be achieved with Lemmas \ref{lem:consPcoh} and \ref{lem:Pcohcons}, and relies on the following lemma.

\begin{lemma}\label{lem:Pcoh}
Let $(U_j)_{j\in J}$ be a locally finite family of open subsets of $X$ belonging to $\shscr_X$, and, for each $j\in J$, let $G_j$ be a coherent $\sho_S$-module. Then $K=\bigoplus_{j\in J}(\CC_{U_j}\boxtimes G_j)$ is $\bP$-coherent.
\end{lemma}

\begin{proof}
Proving $\bP$-coherency of $K$ amounts to showing
\begin{itemize}
\item{\cite[(A.5)]{KS4}}
the existence of a cover $\psi:\fW\to K$, \ie an element of $H(\fW,K)$ such that $\alpha_{\fW,K}(\psi)\in\Hom(L(\fW),K)$ is an epimorphism,
\item{\cite[(A.6)]{KS4}}
and for each morphism $\psi:\fW\to K$, the existence of $\fW'$ and $\varphi\in\Hom_{\bP}(\fW',\fW)$ such that $\fW'\To{\varphi}\fW\To{\psi}K$ is exact.
\end{itemize}

For the first condition, we recall that each $G_j\in\Mod_{\coh}(\sho_S)$ has a cover $\fV_j$ (\cite[Prop.\,A.8]{KS4}) and, since $(U_j)_{j\in J}$ is locally finite on $X$, the family $\fW=(U_j\times\fV_j)_{j\in J}$ is locally finite on $\XS$ and defines a cover of $K$.

For the second condition, the proof is the same as that of Property (A.2), if we replace in that proof $\psi\in\Hom_\bP(\fW,\fW')$ with $\psi\in H(\fW,K)$. The arguments are then completely parallel.
\end{proof}

\begin{lemma}\label{lem:consPcoh}
Let $F$ be an $S$-$\R$-constructible sheaf on $\XS$. Then $F$ is $\bP$\nobreakdash-coherent.
\end{lemma}

\begin{proof}
We consider the epimorphism \eqref{Estar} and we note that $K$ is $\bP$-coherent,~by~Lem\-ma \ref{lem:Pcoh}. Since both $F$ and $K$ are $S$-$\R$-cons\-tructible, so is $\ker u$. Applying~\eqref{Estar} to $\ker u$ yields a similar epimorphism $u':K'\to \ker u$ so that $F=\coker u'$ in $\Mod_\wrc(\pXOS)$. From Corollary \ref{cor:KS4} it follows that $F$~is $\bP$-coherent.
\end{proof}

\begin{lemma}\label{lem:Pcohcons}
If $F\in\Mod_\wrc(\pXOS)$ is $\bP$-coherent, then it is $S$-$\R$-constructible.
\end{lemma}

\begin{proof}
We have to prove that, for each $x\in X$, $F|_{\{x\}\times S}$ is $\sho_S$-coherent. According to \cite[Th.\,A.5]{KS4}, there exists $\fW^\cbbullet\in \rD^-_\coh(\bP)$ and an isomorphism $L(\fW^{\cbbullet})\simeq F$ in~$\rD^-(\bA)$. Therefore, $L(\fW^{\cbbullet})|_{\{x\}\times S}\simeq F|_{\{x\}\times S}$ in $\rD^-(\sho_S)$. We note that $\fW^{\cbbullet}|_{\{x\}\times S}$ is a complex in $\rD^-_\coh(\bP(\sho_S))$ introduced in \cite[A.2]{KS4} and thus the coherence of $F|_{\{x\}\times S}$ follows by the equivalence of categories of \cite[Th.\,A.9]{KS4}.
\end{proof}

\subsection*{Detailed proof of Theorem \ref{L:A3}}
Recall that for a real analytic manifold $X$, $\sha_{\XS_\R}$ denotes the sheaf of rings of real analytic functions on $\XS_\R$. The sheaf $\DXSR$ consists of differential operators with real analytic coefficients on $\XS_\R$ which commute with the holomorphic differential operators on $S$. The functor $\TH^S_X$ is defined from $\rD^\rb(\pXOS)\to \rD^+(\DXSR)$. It will be more convenient to use right $\shd$-modules, so we will make use of the functor $\TH^{S,\vee}_X$ obtained from $\TH^S_X$ by side-changing, that is, by tensoring with the sheaf $\omega_{\XS_\R/S}$ of real analytic forms of maximal degree $\dim_\RR X+\dim_\CC S$ (since we only deal with forms with anti-holomorphic degree with respect to $S$).

For a real analytic map $f:X\to Y$, the pushforward functor $\Df_!$ for right $\DXSR$-modules is defined by means of the transfer module $\shd_{(X\to Y)\times S_\R/S}$ and its bounded $\DXSR$-locally free resolution $\Sp_{(X\to Y)\times S_\R/S}$: we~have (with the usual structure of $(\DXSR,\DYSR)$ bi-module)
\[
\shd_{(X\to Y)\times S_\R/S}=\sha_{\XS_\R}\otimes_{f^{-1}\sha_{\YS_\R}}f^{-1}\DYSR
\]
and (by means of the sheaf $\Theta_{\XS_\R/S}$ of real analytic vector fields on $\XS_R$ which commute with holomorphic vector fields on $S$, and the associated Spencer complex of $\DXSR$)
\[
\Sp(\DXSR)\otimes_{\sha_{\XS_\R}}\shd_{(X\to Y)\times S_\R/S}\simeq\Sp(\DXSR)\otimes_{f^{-1}\sha_{\YS_\R}}f^{-1}\DYSR.
\]

\subsubsection*{A reminder on integration of currents}
Let $\Db_{\XS_\R}$ denote the sheaf of distributions on $\XS_\R$, that we regard as a left $\DXSR$-module. Furthermore, we consider the shifted real analytic de Rham complex $\Omega_{\XS_\R/S}^\cbbullet$ \emph{with term in degree zero being the sheaf $\omega_{\XS_\R/S}$}. We denote by $\Db^\vee_{\XS_\R}$ denote the right $\DXSR$-module $\omega_{\XS_\R/S}\otimes_{\sha_{\XS_\R}}\Db_{\XS_\R}$. The distributional de~Rham complex
\[
(\Omega_{\XS_\R/S}^\cbbullet\otimes_{\sha_{\XS_\R}}\Db_{\XS_\R})\otimes_{\sha_{\XS_\R}}\DXSR
\]
(the left structure of $\DXSR$ is used for defining the complex and the right structure of $\DXSR$ induces the right structure of the cohomologies) is a resolution of $\Db^\vee_{\XS_\R}$ as a right $\DXSR$-module.

For $f:X\to Y$ as above, we set
\[
\shc^\cbbullet_{\XS_\R}:=\Db^\vee_{\XS_\R}\otimes_{\DXSR}\Sp_{(X\to Y)\times S_\R/S},
\]
so that the pushforward of $\Db^\vee_{\XS_\R}$ can be expressed as
\begin{align*}
\Df_!\Db^\vee_{\XS_\R}&\simeq f_!(\Db^\vee_{\XS_\R}\otimes_{\DXSR}\Sp_{(X\to Y)\times S_\R/S})=f_!\shc^\cbbullet_{\XS_\R}\\
&\simeq f_!(\Db^\vee_{\XS_\R}\otimes_{\DXSR}\Sp(\DXSR)\otimes_{f^{-1}\sha_{\YS_\R}}f^{-1}\DYSR)\\
&\simeq f_!(\Db^\vee_{\XS_\R}\otimes_{\DXSR}\Sp(\DXSR))\otimes_{\sha_{\YS_\R}}\DYSR\\
&\simeq f_!(\Omega_{\XS_\R/S}^\cbbullet\otimes_{\sha_{\XS_\R}}\Db_{\XS_\R})\otimes_{\sha_{\YS_\R}}\DYSR,
\end{align*}
with a suitable definition of the differential of the complexes involved above.

On the other hand, the integration of currents is a morphism of complexes
\[
\tint_f:f_!(\Omega_{\XS_\R/S}^\cbbullet\otimes_{\sha_{\XS_\R}}\Db_{\XS_\R})\to (\Omega_{\YS_\R/S}^\cbbullet\otimes_{\sha_{\YS_\R}}\Db_{\YS_\R}).
\]
which satisfies, for a composition $f=h\circ g:X\to Z\to Y$:
\[
\tint_{f}=\tint_h\circ\tint_g.
\]

This integration morphism can be enhanced at the $\cD$-module level and we obtain:

\begin{lemma}\label{lem:compatibilityfg}
There exists a morphism of right $\DYSR$-modules
\[
\Dint_f:\Df_!\Db^\vee_{\XS_\R}\to\Db^\vee_{\YS_\R}
\]
which is compatible with composition of morphisms, in the sense that the following diagram commutes
\[
\xymatrix@C=.3cm{
\Df_!\Db^\vee_{\XS_\R}\ar@{}[r]|(.43){\simeq}\ar@/_2pc/[rrrrrrr]^{\Dint_{f}}&\Dh_!(\Dg_!\Db^\vee_{\XS_\R})\ar[rrr]^-{\Dh_!(\Dint_g)}&&&\Dh_!\Db^\vee_{\ZS_\R}\ar[rrr]^-{\Dint_h}&&&\Db^\vee_{\YS_\R}
}
\]
\end{lemma}

\begin{proof}[Proof of Theorem \ref{L:A3}]
We will first construct a natural transformation
\begin{equation}\label{eq:nattransfTHS}
\varphi_f(L(\cbbullet)):\Df_!\TH^{S,\vee}_X(f^{-1}L(\cbbullet))_{|{\bP}_Y^\mathrm{op}}\to\TH^{S,\vee}_Y(L(\cbbullet))_{|{\bP}_Y^\mathrm{op}},
\end{equation}
where $\bP_Y$ is the category given in Definition~\ref{DVTMF}, and we use $\bP_Y^\mathrm{op}$ because the functors $\TH^{S,\vee}$ are contravariant. Here, we regard both $\TH^{S,\vee}_Y(L(\cbbullet))$ and $\Df_!\TH^{S,\vee}_X(f^{-1}L(\cbbullet))$ as functors from $\bP_Y^\mathrm{op}\to\rD^\rb(\DYSR)$, although $\TH^{S,\vee}_Y(L(\cbbullet))$ takes values in $\Mod(\DYSR)$ (see below).

\subsubsection*{Step 1}
Given $Y'=Y\moins\Omega$ with $\Omega\in\shscr_Y$ and $V\in\shscr_S$, we~have (\cf\cite[(2)]{FMFS1})
\[
\TH^{S,\vee}_Y(\C_{Y'}\boxtimes\sho_V)=\Gamma_{Y'\times V}(\Db^\vee_{\YS_\R})\in\Mod(\DYSR)
\]
(we consider right $\DYSR$-modules here), and since $f^{-1}(\C_{Y'}\boxtimes\sho_V)\simeq \C_{f^{-1}Y'}\boxtimes\sho_V$, we have
\[
\TH^S_X(f^{-1}(\C_{Y'}\boxtimes\sho_V))\simeq\Gamma_{f^{-1}(Y')\times V}(\Db^\vee_{\XS_\R}).
\]

On noting that $f_!\Gamma_{f^{-1}(Y')\times V}=\Gamma_{Y'\times V}f_!$, we obtain a morphism
\[
\xymatrix@C=.3cm{
\Df_!\bigl[\Gamma_{f^{-1}(Y')\times V}(\Db^\vee_{\XS_\R})\bigr]\ar@{}[r]|(.54){\simeq}
& \Gamma_{Y'\times V}\bigl[\Df_!(\Db^\vee_{\XS_\R})\bigr]\ar[rrrr]^-{\Gamma_{Y'\times V}(\Dint_f)}&&&&\Gamma_{Y'\times V}\Db^\vee_{\YS_\R}\\
\Df_!\TH^{S,\vee}_X(f^{-1}(\C_{Y'}\boxtimes\sho_V))\ar[rrrrr]^-{\varphi_f(\C_{Y'}\boxtimes\sho_V)}\ar[u]^\wr&&&&&\TH^{S,\vee}_Y(\C_{Y'}\boxtimes\sho_V)\ar[u]_\wr
}
\]
From Lemma \ref{lem:compatibilityfg} we obtain:
\begin{equation}\label{eq:compatibilityfg}
\varphi_f(\C_{Y'}\boxtimes\sho_V)=\varphi_h(\C_{Y'}\boxtimes\sho_V)\circ\Dh_!(\varphi_g(\C_{h^{-1}(Y')}\boxtimes\sho_V)).
\end{equation}

\subsubsection*{Step 2}
For $\Omega=Y\moins Y'$ as above, let us apply the triangulated functor $\TH_Y^{S,\vee}$ to the exact sequence
\[
0\to\C_\Omega\boxtimes\sho_V\to\C_Y\boxtimes\sho_V\to\C_{Y'}\boxtimes\sho_V\to0.
\]
We still obtain an exact sequence, due to the isomorphism of distinguished triangles
\[
\xymatrix@C=.4cm{
\TH^{S,\vee}_Y(\C_{Y'}\boxtimes\sho_V)\ar[r]\ar[d]^\wr&\TH^{S,\vee}_Y(\C_Y\boxtimes\sho_V)\ar[r]\ar[d]^\wr&\TH^{S,\vee}_Y(\C_\Omega\boxtimes\sho_V)\dpl\To{+1}\ar[d]^\wr&\\
0\dpl\to\Gamma_{Y'\times V}\Db^\vee_{\YS_\R}\ar[r]&\Gamma_{Y\times V}\Db^\vee_{\YS_\R}\ar[r]&\Gamma_{Y\times V}T\shh om(\C_{\Omega\times V}, \Db^\vee_{\YS_\R})\dpl\to0
}
\]

Let us apply the functor $\Df_!\TH^{S,\vee}_X(f^{-1}(\cbbullet))$ to the same exact sequence. We realize the corresponding distinguished triangle in $\rD^\rb(\DYSR)$ as an exact sequence of complexes in $\rC^b(\DYSR)$ (setting $\shc^\cbbullet_{f^{-1}(Y')\times V}=\Gamma_{f^{-1}(Y')\times V}\shc^\cbbullet_{\XS_R}$), and there is a unique morphism $\varphi_f(\C_\Omega\boxtimes\sho_V)$ in $\rC^b(\DYSR)$ making the following diagram commutative:
\[
\xymatrix@R=1.2cm{
0\ar[r]&f_!\shc^\cbbullet_{f^{-1}(Y')\times V}\ar[r]\ar[d]_{\varphi_f(\C_{Y'}\boxtimes\sho_V)}&f_!\shc^\cbbullet_{X\times V}\ar[r]\ar[d]_{\varphi_f(\C_Y\boxtimes\sho_V)}&f_!\shc^\cbbullet_{f^{-1}(\Omega)\times V}\ar[r]\ar@{.>}[d]_{\varphi_f(\C_\Omega\boxtimes\sho_V)}\ar[r]&0\\
0\ar[r]&\TH^{S,\vee}_Y(\C_{Y'}\boxtimes\sho_V)\ar[r]&\TH^{S,\vee}_Y(\C_Y\boxtimes\sho_V)\ar[r]&\TH^{S,\vee}_Y(\C_\Omega\boxtimes\sho_V)\ar[r]&0
}
\]
We regard $\varphi_f(\C_\Omega\boxtimes\sho_V)$ as a morphism in $\rD^\rb(\DYSR)$, and given a composition $f=h\circ g$, the relation \eqref{eq:compatibilityfg} implies a similar relation.

\subsubsection*{Step 3}
In order to obtain \eqref{eq:nattransfTHS}, we need to check that it is compatible with morphisms in $\shs$. Due to the description of morphisms (Definition \ref{DVTMF}), we are left with considering
$\Omega\subset\Omega'\subset Y$ and $V\subset V'\subset S$ and a section
$\alpha\in\Gamma(\Omega\times \ov{V};\pYOS)$ defining a morphism
$\Omega\times V\to \Omega'\times V'$ in $\bP_Y$.
Then the commutativity $(\alpha\cdot)\circ\varphi_f(\C_{\Omega'}\boxtimes\sho_{V'})=\varphi_f(\C_\Omega\boxtimes\sho_V)\circ(\alpha\cdot)$ is readily checked.

In such a way we have constructed a natural transformation \eqref{eq:nattransfTHS}. The relation \eqref{eq:compatibilityfg} holds for $\varphi_f$.

\subsubsection*{Step 4}
In order to end the proof of the theorem, we show how to extend $\varphi_f$ as a natural transformation between the extended functors
$\TH^{S,\vee}_Y(L(\cbbullet)),\Df_!\TH^{S,\vee}_X(f^{-1}L(\cbbullet)): \rD^+(\bP_Y^\mathrm{op})=\rD^-(\bP_Y)^\mathrm{op}\to\rD^+(\DYSR)$, which will suffice, according to Proposition~\ref{PL}.

Firstly, since $f$ has finite cohomological dimension, the above functors extend as functors from $\rC^+(\bP_Y^\mathrm{op})$ to $\rD^+(\DYSR)$, and by Step 3, $\varphi_f$ extends as a natural transformation between these functors. Then the latter functors define functors $\rD^+(\bP_Y^\mathrm{op})\to\rD^+(\DYSR)$, and $\varphi_f$ extends similarly as a natural transformation between them.

Lastly, \eqref{eq:compatibilityfg} also extends as
\begin{equation}\label{eq:compatibilityfgD}
\varphi_f(\cbbullet)=\varphi_h(\cbbullet)\circ\Dh_!(\varphi_g(h^{-1}(\cbbullet))).\qedhere
\end{equation}
\end{proof}

\subsection*{Complement to the proof of Lemma \ref{Ldirim}}
We keep the setting of Theorem \ref{L:A3}. For $f:X\to Y$, we denote the counit of the adjunction $(f^{-1},Rf_*)$ on $\rD^\rb(\pXOS)$ by
\[
\delta_f(\cbbullet):f^{-1}Rf_*(\cbbullet)\to(\cbbullet).
\]
For a composition $f=h\circ g$, the following relation is standard:
\begin{equation}\label{eq:detlacompatiblefg}
\delta_f(\cbbullet)=\delta_g(\cbbullet)\circ g^{-1}(\delta_h(Rg_*\cbbullet)).
\end{equation}
As we need to work both with $Rf_*$ (because of the above adjunction) and $Rf_!$ (because of Theorem \ref{L:A3}), we will restrict $Rf_*$ as acting on the sub-category $\rD^\rb_{\rc,f!}(\pXOS)$ consisting of $S$-$\R$-constructible complexes with $f$-proper support, so that $Rf_*$ coincides with $Rf_!$ on this sub-category and has essential image in $\rD^\rb_\rc(\pYOS)$. From now on, we only consider $\delta_f(\cbbullet)$ for $\cbbullet$ being an object or morphism in $\rD^\rb_{\rc,f!}(\pXOS)$. Note however that $f^{-1}Rf_*(\cbbullet)$ is an object or a morphism in $\rD^\rb_{\rc}(\pXOS)$ and possibly not in $\rD^\rb_{\rc,f!}(\pXOS)$, but \eqref{eq:detlacompatiblefg} remains meaningful (and valid) since $Rg_*$ is a functor from $\rD^\rb_{\rc,f!}(\pXOS)$ to $\rD^\rb_{\rc,h!}(\pZOS)$. We will thus replace below the $*$-pushforwards with the $!$-pushforwards.

Applying the contravariant functor $\TH^S_X$ to this restricted $\delta_f$ leads to a natural transformation between functors $\rD^\rb_{\rc,f!}(\pXOS)\to\rD^+(\DXSR)$:
\[
\epsilon_f(\cbbullet)=\TH^S_X(\delta_f(\cbbullet)):\TH^S_X(\cbbullet)\to\TH^S_X(f^{-1}Rf_!(\cbbullet)),
\]
which satisfies, for a composition $f=h\circ g$, the following relation:
\[
\epsilon_f(\cbbullet)=\TH^S_X\bigl(g^{-1}(\delta_h(Rg_!\cbbullet))\bigr)\circ\epsilon_g(\cbbullet).
\]

The natural transformation $\mu_f(\cbbullet)$ from the functor $\Df_!\TH_X^S(\cbbullet):\rD^\rb_{\rc,f!}(\pXOS)\to
\rD^+(\DYSR)$ to the functor
$\TH_Y^S(Rf_!\cbbullet):\rD^\rb_{\rc}(\pXOS)\to\rD^+(\DYSR)$, as defined in Lemma \ref{Ldirim}, now reads, by~means of $\varphi_f$ of Theorem \ref{L:A3},
\[
\mu_f(\cbbullet)=\varphi_f(Rf_!(\cbbullet))\circ\Df_!(\epsilon_f(\cbbullet)).
\]

We aim at proving that the following diagram of natural transformations commutes:
\[
\xymatrix@C=1.7cm{
\Dh_!\Dg_!\TH_X^S(\cbbullet)\ar@{=}[d]\ar[r]^-{\Dh_!(\mu_g(\cbbullet))}&\Dh_!\TH^S_Z(Rg_!(\cbbullet))\ar[r]^-{\mu_h(Rg_!(\cbbullet))}&\TH^S_Y(Rh_!Rg_!(\cbbullet))\ar@{=}[d]\\
\Df_!\TH_X^S(\cbbullet)\ar[rr]^-{\mu_f(\cbbullet)}&&\TH^S_Y(Rf_!(\cbbullet))
}
\]
where the vertical equalities are defined from the standard isomorphisms.

Commutativity amounts to the equality, as natural transformations,
\[
\varphi_f(Rf_!(\cbbullet))\circ\Df_!(\epsilon_f(\cbbullet))=\varphi_h(Rf_!(\cbbullet))\circ\Dh_!(\epsilon_h(Rg_!(\cbbullet)))\circ\Dh_!\mu_g(\cbbullet),
\]
which is implied, according to \eqref{eq:compatibilityfgD} and the isomorphism $\Df_!=\Dh_!\Dg_!$, by
\begin{align*}
\varphi_g(h^{-1}(Rf_!(\cbbullet)))\circ\Dg_!(\epsilon_f(\cbbullet)) &=\epsilon_h(Rg_!(\cbbullet))\circ\mu_g(\cbbullet)\\
&=\epsilon_h(Rg_!(\cbbullet))\circ\varphi_g(Rg_!(\cbbullet))\circ\Dg_!(\epsilon_g(\cbbullet)),
\end{align*}
and itself is implied by the equality
\[
\varphi_g(h^{-1}(Rf_!(\cbbullet)))\circ\Dg_!\TH^S_X\bigl(g^{-1}(\delta_h(Rg_!\cbbullet))\bigr)=\epsilon_h(Rg_!(\cbbullet))\circ\varphi_g(Rg_!(\cbbullet)).
\]
It is thus enough to apply the equality
\[
\varphi_g(h^{-1}Rh_!(\cbbullet))\circ\Dg_!\TH^S_X\bigl(g^{-1}(\delta_h(\cbbullet))\bigr)=\TH^S_Z(\delta_h(\cbbullet))\circ\varphi_g(\cbbullet)
\]
to $Rg_
!(\cbbullet)$. This is an equality of natural transformations
\[
\Dg_!\TH_X^S(g^{-1}(\cbbullet))\to\TH_Z^S(h^{-1}Rh_!(\cbbullet)),
\]
both source and target regarded as functors $\rD^\rb_{\rc,h!}(\pZOS)\to\rD^+(\DZSR)$. We~will prove the latter equality, which amounts to the commutativity of the diagram of natural transformations
\[
\xymatrix@C=3.5cm{
\Dg_!\TH_X^S(g^{-1}(\cbbullet))\ar[r]^-{\Dg_!\TH_X^S(g^{-1}(\delta_h(\cbbullet)))}\ar[d]_{\varphi_g(\cbbullet)}&\Dg_!\TH_X^S(g^{-1}h^{-1}Rh_!(\cbbullet))\ar[d]^{\varphi_g(h^{-1}Rh_!(\cbbullet))}\\
\TH_Z^S(\cbbullet)\ar[r]^-{\TH_Z^S(\delta_h(\cbbullet))}&\TH_Z^S(h^{-1}Rh_!(\cbbullet))
}
\]
The commutativity follows from the property that $\varphi_g$ is a natural transformation $\Dg_!\TH_X^S\circ g^{-1}\to\TH_Z^S$ and the lower horizontal natural transformation is the image by $\varphi_g$ of the upper one.\qed

\providecommand{\sortnoop}[1]{}\providecommand{\eprint}[1]{\href{http://arxiv.org/abs/#1}{\texttt{arXiv\string:\allowbreak#1}}}\providecommand{\hal}[1]{\href{https://hal.archives-ouvertes.fr/hal-#1}{\texttt{hal-#1}}}\providecommand{\doi}[1]{\href{http://dx.doi.org/#1}{\texttt{doi\string:\allowbreak#1}}}
\providecommand{\bysame}{\leavevmode\hbox to3em{\hrulefill}\thinspace}
\providecommand{\MR}{\relax\ifhmode\unskip\space\fi MR }
\providecommand{\MRhref}[2]{%
  \href{http://www.ams.org/mathscinet-getitem?mr=#1}{#2}
}
\providecommand{\href}[2]{#2}

\addtocontents{toc}{\protect\egroup}

\end{document}